\documentclass[reqno, 12pt]{amsart}

\usepackage{amssymb,amsmath,amsfonts}
\usepackage[margin=0.67in]{geometry}
\usepackage{dsfont}
\usepackage{color}
\usepackage{graphicx}
\usepackage{latexsym}
\usepackage{amsmath}
\usepackage{amssymb}
\usepackage{graphics}
\usepackage[dvips]{epsfig}
\usepackage{mathrsfs}
\usepackage{mathtools}
\usepackage{leqno}
\usepackage{stmaryrd,cite}
\usepackage{cases}
\usepackage{enumerate}
\usepackage[usenames,dvipsnames,svgnames]{xcolor}
\definecolor{refkey}{rgb}{1,0,0.5}
\definecolor{labelkey}{rgb}{0,0.4,1}
\usepackage{todonotes}
\usepackage{marginnote}

\presetkeys{todonotes}{fancyline, color=green!40}{}

\makeatletter
\renewcommand{\@todonotes@drawMarginNoteWithLine}{%
	\begin{tikzpicture}[remember picture, overlay, baseline=-0.75ex]%
	\node [coordinate] (inText) {};%
	\end{tikzpicture}%
	\marginnote[{
		\@todonotes@drawMarginNote%
		\@todonotes@drawLineToLeftMargin%
	}]{
		\@todonotes@drawMarginNote%
		\@todonotes@drawLineToRightMargin%
	}%
}
\makeatother
\usepackage{listings}
\usepackage{ifpdf}

\ifpdf
\usepackage[CJKbookmarks=true,
         hyperindex=true,
         pdfstartview=FitH,
         bookmarksnumbered=true,
         bookmarksopen=true,
         colorlinks=true,
         citecolor=blue,
         linkcolor=blue,
         urlcolor=blue,
         pdfborder=001,
         pdfauthor={},
         pdftitle={},
         pdfkeywords={},
         ]{hyperref}

\allowdisplaybreaks

\numberwithin{equation}{section}

\newtheorem{thm}{Theorem}[section]

\newtheorem{lem}[thm]{Lemma}
\newtheorem{prop}[thm]{Proposition}
\newtheorem{rmk}[thm]{Remark}
\newtheorem{defn}[thm]{Definition}

\newcommand{\be}{\begin{equation}}
\newcommand{\ee}{\end{equation}}

\newcommand{\bee}{\begin{equation*}}
\newcommand{\eee}{\end{equation*}}

\newcommand{\bse}{\begin{subequations}}
\newcommand{\ese}{\end{subequations}}

\newcommand{\bs}{\begin{split}}
\newcommand{\es}{\end{split}}

\begin{document}
\author{Hairong Liu$^{1}$}\thanks{$^{1}$School of Mathematics and Statistics, Nanjing University of Science and Technology, Nanjing, 210094,  China,
E-mail: hrliu@njust.edu.cn}

\author{Hua Zhong$^{2}$}\thanks{$^{2}$School of Mathematics, Southwest Jiaotong University, Chengdu, 611756, China,
E-mail: huazhong@swjtu.edu.cn}

\lstset{basicstyle=\ttfamily\large}

\title[compressible  Navier-Stokes equations with density-dependent
 viscosities] {Well-posedness of the 3-D compressible  Navier-Stokes equations with density-dependent
 viscosities in exterior domains with far-field vacuum}

\begin{abstract}
This paper investigates the local existence and uniqueness of  strong solutions to the three-dimensional compressible Navier-Stokes equations with density-dependent viscosities in  exterior domains.
When both the shear and bulk viscosity coefficients  depend on the density in a power law ($\rho^{\delta}$ with $0<\delta<1$) and Navier-slip boundary condition on the velocity is imposed,
 base on  a reformulation of the problem using new variables to handle the degeneracy near vacuum,  we establish the local well-posedness of regular  solutions  with far-field vacuum  in inhomogeneous Sobolev spaces.
 Compared to the Cauchy problem,  the initial-boundary value problem requires establishing non-standard weighted estimates and handling the unavailability of boundary conditions for higher-order terms. Our approach addresses these challenges via the conormal space technique.


\end{abstract}

\maketitle
\section{Introduction and Main Theorems}
We consider the three-dimensional compressible Navier-Stokes equations:
\begin{equation}\label{prob1}
	\left\{
	\begin{array}{ll}
		\rho_{t}+\mbox{div}(\rho u)=0,\\
		\rho\big(u_{t}+u\cdot\nabla u\big)+\nabla P=\mbox{div}\mathbb{T},
		\end{array}
	\right.
\end{equation}
where $\rho$, $u=(u_1,u_2,u_3)$ denote the density and  the velocity,  respectively.
For polytropic gases, the constitutive relation is given by
\begin{equation}\label{P}
	P=A\rho^{\gamma},\quad A>0,\quad \gamma>1,
\end{equation}
where $A$ is an entropy constant and $\gamma$ is the adiabatic exponent.
$\mathbb{T}$ denotes the viscous stress tensor with the form
\begin{equation}\label{T}
	\mathbb{T}=\mu(\rho)\left(\nabla u+(\nabla u)^T\right)+\lambda(\rho)\mbox{div}u\mathbb{I}_{3}
\end{equation}
where $\mathbb{I}_3$ is the $3\times3$ identity matrix.  The viscosity coefficients $\mu$, $\lambda$ are the functions of $\rho$ and could be expressed as
\begin{equation}\label{T}
	\mu(\rho)=\alpha \rho^{\delta},\quad \lambda(\rho)=\beta\rho^{\delta},
\end{equation}
for some constant $\delta\geq0$, $\alpha$ and $\beta$ are both constants satisfying
\begin{equation}\label{alpha}
	\alpha>0,\quad 2\alpha+3\beta\geq0.
\end{equation}
In this work, we focus on the case $0<\delta<1$ in \eqref{T}.

In the theory of gas dynamics, the compressible Navier-Stokes equations  can be  derived from the Boltzmann equations via the Chapman-Enskog expansion. Under appropriate physical assumptions, the viscosity coefficients and the heat conductivity coefficient are not constants but functions of the absolute temperature $\theta$ such as:
\begin{align*}
\mu(\theta) = a_1 \theta^{\frac{1}{2}} F(\theta), \quad
\lambda(\theta) = a_2 \theta^{\frac{1}{2}} F(\theta), \quad
\kappa(\theta) = a_3 \theta^{\frac{1}{2}} F(\theta)
\end{align*}
for some positive constants $a_i$ ($i=1,2,3$) (see \cite{ChapmanCowling},\cite{LiQin}). For the cut-off inverse power force models, where the intermolecular potential varies as $r^{-a}$ with $r$ being the intermolecular distance, the function $F(\theta)$ takes a power-law form
$F(\theta) = \theta^{b}$ with $b = \frac{2}{a} \in [0, +\infty)$.
According to Liu-Xin-Yang\cite{LiuXinYang}, for isentropic and polytropic fluids, this temperature dependence translates into a density dependence via the equation of state. Using the laws of Boyle and Gay-Lussac:
$P = R\rho\theta = A\rho^{\gamma}$ for constants  $R, A>0$,
one finds $\theta = A R^{-1} \rho^{\gamma-1}$. Substituting this into the expressions for the viscosities shows that they become functions of the density of the form
\begin{equation*}
\mu(\rho) \sim \rho^{\delta},
\end{equation*}
where the exponent $\delta$ satisfies $0 < \delta < 1$ in many physical scenarios.\\

\noindent{ \bf Problem Setting}:

Let $U$ be a simply connected bounded smooth domain in $\mathbb{R}^{3}$, and $\Omega\equiv\mathbb{R}^{3}\backslash \bar{U}$ be the exterior domain.  We study the initial-boundary value problem of (\ref{prob1}) in $\Omega$ with:
\begin{itemize}
 \item Initial conditions:
\begin{equation}\label{initial}
	(\rho, u)|_{t=0}=(\rho_0(x), u_0(x))\quad \mbox{with}\  \rho_0(x)>0,\quad  \inf \rho_0(x)=0. 
\end{equation}
\item Navier-slip boundary conditions on $u$:
\begin{equation}\label{boundary2-1}
	u\cdot n=0,\quad \left(S(u)\cdot n\right)_{\tau}=-(\vartheta u)_{\tau},\quad \mbox{on}\quad \partial\Omega.
\end{equation}
\end{itemize}
 Here $\vartheta$ is a coefficient which measures the tendency of the fluid to slip on the boundary,
  $n$ stands for the outward unit normal to $\partial\Omega$, $S$ is the strain tensor defined as
\begin{equation*}
	S(u)=\frac{1}{2}\left(\nabla u+(\nabla u)^{T}\right).
\end{equation*}
For a vector field $v$ on $\partial\Omega$, $v_{\tau}$ means its tangential part, that is
$v_{\tau}=v-(v\cdot n)n$.
These conditions,  first introduced by Navier  \cite{Navierslip},  model the situation that there is a stagnant layer of fluid close to the wall allowing a fluid to slip and the slip velocity is proportional to the shear stress.
The boundary condition (\ref{boundary2-1}) is equivalent to the following conditions in the sense of  the distribution
\begin{equation}\label{boundary2-2}
	u\cdot n=0,\quad \mbox{curl}u\times n= -2\left(\vartheta\mathbb{I}_3-S(n) \right)u=:-K(x)u,\quad \mbox{on}\quad \partial\Omega,
\end{equation}
 see \cite{XX2013}.
 In the special case that the boundary $\partial\Omega$ is flat, one has $S(n)=0$.
For the sake of simplicity, we assume that $K(x)$ is a positive semi-definite matrix.
However, this requirement can be replaced by the condition that $K$ is sufficiently small (e.g., see  (\ref{K-2}) in  Lemma \ref{lem-u} below).
\begin{itemize}
\item Far field behavior:
\begin{equation}\label{1.8}
	(\rho,u)(x,t)\rightarrow(0,0),\quad\text{as}\quad |x|\rightarrow\infty,\quad \mbox{for}\ t\geq0.
\end{equation}
\end{itemize}
Moreover, we assume that the initial data satisfy  the compatibility condition  with the boundary conditions.

Throughout this paper, we take the following simplified notations for the standard inhomogeneous and homogeneous Sobolev space:
\begin{align*}
	& \|f\|_{L^{p}}=\|f\|_{L^{p}(\Omega)};\quad 	\|f\|_s=\|f\|_{H^s(\Omega)};\quad D^1=\{f\in L^6(\Omega):\,\,\|f\|_{D^1}=\|\nabla f\|_{L^2}<+\infty\};\\
	& D^{k,r}=\{f\in L_{loc}^1(\Omega):\,\,\|f\|_{D^{k,r}}=\|\nabla^kf\|_{L^r}<+\infty\};\ D^k=D^{k,2} \, \,\text{for}\,\,  k\geq 2;\
\int_{\Omega}fdx=\int_{\Omega}f.
\end{align*}
One can get a detailed and precise study of  homogeneous Sobolev space in \cite{Galdi1994book}.
\vspace{1mm}

\noindent {\bf Background}:

For the constant viscous fluid (i.e. $\delta=0$ in (\ref{T})), there is a lot of literature on the well-posedness of classical solutions to isentropic compressible Navier-Stokes equations.
When $\inf_{x}\rho_0>0$, it is well-known that the local existence of classical solutions for (\ref{prob1})-(\ref{1.8}) can be obtained by a standard Banach fixed point argument, see  Nash \cite{nash}.  And this method has been extended to be a global one by Matsumura-Nishida \cite{Matsumura1980} for initial data close to a non-vacuum equilibrium in some Sobolev space $H^s$ $(s > \frac{5}{2})$. However, when the density function connects to vacuum continuously, this approach is invalid because of $\inf_{x}\rho_0=0$, which occurs when some physical requirements are imposed, such as finite total initial mass, finite total initial energy, or vacuum appearing locally in some open sets.   When $\inf\rho_0(x)=0$, the local-in-time
well-posedness of strong solutions with vacuum was solved by Cho-Choe-Kim  \cite{Cho2004} and Cho-Kim  \cite{Cho2006}
in $\mathbb{R}^3$, where they introduced an initial compatibility condition to compensate the lack of a positive lower
bound of density.
Later, Huang-Li-Xin \cite{HLX2012} extended the local existence to a global one
with smooth initial data that are of small energy but possibly
large oscillations in $\mathbb{R}^3$. Recently,  Cai-Li\cite{CL2021} established  the global existence in 3D bounded domains, and in exterior domain with Navier-slip boundary by  Cai-Li-Lv\cite{CLL2021}. Jiu-Li-Ye \cite{JLY2014} proved the global existence of classical solution with arbitrarily
large data and vacuum in $\mathbb{R}$.

When viscosity coefficients are density-dependent (i.e. $\delta>0$ in (\ref{T})), the Navier-Stokes system has been received extensive
attentions in recent years, especially for the case with vacuum, where the well-posedness of solutions become
more challenging due to the degenerate viscosity.
For the physically significant case where the coefficients $\mu$ and $\lambda$ are not constants, numerous studies have examined the one-dimensional or spherically symmetric isentropic Navier-Stokes equations, please see  \cite{JXZ2005}, \cite{YYZ2001}, \cite{YZ2002-1}, \cite{YZ2002}, \cite{ZF2006} and the references therein.
For the multi-dimensional case,
a remarkable discovery of a new mathematical entropy function has been found by
Bresch-Desjardins \cite{BD2003} for the viscosity satisfying some mathematical relation, which provides
additional regularity on some derivative of the density. This observation was applied widely
in proving the global existence of weak solutions with vacuum for Navier-Stokes equations
and some related models (see \cite{BVY2022}, \cite{JX2008}, \cite{LX2015}, \cite{VY2016} and so on).
For strong solutions,  when $\delta=1$, Li-Pan-Zhu \cite{LPZ20172D} obtained the local existence of 2-D classical solution with far field vacuum, which also applies to the 2-D shallow water equations.
When $1<\delta\leq\min{\{3,\frac{\gamma+1}{2}\}}$, Li-Pan-Zhu\cite{LPZ2019} established the
local existence of 3-D classical solutions with arbitrarily large data and vacuum.  Recently, Xin-Zhu\cite{XZ2021advance} have proved the global well-posedness of regular solutions with vacuum for a class of smooth initial data that
are of small density but possibly large velocities. This result is the first one on the global existence of smooth solution which have large velocities and contain vacuum state for such degenerate system in three space dimensions.
 When $0<\delta<1$,
Xin-Zhu \cite{XZ2021jmpa} identify a class of initial data admitting a local regular solution with far field vacuum and finite energy in some inhomogeneous Sobolev spaces by introducing some new variables and initial compatibility conditions for the Cauchy problem in 3-D space.  Cao-Li-Zhu \cite{CLZ2022} proved the global existence
of 1-D classical solution with large initial data. Some other interesting results and discussions can also be
found in \cite{CLZ2024}, \cite{DXZ2023},  \cite{GLZ2019}, \cite{GL2016}, \cite{LXZ2016} and the references therein.
Recently, the authors of the present paper and their collaborators have investigated  the local existence and uniqueness of  strong solutions to
the initial-boundary value problem of 3-D compressible magnetohydrodynamic equations  with  viscosities as in (\ref{T}) with $\delta=1$ (i.e. $\mu(\rho)=\alpha \rho$ and $ \lambda(\rho)=\beta\rho$) in exterior domains with far-field vacuum \cite{LLZ2025}.
\vspace{2mm}

\noindent{\bf Key Difficulties and Strategies}:

In this paper, we focus on the local strong solution for 3-D Navier-Stokes equations \eqref{prob1} with density-dependent viscosities ( $0<\delta<1$ in \eqref{T}) in  exterior domains  with Navier-slip boundary conditions. 	 
The analysis of the degeneracies in momentum equations and the boundary conditions require some special attentions. The major difficulties include:

(i) Singular elliptic equations with singular source terms. For the case $0<\delta<1$, if  $\rho>0$,  the momentum equation $\eqref{prob1}_2$ can be formally rewritten as
\begin{align*}
u_{t}+u\cdot\nabla u+\frac{A\gamma}{\gamma-1}\nabla \rho^{\gamma-1}+\rho^{\delta-1}Lu=\frac{\delta}{\delta-1}\nabla \rho^{\delta-1}\cdot Q(u),
\end{align*}
where
\begin{align*}
Lu=-\alpha \Delta u-(\alpha+\beta)\nabla\mbox{div}u,\quad Q(u)=\alpha\left(\nabla u+(\nabla u)^{T}\right)+\beta \mbox{div}u\mathbb{I}_3.
\end{align*}
Note that the coefficient $\rho^{\delta-1}$ in front of the Lam\'{e} operator $L$ tends to $\infty$ as $\rho\rightarrow 0$ in the far field field instead of equaling to $1$ (the case $\delta=1$) in \cite{LPZ20172D}, \cite{LLZ2025} or tending to $0$  (the case $\delta>1$) in \cite{XZ2021advance}.
Therefore, the well-definedness of the term $\rho^{\delta-1}Lu$ must be established.
Moreover, the source term involves the singular gradient $\nabla \rho^{\delta-1}$. As a result, the quantities
\begin{align*}
\left(\nabla\rho^{\delta-1}, \rho^{\delta-1}Lu\right)
\end{align*}
will play a crucial role in our analysis of higher-order regularity for the fluid velocity.

(ii) For the initial-boundary value problem, particularly in establishing a priori estimates for $u$, we encounter some difficulties compared to the Cauchy problem. First, we apply the elliptic estimates to the equation
 \begin{equation*}
L(\rho^{\delta-1}u)=u_t+\mbox{other terms }
\end{equation*}
with the Navier-slip boundary conditions (\ref{boundary2-2}) to obtain an estimate for $\|\rho^{\delta-1}\nabla^2u\|_{L^2}$. This yields
\begin{equation*}
\|\rho^{\delta-1}u\|_{D^2}^2\leq C \|u_t\|^2_{L^2}+\|\mbox{other terms}\|^2_{L^2}+\|\nabla(\rho^{\delta-1}u)\|^2_{L^2}.
\end{equation*}
Consequently, it is necessary to estimate the weighted first-order derivative term $\|\rho^{\delta-1}\nabla u\|_{L^2}$.

Second, the estimation of $\int_0^t\|\rho^{\delta-1} \nabla^2 u\|^2_{D^2}ds$ is particularly challenging. Directly applying elliptic estimates to $L(\rho^{\delta-1} u)$ fails to provide control over $\|\rho^{\delta-1} \nabla^4u\|_{L^2}$ due to the absence of estimates for $\nabla^4\rho^{\delta-1}$, which seems
impossible in the current $H^3$ framework. Alternatively, the method of applying elliptic estimates to $L(\rho^{\delta-1} \nabla^2u)$, used in \cite{XZ2021jmpa} for the Cauchy problem, is also infeasible here because $D^2 u$ lacks boundary conditions. In order to overcome these difficulties, rather than estimating quantity $\|\xi \nabla^2 u\|^2_{D^2}$ directly, we instead develop estimates for quantity $\|\xi \nabla u\|^2_{D^3}$ based on the conormal space technique. Specifically, we estimate $\nabla u$ by  decomposing  into three components to be handled separately: the tangential derivative,
the normal component of the normal derivative,
and the tangential component of the normal derivative (detailed in Lemma \ref{lem-u-2}).\\

\noindent{\bf Main Results:}
\vspace{1mm}

Based on the above observations, and motivated by \cite{LPZ20172D}, the definition for regular solution to initial-boundary value problem is given as following.

\begin{defn}\label{def1}
(Regular solutions to initial-boundary value problem) Let $T>0$ be a finite constant. $(\rho,u)$ is called a regular solution to initial-boundary value problem (\ref{prob1})-(\ref{1.8}) in $\Omega\times[0,T]$ if $(\rho,u)$ solves this problem in the sense of distribution and
\begin{align*}
	(i)&\  \rho> 0, \quad \rho^{\gamma-1}\in C([0,T];H^3), \quad \nabla \rho^{\delta-1}\in L^{\infty}([0,T];D^1 \cap D^2);\\
     (ii)&  \ u\in C([0,T];H^3)\cap L^2([0,T];D^4),\quad u_t\in C([0,T];H^1)\cap L^2([0,T];D^2),\\
	& \rho^{\delta-1}\nabla u\in L^{\infty}([0,T];H^2)\cap L^2([0,T]; D^1\cap D^3),\quad \rho^{\frac{\delta-1}{2}}\nabla u_t\in L^\infty([0,T];L^2),\\
	&\rho^{\delta-1}\nabla^2 u\in C([0,T];H^1)\cap L^2([0,T];D^2).
	\end{align*}
\end{defn}

Now we are ready to state our main results.
\begin{thm}\label{thm1}
Let parameters $(\gamma,\delta,\alpha,\beta)$ satisfy
\begin{align}\label{constant}
	\gamma>1,\quad 0<\delta<1,\quad \alpha>0,\quad 2\alpha+3\beta\geq0.
	\end{align}
If the initial data $(\rho_0,u_0)$ satisfy the conditions:
\begin{align}\label{initial data}
	\rho_0>0,\quad(\rho_0^{\gamma-1},u_0)\in H^3,\quad \nabla\rho_0^{\delta-1}\in D^1\cap D^2,\quad \nabla \rho_0^{\frac{\delta-1}{2}}\in L^4,
	\end{align}
and the initial compatibility conditions:
\begin{align}\label{comp}
\nabla u_0&=\rho_0^{1-\delta}g_1,\quad Lu_0=\rho_0^{1-\delta}g_2,\quad
\nabla(\rho_0^{\delta-1}Lu_0)=\rho_{0}^{1-\delta}g_3,
\end{align}
for some $(g_1,g_2,g_3)\in L^2$.
Moreover, we assume that the initial data satisfy  the compatibility condition  with the boundary conditions.
Then there exists a time
$T_{*}>0$  and a unique local regular solution $(\rho,u)$ to the initial-boundary value problem (\ref{prob1})-(\ref{1.8}) satisfying:
\begin{align}\label{thm1-reg}
&\rho^{\frac{\delta-1}{2}}u_{tt}\in  L^2([0, T_{*}];L^2),\quad \rho^{\delta-1}\nabla^2u_t\in L^2([0, T_{*}];L^2),\nonumber\\
&\rho^{1-\delta}\in L^{\infty}([0,T_{*}];D^{1,6}\cap D^{2,3}\cap D^{3}),\nonumber\\
&\nabla \rho^{\delta-1}\in C([0, T_{*}];D^1\cap D^2), \quad \rho^{-1}\nabla\rho\in L^{\infty}([0, T_{*}];L^{\infty}\cap L^6\cap D^{1,3}\cap D^{2}).
\end{align}
Moreover, if $1<\gamma\leq 2$, then $(\rho,u)$ is a strong solution to (\ref{prob1})-(\ref{1.8}) in $[0,T_{*}]\times \Omega$.
\end{thm}

\begin{rmk}
(i)\ The assumption $\nabla \rho_0^{\frac{\delta-1}{2}}\in L^4$ will only be used in
the  limit from the non-vacuum flows to the flow with far-field vacuum in $\S$3.5.\\
(ii)\  It should be noted  that the compatibility conditions (\ref{comp}) employed here differ from those for the Cauchy problem in \cite{XZ2021jmpa}. This adjustment is required by the need to estimate the weighted first-order derivative term $\|\rho^{\delta-1}\nabla u\|^2_{L^2}$ in order to close the a priori estimates for the initial-boundary value problem.
\end{rmk}

\begin{rmk}
We remark that the initial data \eqref{initial data} contains a large class of functions, for example,
\begin{align*}
	\rho_{0}(x)=\frac{1}{1+|x|^{2\iota}},\quad u_0(x)\in C_0^3(\Omega),
\end{align*}
where $\frac{3}{4(\gamma-1)}< \iota<\frac{1}{4(1-\delta)}$.  Furthermore, if $\rho_{0}(x)$ decays like $|x|^{-2l}$ in the far-field,  then  $\nabla\rho_0^{\delta-1}$ (with $0<\delta<1$) cannot belong to  $L^2$  for   any $l>0$.
\end{rmk}

\begin{rmk} Theorem \ref{thm1} shows that $\nabla \rho^{\delta-1}\in L^{\infty}$ $(0< \delta<1)$ is bounded, which means that the vacuum occurs if and only if in the far field.
\end{rmk}
\vspace{1mm}

\noindent {\bf Structure of the Paper:}

The rest of this paper is organized as follows.
Section \ref{section2} presents preliminary lemmas, including Sobolev inequalities and elliptic estimates.
Section \ref{section3} is devoted to proving Theorem \ref{thm1}. First,
we enlarge the original initial-boundary value problem (\ref{prob1})-(\ref{1.8})  into (\ref{ibvp-sect3}) in terms of the following variables
\begin{align*}
 \phi=\frac{A\gamma}{\gamma-1}\rho^{\gamma-1},\quad \psi=\frac{\delta}{\delta-1}\nabla\rho^{\delta-1},\quad u
 \end{align*}
in $\S$\ref{sub3.1}. Then the original Navier-Stokes equations (\ref{prob1}) can be  enlarged as
\begin{equation*}
\left\{
\begin{array}{llll}
\phi_t+u\cdot\nabla\phi+(\gamma-1)\phi\mbox{div}u=0,\\
u_{t}+u\cdot\nabla u+\nabla\phi +a\phi^{2\kappa}Lu=\psi\cdot Q(u)\\
\psi_t+\nabla(u\cdot\psi)+(\delta-1)\psi\mbox{div}u+\delta a \phi^{2\kappa}\nabla \mbox{div}u=0,
\end{array}
\right.
\end{equation*}
where
\begin{align*}
Lu=-\alpha \Delta u-(\alpha+\beta)\nabla\mbox{div}u,\quad Q(u)=\alpha\left(\nabla u+(\nabla u)^{T}\right)+\beta \mbox{div}u\mathbb{I}_3.
\end{align*}

  Second, in $\S$\ref{sub3.2}, we construct a  linearization (\ref{linear}) of the  nonlinear problem (\ref{ibvp-sect3}). Our approach is based on a careful structural analysis of the nonlinear equations. For the linearized problem, we establish the existence of global approximate solutions,  provided that $\phi(x,0) \equiv \phi_0$ has a positive lower bound $\sigma$.
Regarding the structure of the linearized equations (\ref{linear}), the following two points should be noted:
(i) \
In the linearization procedure, following the approach of  \cite{XZ2021jmpa},
 we first linearize the equation for $\xi=\phi^{2\kappa}$ as:
 \begin{align}\label{intro-xi}
\xi_t+v\cdot\nabla\xi+(\delta-1)g\mbox{div}v=0,
\end{align}
and then use $\xi$ to redefine $\psi=\frac{a\delta}{\delta-1}\nabla \xi$.
 The linearized equations for $u$ are chosen as
\begin{align*}
u_{t}+v\cdot\nabla u+\nabla\phi +a\xi Lu=\psi\cdot Q(u).
\end{align*}
Combining equation (\ref{intro-xi}) with the relation $\psi=\frac{a\delta}{\delta-1}\nabla \xi$  yields
 \begin{align*}\label{intro-psi}
	\psi_t+\sum_{k=1}^3A_k(v)\partial_k\psi+(\nabla v)^T\psi+a\delta(g\nabla\mbox{div}v+\nabla g\mbox{div}v)=0,
\end{align*}
which provides the necessary estimates for $\psi$.
It is important to emphasize that, when establishing a priori estimates, $\xi$ and $\phi$ are treated as independent functions, without assuming any relation such as $\xi = \phi^{2\kappa}$.

(ii) \ Unlike the standard choice in the  Cauchy problem (see, e.g., \cite{XZ2021jmpa}), where terms such as
$v\cdot\nabla v$ and $\psi\cdot Q(v)$ are typically used in the linearization of the velocity equation,
we instead employ the terms $v\cdot\nabla u$ and $\psi \cdot Q(u)$ in the linearized equation for $u$.
This  choice is essential to our analysis, as it enables the derivation of singular weighted estimates for the velocity field.

In $\S$\ref{sub3.3}, the a priori estimates independent of the lower bound $\sigma$ of the solutions $(\phi^{\sigma},\xi^{\sigma},u^{\sigma})$ to the linearized problem (\ref{linear}) are established.
For the initial-boundary value problem, particularly in establishing a priori estimates for the velocity field, we encounter two main difficulties compared with the Cauchy problem.
First,
 the weighted first-order derivative term $\|\xi\nabla u\|_{L^2}$ will appear as a remainder when applying elliptic estimates to the operator $L(\xi u)$ on exterior domains.
 To close the a priori estimates, it becomes necessary to establish a more singular weighted estimate for $\|\xi\nabla u\|_{L^2}$, rather than the natural weighted estimate  $\|\sqrt{\xi}\nabla u\|_{L^2}$.
Second, the estimation of $\|\xi \nabla^2 u\|_{D^2}$
presents further challenges. On one hand, directly applying elliptic estimates to $L(\xi u)$ does not yield control of
$\|\xi \nabla^4u\|_{L^2}$, due to the lack of estimates for $\nabla^4\xi$. On the other hand, applying elliptic estimates to $L(\xi D^2u)$
is also not feasible, as no boundary conditions are available for $D^2u$. To circumvent these issues, we employ an approach based on conormal space techniques.

In $\S$\ref{sub3.4}, based on the above analysis for the choice of the linearization, the unique solvability of the classical solution away from vacuum to the nonlinear reformulated problem (\ref{ibvp-sect3}) through an iteration process is given, whose life span is uniformly positive with respect to the lower bound $\sigma $ of $\phi_0$.
Based on the conclusions in $\S$\ref{sub3.4}, the solution to the nonlinear reformulated problem which allows for vacuum in the far field can be recovered by taking the limit as $\sigma\rightarrow 0$ in  $\S$\ref{sub3.5}.
Finally, in $\S$\ref{sub3.6}, it is shown that the existence of a solution to the reformulated problem (\ref{ibvp-sect3})  in fact implies Theorem \ref{thm1}.

\section{Preliminaries}\label{section2}
In this section, we list some basic lemmas that will be used later. The first one is the well-known Gagliardo-Nirenberg inequalities.
\begin{lem}\label{lem-2-1}
Let $\Omega\subset\mathbb{R}^{3}$ be the exterior of a bounded domain with  smooth boundary. Then
for $p\in[2,6]$, $q\in(1,\infty)$, and $r\in(3,\infty)$, there exists some generic constant $C>0$
such that for
$f\in H^{1}(\Omega)$, and $g\in L^{q}(\Omega)\cap D^{1,r}(\Omega)$,
it holds that
\begin{align}
\|f\|^p_{L^{p}}\leq C\|f\|^{\frac{6-p}{2}}_{L^{2}}\|\nabla f\|^{\frac{3p-6}{2}}_{L^{2}},\quad
\|g\|_{L^{\infty}}\leq C\|g\|^{\frac{q(r-3)}{3r+q(r-3)}}_{L^{q}}\|\nabla g\|^{\frac{3r}{3r+q(r-3)}}_{L^{r}}.\nonumber
\end{align}
\end{lem}

The second lemma presents  some compactness results obtained via  Aubin-Lions Lemma.
\begin{lem}\label{lem-AL}(\cite{A-L})
Let $X_0\subset X\subset X_1$ be three Banach spaces. Suppose that $X_0$ is compactly embedded in $X$, and $X$ is continuously embedded in $X_1$. Then the following statements hold:
\\ $(i)$\  If $F$ is bounded in $L^p([0,T]; X_0)$ for $1\leq p<\infty$, and $F_t$ is bounded in $L^1([0,T]; X)$, then $F$ is relatively compact in $L^p([0,T];X)$;\\
$(ii)$\ If $F$ is bounded in $L^{\infty}([0,T|;X_0)$ and $F_t$ is bounded in $L^{p}([0,T];X_1)$ for $p>1$,
then $F$ is relatively compact in $C([0,T];X)$.
\end{lem}

The following lemma  allows one to estimate the $H^{m}$-norm of a vector-valued function $v$ based on  its
$H^{m-1}$-norm of $\mbox{curl}v$ and $\mbox{div} v$ (see \cite{XX2007}).
\begin{lem}\label{lem-div-curl}
Let $\Omega$ be a domain in $\mathbb{R}^{N}$ with smooth boundary $\partial\Omega$ and outward normal $n$. Then there exists a constant $C>0$,
such that
\begin{align*}
&\|v\|_{H^{s}(\Omega)}\leq C\left(\|\mbox{div}v\|_{H^{s-1}(\Omega)}+\|\mbox{curl}v\|_{H^{s-1}(\Omega)}
+ |v\cdot n|_{H^{s-\frac{1}{2}}(\partial\Omega)}+\|v\|_{L^2}\right),\\
&\|v\|_{H^{s}(\Omega)}\leq C\left(\|\mbox{div}v\|_{H^{s-1}(\Omega)}+\|\mbox{curl}v\|_{H^{s-1}(\Omega)}
+ |v\times n|_{H^{s-\frac{1}{2}}(\partial\Omega)}+\|v\|_{L^2}\right),
\end{align*}
for any $v\in H^{s}(\Omega)$, $s\geq1$.
\end{lem}

In addition, since the first Betti number of the exterior domain $\Omega$  vanishes,   we can apply Theorem 3.2 in \cite{wahl} to obtain the following  refined estimate  of $\|\nabla v\|_{L^2}$ which is crucial in the case of  exterior domain to deal with  $\|\xi u\|_{D^1}$. We also note that the topological property of $\Omega$ is necessary  for the following  Proposition.
\begin{prop} (\cite{wahl})\label{prop-imp} Let $\nabla v\in L^{2}(\Omega)$, and $v\cdot n|_{\partial\Omega}=0$. The estimate
	\begin{equation}\label{important}
		\|\nabla v\|_{L^2}\leq C_{\Omega}(\|\mbox{div}v\|_{L^2}+\|\mbox{curl}v\|_{L^2})
	\end{equation}
	is true for all $v$ as above if and only if $\Omega$ has a first Betti number of zero.
\end{prop}

\begin{lem}(\cite{DL2024})\label{lem-ell}
	Assume that $\Omega$ is an exterior domain of the simply connected domain $O$ in $\mathbb{R}^3$
 with smooth boundary.
		For any $q\in[2,4]$ there exists some positive constant $C=C(q,O)$ such that for every $v\in\{D^{1,2}|v(x)\rightarrow 0\quad\mbox{as}\quad |x|\rightarrow\infty\}$, it holds
		\begin{equation}\label{est-boundary}
			|v|_{L^{q}(\partial\Omega)}\leq C\|\nabla v\|_{L^2(\Omega)}.
		\end{equation}
Moreover, for $k\geq1$, if $v\in \{D^{k+1,2}\cap D^{1,2}v(x)\rightarrow0, \text{as}\  |x|\rightarrow\infty\}$ with $v\cdot n|_{\partial\Omega}=0$ or $v\times n|_{\partial\Omega}=0$, then there exists some constant $C=C(k, O)$ such that
	\begin{align}
		\|\nabla v\|_{H^k}\leq C(\|\mbox{div}v \|_{H^k}+\|\mbox{curl}v \|_{H^k}+\|\nabla v\|_{L^2}).
	\end{align}
\end{lem}

Next, we present some  elliptic estimates for Lam\'{e} equations subject to  Navier-slip  boundary conditions  in the smooth exterior domain $\Omega$ by using Lemmas \ref{lem-div-curl}-\ref{lem-ell}, which will be frequently used in the proofs of the main results later.

\begin{lem}\label{lem-elliptic-1} \cite{DL2024}
Assume that  $\Omega$ is the exterior domain of a bounded  domain in $\mathbb{R}^3$ with smooth boundary.
Suppose that $u$ satisfy the Lam\'{e} equation:
	\begin{equation*}
		\left\{
		\begin{array}{llll}
			-\mu\Delta u-(\mu+\lambda)\nabla\mbox{div}u=f\quad \mbox{in}\quad \Omega,\\
			u\cdot n=0,\quad \mbox{curl} u\times n=-K(x)u\quad  \mbox{on}\quad \partial\Omega,\\
u(x)\rightarrow 0, \quad \mbox{as}\ |x|\rightarrow+\infty,
		\end{array}
		\right.
	\end{equation*}
where $n$ stands for the outward unit normal to $\partial\Omega$, and  $K(x)$ is a $3\times 3$ symmetric bounded matrix.
There exists a positive constant $C$ depending only on $\mu,\lambda,\Omega$ and the matrix $K$ such that it holds
\begin{equation*}
	\|\nabla^2u\|^2_{H^{s}}\leq C\big(\|f\|^2_{H^{s}}+\| \nabla u\|^2_{L^{2}}\big),
\end{equation*}
where $s=0,1$.
 \end{lem}

Next, we provide the elliptic estimate in the boundary conditions with some unknown variables.
\begin{lem}\label{Appendix}
Assume $\Omega$ is the same as in Lemma \ref{lem-elliptic-1},
	suppose that $u$ satisfy the Lam\'{e} equation:
	\begin{equation}\label{equ-lemma}
		\left\{
		\begin{array}{llll}
			-\mu\Delta(\xi u)-(\mu+\lambda)\nabla\mbox{div}(\xi u)=f\quad \mbox{in}\quad \Omega,\\
		\xi u\cdot n=0,\quad \mbox{curl} (\xi u)\times n=(\nabla\xi \cdot n)u-K(x)\xi u\quad \mbox{on}\quad \partial\Omega,\\
\xi u(x)\rightarrow 0, \quad \mbox{as}\ |x|\rightarrow+\infty.
		\end{array}
		\right.
	\end{equation}
	There exists a positive constant $C$ depending only on $\mu,\lambda,\Omega$ and the matrix $K$ such that it holds
	\begin{equation*}
		\|\nabla^2(\xi u)\|^2_{H^{s}}\leq C\left(\|f\|^2_{H^{s}}+\| \xi\nabla u\|_{L^2}^2\right)
		 +C\left(\|\nabla^2\xi\|^2_{H^1}+\|\nabla\xi\|^2_{L^\infty}\right)\|u\|^2_{H^{s+1}},	
	\end{equation*}
	where $s=0,1$.
\end{lem}

At last in the section, we recall the elliptic estimate for the Dirichlet boundary.
\begin{lem}\label{lem6-0-1}(\cite{M})
	Let $\Omega$ be any exterior domain or the half space. Suppose that $u$ satisfy the Lam$\acute{e}$ equation:
\begin{equation*}
\left\{
\begin{array}{lll}
-\mu\Delta u-(\mu+\lambda)\nabla\mbox{div} u=f,\\
u|_{\partial\Omega}=0,\\
 u(x)\rightarrow 0\quad \mbox{as}\ |x|\rightarrow+\infty.
\end{array}
\right.
\end{equation*}
	 Then for $s=0,1$,
	\begin{equation*}
			\|\nabla^{2}u\|_{H^s}^2\leq C(\|f\|^2_{H^{s}}+\|\nabla u\|_{L^2}^2).
	\end{equation*}
	\end{lem}

\section{Existence of Regular Solutions}\label{section3}
In this section, we will give the proof of Theorem \ref{thm1} by subsections \ref {sub3.1}-\ref{sub3.6}.  For this purpose, we first reformulate the original  initial-boundary value problem (\ref{prob1})-(\ref{1.8})
in terms of some new variables, and then establish the local well-posedness of the strong solution to the reformulation problem (\ref{ibvp-sect3}).
\vspace{1mm}

\subsection{Reformulation}\label{sub3.1}

 In this subsection, we introduce some new variables as
 \begin{align*}
 \phi=\frac{A\gamma}{\gamma-1}\rho^{\gamma-1},\quad \psi=\frac{\delta}{\delta-1}\nabla\rho^{\delta-1},
 \end{align*}
and set  some constants as
\begin{align*}
a=\left(\frac{A\gamma}{\gamma-1}\right)^{\frac{1-\delta}{\gamma-1}},\quad \kappa=\frac{\delta-1}{2(\gamma-1)}<0.
\end{align*}
Then  the original initial-boundary value problem (\ref{prob1})-(\ref{1.8})  can be  reformulated
as
\begin{equation}\label{ibvp-sect3}
\left\{
\begin{array}{llll}
\phi_t+u\cdot\nabla\phi+(\gamma-1)\phi\mbox{div}u=0,\\
u_{t}+u\cdot\nabla u+\nabla\phi +a\phi^{2\kappa}Lu=\psi\cdot Q(u),\\
\psi_t+\nabla(u\cdot\psi)+(\delta-1)\psi\mbox{div}u+\delta a \phi^{2\kappa}\nabla \mbox{div}u=0,\\
(\phi,u,\psi)|_{t=0}=(\phi_0,u_0,\psi_0)\equiv\left(\frac{A\gamma}{\gamma-1}\rho_0^{\gamma-1},u_0,
 \frac{\delta}{\delta-1}\nabla\rho_0^{\delta-1}\right),\\
 u\cdot n=0,\quad   \mbox{curl}u\times n=-K(x)u,\quad \mbox{on}\quad \partial\Omega,\\
(\phi, u,\psi) \rightarrow (0, 0, 0), \ \mbox{as}\  |x|\rightarrow+\infty,\  t\geq 0,
\end{array}
\right.
\end{equation}
where
\begin{align*}
Lu=-\alpha \Delta u-(\alpha+\beta)\nabla\mbox{div}u,\quad Q(u)=\alpha\left(\nabla u+(\nabla u)^{T}\right)+\beta \mbox{div}u\mathbb{I}_3.
\end{align*}

	To prove Theorem \ref{thm1}, we turn to  establish the local well-posedness to the reformulated problem (\ref{ibvp-sect3}).  More precisely,

\begin{thm}\label{thm2}
Let parameters $(\gamma,\delta,\alpha,\beta)$ satisfy
\begin{equation*}
	\gamma>1,\quad 0<\delta<1,\quad \alpha>0,\quad 2\alpha+3\beta\geq0.
	\end{equation*}
If the initial data $(\phi_0,u_0,\psi_0)$ satisfy the following  conditions:
 \begin{equation}\label{intial2}
		\phi_0>0,\quad (\phi_0,u_0)\in H^3,\quad \psi_0\in D^1\cap D^2,\quad \nabla\phi_0^{\kappa}\in L^4,
	\end{equation}
	and the initial compatibility conditions:
	\begin{align}\label{g}
 g_1=\phi_0^{2\kappa}\nabla u_0,\quad  g_2=a\phi_0^{2\kappa}Lu_0,\quad
g_3=\phi_0^{2\kappa}\nabla (a\phi_0^{2\kappa}Lu_0),
\end{align}
for some $(g_1,g_2,g_3)\in L^2$. Moreover, we assume that the initial data satisfy  the compatibility condition  with the boundary conditions. Then there exist a time $T_{*}>0 $ and a unique  local strong solution $(\phi,u,\psi=\frac{a\delta}{\delta-1}\nabla\phi^{2\kappa})$ to the initial-boundary value problem (\ref{ibvp-sect3}) satisfying
\begin{align}\label{thm2-reg}
&\phi\in C([0,T_{*}];H^3), \quad \phi_t\in C([0,T_{*}];H^2), \quad
\phi^{-1}\nabla\phi\in L^\infty([0,T_{*}];L^\infty\cap L^6\cap D^{1,3}\cap D^2), \nonumber\\[2mm]
&\psi\in C([0,T_{*}];D^1 \cap D^2),  \quad \psi_t\in C([0,T_{*}];H^1), \quad \psi_{tt}\in L^2([0,T_{*}];L^2), \nonumber\\[2mm]
&\phi^{-2\kappa}\in L^\infty([0,T_{*}];L^\infty\cap D^{1,6}\cap D^{2,3}\cap D^3), \nonumber\\[2mm]
 &u\in C([0,T_{*}];H^3)\cap L^2([0,T_{*}]; H^4),\quad \phi^{2\kappa}\nabla u\in L^\infty([0,T_{*}];H^2), \nonumber\\[2mm]
& \phi^{2\kappa}\nabla^2 u\in L^2([0,T_{*}];D^2),\quad
 (\phi^{2\kappa}\nabla^2 u)_t\in L^2([0,T_{*}];L^2), \nonumber\\[2mm]
 & u_t\in C([0,T_{*}];H^1)\cap L^2([0,T_{*}]; H^2), \quad \phi^{\kappa}u_{tt}\in L^2([0,T_{*}];L^2).
 \end{align}
\end{thm}

The proof of Theorem \ref{thm2} will be divided into four subsequent subsections \ref{sub3.2}-\ref{sub3.5}.
In subsection \ref{sub3.6}, it will be demonstrated that Theorem \ref{thm2} implies Theorem \ref{thm1}.

\subsection{Linearization}\label{sub3.2}

Guided by the idea in \cite{XZ2021jmpa}, we first linearize the equation for $\xi=\phi^{2\kappa}$ as
\begin{align}\label{equ-xi}
 \xi_t+v\cdot\nabla\xi+(\delta-1)g\mbox{div}v=0,
\end{align}
and then use $\xi$ to define $\psi=\frac{a\delta}{\delta-1}\nabla\xi$.
The linearized equations for $u$ are chosen as
\begin{align*}
u_{t}+v\cdot\nabla u+\nabla\phi +a\xi Lu=\psi\cdot Q(u).
\end{align*}
In order to proceed with nonlinear system (\ref{ibvp-sect3}), we first need to address  the following linearized problem:
\begin{equation}\label{linear}
\left\{
\begin{array}{llll}
\phi_t+v\cdot\nabla\phi+(\gamma-1)\phi\mbox{div}v=0,\\
u_{t}+v\cdot\nabla u+\nabla\phi +a\xi Lu=\psi\cdot Q(u),\\
\xi_t+v\cdot\nabla\xi+(\delta-1)g\mbox{div}v=0,\\
(\phi,u,\xi)|_{t=0}=(\phi_0,u_0,\xi_0)\equiv(\phi_0,u_0,\phi_0^{2\kappa}),\\
u\cdot n=0,\quad  \mbox{curl}u\times n=-K(x)u\quad \mbox{on}\quad \partial\Omega,\\
(\phi,u,\xi)\rightarrow(\phi^{\infty},0,\xi^{\infty})\equiv (\phi^{\infty},0,(\phi^{\infty})^{2\kappa}),\quad
\mbox{as}\ |x|\rightarrow +\infty,\quad t\geq0,
\end{array}
\right.
\end{equation}
where $\phi^{\infty}$ is a positive constant, $v$ is a  known vector and $g$ is a known real  function satisfying $(v(0, x), g(0, x)) =(u_0, \xi_0) =(u_0, \phi_0^{2\kappa})$ and
$v$ has the same boundary conditions as $u$, that is
\begin{equation*}
v\cdot n|_{\partial\Omega}=0,\quad  \mbox{curl}v\times n|_{\partial\Omega}=-K(x)v.
\end{equation*}
Moreover,  $g$ and $v$ belong  to the following spaces:
\begin{align*}
&g\in L^{\infty}\cap C([0,T]\times \Omega),\
 \nabla g\in C([0,T];H^2),\ g_t\in C([0,T];H^2),\nonumber\\
&v\in C([0,T]; H^3)\cap  L^2([0,T]; H^4),\ v_t\in C([0,T]; H^1)\cap  L^2([0,T]; H^2),\ v_{tt}\in L^2([0,T]; L^2).
\end{align*}

In the subsequent two subsections, we first solve the linearized problem \eqref{linear} when the initial density is away
from vacuum ($\inf\phi_0>0$), then we will establish some  uniform estimates which are independent of the lower bound of $\phi_0$.

The following global well-posedness of a strong solution to (\ref{linear}) can be obtained  when $\phi_0> \sigma$ for some positive constant $\sigma$ by the standard theory(\cite{Majda1986book}).
\begin{lem}\label{global}
Assume that
$(\phi_0,u_0,\xi_0)$
satisfy
\begin{equation*}
\phi_0>\sigma,\quad  \phi_0-\phi^{\infty}\in H^3,\quad u_0\in H^3,
\end{equation*}
for some positive constant $\sigma$. Then for any finite constant $T>0$, there exists a unique strong solution $(\phi, u, \xi)$ in $[0,T]\times \Omega$
to (\ref{linear}) such that
\begin{align}\label{lem3.2-reg}
&\phi-\phi^{\infty}\in C([0,T]; H^3),\quad \phi_t\in C([0,T]; H^2),\nonumber\\
&\xi\in L^{\infty}\cap C([0,T]\times \Omega),\quad
 \nabla\xi\in C([0,T];H^2),\quad \xi_t\in C([0,T];H^2), \nonumber\\
&u\in C([0,T]; H^3)\cap L^2([0,T]; H^4), \  u_t\in C([0,T]; H^1)\cap L^2([0,T]; H^2),\nonumber\\
& u_{tt}\in L^2([0,T]; L^2).
\end{align}
\end{lem}

\begin{rmk}
From the initial assumption on $\phi_0$, it holds that
\begin{align*}
\xi_0\in L^{\infty},\quad \nabla \xi_0\in H^2
\end{align*}
thanks to $\phi_0> \sigma, \phi_0-\phi^{\infty}\in H^3$ and $\xi_0=\phi_0^{2\kappa}$ ($\kappa<0$). However, the boundedness of $\|\xi_0\|_{ L^{\infty}}$ and $\|\nabla \xi_0\|_{H^2}$  may depend on $\sigma$.
\end{rmk}

\subsection{A priori estimates independent of $\sigma$}\label{sub3.3}

Now we are going to establish some uniform estimates on the solutions derived in Lemma \ref{global}. To this
end, we create a fixed value $T>0$ and a positive constant $c_0$ independent of $\sigma$
such that
\begin{align}\label{c0}
		I_0=&2+\phi^{\infty}+\|\phi_0-\phi^{\infty}\|_3+\|u_0\|_3+\|g_1\|_{L^2}
		+\|g_2\|_{L^2}+\|g_3\|_{L^2}\nonumber\\
&+\|\nabla\xi_0\|_{D^1\cap D^2}+\|\xi_0^{-1} \|_{L^\infty\cap D^{1,6}\cap D^{2,3}\cap D^3}+\|\xi_0^{-1}\nabla\xi_0\|_{L^\infty\cap D^{1,3}\cap D^2}\leq c_0,
		\end{align}
where
\begin{align*}\label{g2}
 g_1=\phi_0^{2\kappa}\nabla u_0,\quad  g_2=a\phi_0^{2\kappa}Lu_0,\quad
g_3=\phi_0^{2\kappa}\nabla (a\phi_0^{2\kappa}Lu_0).
\end{align*}


\begin{rmk}
We now present several weighted estimates concerning the initial values, which play an essential role in the  a priori estimates of $u$.
It is important to note that the weight functions utilized in our analysis are distinct from those defined in \cite{XZ2021jmpa}.
\begin{equation}\label{key-phi}
\|\xi_0\nabla^2\phi_0\|_{L^2}\leq Cc_0^{2},
\end{equation}
\begin{equation}\label{key-2u}
\|\xi_0\nabla^2u_0\|_{L^2}\leq Cc_0^{2},\quad
\|\xi_0\nabla(\psi_0 Q(u_0))\|_{L^2}\leq Cc_0^{2}.
\end{equation}
\end{rmk}

\vspace{2mm}

Now we fixe $T>0$, and assume that there exist some time $T^{*}\in(0,T)$ and constants $c_i$ ($i = 1, 2, 3,4$) such that
\begin{equation*}
1\leq c_0\leq c_1\leq c_2\leq c_3\leq c_4,
\end{equation*}
\begin{align}\label{known}
&\sup_{0\leq t\leq T^{*}}\|\nabla g(t)\|^2_{D^1\cap D^2}\leq c_1^2,\nonumber\\
&\sup_{0\leq t\leq T^{*}}\|v(t)\|^2_1+\int_0^{T^{*}}\left(\|\nabla^2 v\|_{L^2}^2+\|v_t\|_{L^2}^2\right)ds\leq c_2^2,\nonumber\\
&\sup_{0\leq t\leq T^{*}}\left(\|\nabla^2v(t)\|_{L^2}^2+\|v_t(t)\|_{L^2}^2+\|g\nabla^2v(t)\|_{L^2}^2\right)+\int_0^{T^{*}}\left(\|\nabla^3v\|_{L^2}^2+\|\nabla v_t\|_{L^2}^2\right)ds\leq c_3^2,\nonumber\\
&\sup_{0\leq t\leq T^{*}}\left(\|\nabla^3v(t)\|_{L^2}^2+\|\nabla v_t(t)\|_{L^2}^2
\right)+\int_0^{T^{*}}\left(\|\nabla^4v\|_{L^2}^2+\|\nabla^2v_t\|_{L^2}^2+\|v_{tt}\|_{L^2}^2\right)ds
\leq c_4^2,\nonumber\\
&\sup_{0\leq t\leq T^{*}}\left(\|g\nabla^2v(t)\|_{D^1}^2
\right)+\int_0^{T^{*}}\left(\|(g\nabla^2v)_t\|_{L^2}^2+\|g\nabla^2v\|_{D^2}^2\right)ds
\leq c_4^2,\nonumber\\
&\sup_{0\leq t\leq T^{*}}\left(\|g_t(t)\|_{L^\infty}^2+\|\nabla g_t(t)\|_{L^2}^2\right)\leq c_4^2,
\end{align}
where $T^{*}$ and $c_i$ ($i=1,2,3,4$) will be determined later (see (\ref{c_i})), and depend only
on $c_0$, $T$ and the fixed constants $A,\alpha,\beta,\gamma$ and $\delta$.

Let $(\phi,\xi, u)$ be the unique strong solution to (\ref{linear}) on $[0,T]\times \Omega$. A series of uniform local (in time) estimates will be established below.  Hereinafter, we use $C\geq 1$ to denote a generic positive constant that depends only on fixed constants $A,\alpha,\beta,\gamma, \delta$ and $T$.\\

{\it{3.3.1 The a priori estimates for $\phi$ and $\psi$}}

The a priori estimates for $\phi$ (Lemma \ref{lem-phi}) and $\psi$ (Lemma \ref{lem-psi}) are analogous to those for the Cauchy problem of the Navier-Stokes equation (see Lemmas 3.2-3.3 in \cite{XZ2021jmpa}).
\begin{lem}\label{lem-phi}
Let $(\phi,\xi, u)$ be the unique  strong solution to (\ref{linear}) on $[0,T]\times \Omega$. Then for $0\leq t\leq T_1=\min\{T^{*},(1+c_4)^{-2}\}$, it holds
\begin{equation*}\label{phi-lem}
	\begin{split}
&\|\phi(t)\|_{L^\infty}+\|\phi(t)-\phi^{\infty}\|_3\leq Cc_0,\quad
\|\phi_t(t)\|_{L^2}\leq Cc_0c_2,\\
& \|\phi_t(t)\|_{D^1}\leq Cc_0c_3,\quad\|\phi_t(t)\|_{D^2}\leq Cc_0c_4,\quad \|\phi_{tt}(t)\|_{L^2}\leq Cc_4^{3}.
\end{split}
\end{equation*}
\end{lem}

\begin{lem}\label{lem-psi}
Let $(\phi,\xi, u)$ be the unique strong solution to (\ref{linear}) on $[0,T]\times \Omega$. Then for $0\leq t\leq T_1$, it holds
\begin{equation*}\label{psi}
	\begin{split}
&\|\psi(t)\|_{L^\infty}+\|\psi(t)\|_{D^1\cap D^2}\leq C c_0,\quad
\|\psi_t(t)\|_{L^2}\leq Cc_3^2,\quad\|\psi_t(t)\|_{D^1}\leq Cc_4^2,\\
&\|\xi_t(t)\|_{L^\infty}\leq Cc_3^{\frac{3}{2}}c_4^{\frac{1}{2}}, \quad  \|\nabla\xi_t\|_{L^2}\leq Cc_3^2,
\quad \int_0^t\|\psi_{tt}\|_{L^2}^2ds\leq Cc_4^4.
\end{split}
\end{equation*}
\end{lem}

{\it{3.3.2. The a priori estimates for two $\xi-$related auxiliary variables}}

In order to obtain the uniformly a priori estimates for the solutions to the corresponding nonlinear problem as in
section \ref{sub3.4}, it is helpful to give some more  estimates for another two new $\xi$-related quantities:
\begin{align}\label{zeta-linear}
	\varphi=\xi^{-1},\quad  f=\psi\varphi=\frac{a\delta}{\delta-1}\frac{\nabla\xi}{\xi}=\left(f^{(1)},f^{(2)},f^{(3)}\right).
\end{align}
Then, it is easy to see that
\begin{equation}\label{intial3}
	\varphi_0=\phi_0^{-2\kappa}\in L^\infty\cap D^{1,6}\cap D^{2,3}\cap D^3,\quad f_0=\frac{2a\delta\kappa}{\delta-1}\frac{\nabla\phi_0}{\phi_0}\in L^\infty\cap L^6\cap D^{1,3}\cap D^2,
\end{equation}
thanks to the initial assumptions \eqref{intial2}.
\begin{lem}\label{lem-varphi}
	Let $(\phi,\xi, u)$ be the unique  strong solution to (\ref{linear}) on $[0,T]\times \Omega$. Then for $0\leq t\leq T_2=\min\{T^{*},(1+c_4)^{-4}\}$, it holds
	\begin{align*}
			&\|\varphi(t)\|^2_{D^{1,6}\cap D^{2,3}\cap D^3}+\|f(t)\|^2_{L^\infty\cap L^6\cap D^{1,3}\cap D^2}\leq Cc^4_0,\\
			&\xi(x,t)>\frac{1}{2c_0},\quad \frac{2}{3}\sigma^{-2\kappa}<\varphi(x,t)\leq 2\|\varphi_0\|_{L^\infty}\leq 2c_0,\\
			& \|\varphi_t(t)\|^2_{ L^6\cap D^{1,3}\cap D^2}+\|f_t(t)\|^2_{L^3\cap D^1}\leq Cc^{10}_4.
		\end{align*}
\end{lem}
\begin{proof}
	According to the definition of $\varphi$ and the equation of $\xi$ $(\ref{linear})_3$, it is easy to see that $\varphi$ satisfies the following equation
		\begin{equation}\label{varphi-linear}
		\varphi_t+v\cdot\nabla\varphi-(\delta-1)g\varphi^2\mbox{div}v=0.
	\end{equation}
Along with the particle path $W(x_0,t)$ defined as
\begin{equation*}
	\left\{
	\begin{array}{ll}
		\frac{d}{ds}W(x_0,t)=v(W(x_0,t),s),\quad 0\leq t\leq T;\\
W(x_0,0)=x_0,\quad x\in\Omega,
	\end{array}
	\right.
\end{equation*}
 we obtain
\begin{align*}
	 \varphi(W(x_0,t),t)=\varphi_0(x_0)\left(1+(1-\delta)\varphi_0(x_0)\int_0^tg\mbox{div}v(W(x_0,s),s)ds\right)^{-1}.
	\end{align*}
Then for $0\leq t\leq T_2=\min\{T^{*},(1+c_4)^{-4}\}$,
\begin{align*}
\frac{2}{3}\sigma^{-2\kappa}<\varphi(x,t)\leq 2\|\varphi_0\|_{L^\infty}\leq 2c_0, \quad \text{for}\quad (x,t)\in\Omega\times[0,T_2].
	\end{align*}
	Applying the standard energy estimates and integrating by parts  with the boundary condition $v\cdot n|_\Omega=0$, we obtain the desired estimates for $\varphi$.
 Using the estimates of $\psi$, $\varphi$ and the relation $f=\psi\varphi$, we can directly obtain the estimates for $f$.
\end{proof}

{\it{3.3.3. The a priori estimates for u}}\\

 Based on the estimates of $\phi$ and $\xi$ obtained in Lemmas \ref{lem-phi}-\ref{lem-varphi}, we are now ready to give the estimates for the velocity.
This subsection is devoted to the key distinctions between the initial-boundary value problem and the Cauchy problem, which constitutes a major contribution of this work.

 As mentioned in the introduction, unlike the Cauchy problem, two main difficulties arise in our analysis. First,
 the weighted first-order derivative term $\|\xi\nabla u\|_{L^2}$ will appear as a remainder when applying elliptic estimates to the operator $L(\xi u)$ on exterior domains.
 To close the a priori estimates, it becomes necessary to establish a more singular weighted estimate for $\|\xi\nabla u\|_{L^2}$, rather than the natural weighted estimate  $\|\sqrt{\xi}\nabla u\|_{L^2}$ (see Step 2 in the proof of Lemma \ref{lem-u}).
However, the presence of strongly singular weights in the estimation leads to certain singular terms that require delicate analysis, such as the terms $R_5$ in (\ref{step2-R5}) and  $I_6$ in (\ref{Step4-R_6}).
Second, the estimation of $\|\xi \nabla^2 u\|_{D^2}$
presents further challenges. To circumvent these issues, we employ an approach based on conormal space techniques (see Lemma \ref{lem-u-2}).

\begin{lem}\label{lem-u}
	Let $(\phi,\xi,u)$ be the unique strong solution to (\ref{linear}) on $[0,T]\times \Omega$. Then

(i)\ for
$0\leq t\leq T_2$, it holds
	\begin{align*}
 &\|u(t)\|^2_{L^2}+\int_0^t\|\sqrt{\xi}\nabla u\|^2_{L^2}ds\leq Cc_0^2,\quad
	 \|\xi\nabla u(t)\|_{L^2}^2
			+\int_0^t\|\sqrt{\xi} u_t\|^2_{L^2}ds
			\leq C c_0^4,\\
	&\int_0^t\left(\|u_t\|^2_{L^2}+\|\xi\nabla^{2}u\|^2_{L^2}\right)ds\leq Cc_0^5,\\
&\|u_t(t)\|^2_{L^2}+\int_0^t\|\sqrt{\xi}\nabla u_t\|^2_{L^2}ds
\leq Cc_0^7,\quad
\int_0^t\|\nabla u_t\|^2_{L^2}ds
\leq  Cc_0^8,
\end{align*}
and
\begin{align}\label{nabla^3u}
\|\xi\nabla^2u\|^2_{L^2}\leq Cc_0^6c_2c_3, \quad  \int_0^t(\|\xi \nabla^3u\|^2_{L^{2}}+\|\xi \nabla^2u\|^2_{D^{1}}+\|\xi u\|^2_{D^{3}})ds\leq Cc_0^8.
\end{align}

(ii) for $0\leq t\leq T_3=\min\{T^{*},(1+C_4)^{-6}\}$, there holds
\begin{align*}
\|\xi\nabla u_t(t)\|_{L^2}^2+\int_{0}^t\|\sqrt{\xi} u_{tt}\|_{L^2}^2ds
\leq Cc_0^7c_2^2, \quad
\|\nabla u_t(t)\|_{L^2}^2\leq Cc_0^9c_2^2, \quad
\int_{0}^t\| u_{tt}\|_{L^2}^2ds\leq Cc_0^8c_2^2,
\end{align*}
and
	\begin{align}\label{nabla^2u_t}
		\|\xi u\|^2_{D^{3}} +\|\xi \nabla^3u\|^2_{L^{2}}\leq Cc_0^8c_3^4,\ \int_0^t\|(\xi\nabla^2u)_t\|^2_{L^2}ds\leq Cc_0^8c_2^2,\
\int_0^t\|\nabla^2u_t\|^2_{L^2}ds\leq Cc_0^{10}c_2^2.
\end{align}
\end{lem}

\begin{proof}
	Rewriting  the  equation of  $(\ref{linear})_2$ as
	\begin{equation}\label{equ-u-linear}
		u_{t}+a\xi\left(\alpha\mbox{curl}^2u-(2\alpha+\beta)\nabla\mbox{div}u\right)=-v\cdot\nabla u-\nabla\phi+\psi\cdot Q(u),
	\end{equation}
	with the boundary conditions:
	\begin{equation}\label{boundary-u}
		u\cdot n=0,\quad  \mbox{curl}u\times n=-K(x)u,\quad \mbox{on}\quad \partial\Omega.
	\end{equation}
	{\it \bf {Step 1.}} We will first calculate $\|u\|_{L^2}$.
	Multiplying the equation (\ref{equ-u-linear}) by $u$,  integrating the result identity over $\Omega$, and  performing integration by parts, we obtain
	\begin{align}\label{1orderu}
		&\frac{1}{2}\frac{d}{dt}\|u\|_{L^2}^2+a\int_{\Omega}\xi \left(\alpha|\mbox{curl}u|^2+(2\alpha+\beta)|\mbox{div}u|^2\right)
-a\alpha\int_{\partial\Omega}\xi(\mbox{curl}u\times n)\cdot u\nonumber\\
		&=-a\int_{\Omega}\alpha \mbox{curl}u\cdot (\nabla\xi\times u)+(2\alpha+\beta) \mbox{div}u (u\cdot\nabla\xi )
+\int_{\Omega} \left(-v\cdot\nabla u-\nabla\phi+\psi\cdot Q(u)\right)\cdot u\nonumber\\
&\leq \varepsilon_1\|\sqrt{\xi}\nabla u\|^2_{L^2}
+C(\varepsilon_1)\|\xi^{-1}\|_{L^\infty}\left(\|\psi\|_{L^\infty}^2
+\|v\|^2_{L^\infty}\right)\|u\|_{L^2}^2
+C(\|u\|_{L^2}^2+\|\nabla\phi\|^2_{L^2}),
	\end{align}
where we have used $\psi=\frac{a\delta}{\delta-1}\nabla\xi$.

Note that the boundary integral on the left-hand side
\begin{align*}
-a\alpha\int_{\partial\Omega}\xi(\mbox{curl}u\times n)\cdot u=a\alpha\int_{\partial\Omega}\xi K(x)u\cdot u
\end{align*}
 is a good term provided $K(x)$ is positive.
Otherwise,  we can use (\ref{est-boundary}) to estimate the boundary  integral as follows,
\begin{align}\label{K-2}
&-a\alpha\int_{\partial\Omega}\xi(\mbox{curl}u\times n)\cdot u=a\alpha\int_{\partial\Omega}\xi K(x)u\cdot u\nonumber\\[2mm]
&\leq C\|K\|_{L^{\infty}} \|\nabla(\sqrt{\xi}u)\|^2_{L^2}\nonumber\\[2mm]
&\leq \varepsilon_0\|\sqrt{\xi}\nabla u\|^2_{L^2}+C\|\nabla\sqrt{\xi}\|^2_{L^{\infty}}\|u\|^2_{L^2},
\end{align}
for some small constant $\varepsilon_0>0$ due to the smallness of $K(x)$.

	Noting that
\begin{align}\label{bound}
	\|\sqrt{\xi} \nabla u\|^2_{L^2}\leq C \left(\|\sqrt{\xi}\mbox{curl}u\|^2_{L^2}+\|\sqrt{\xi}\mbox{div}u\|^2_{L^2}
+\|\nabla\sqrt{\xi}u\|^2_{L^2}
\right)
\end{align}
where we  apply the fact $\|\nabla v\|_{L^2}^2\leq C(\|\mbox{div}v\|_{L^2}^2
+\|\mbox{curl}v\|_{L^2}^2)$ with $v\cdot n|_{\partial\Omega}=0$.

Substituting  (\ref{bound}) into (\ref{1orderu}), and choosing $\varepsilon_1$ small implies
\begin{align*}
	\frac{d}{dt}\|u\|^2_{L^2}+\|\sqrt{\xi}\nabla u\|^2_{L^2}
\leq C\left(1+\|\xi^{-1}\|_{L^\infty}\|\psi\|_{L^\infty}^2
+\|\xi^{-1}\|_{L^\infty}\|v\|^2_{L^\infty}\right)\|u\|^2_{L^2}+C\|\nabla\phi\|^2_{L^2}.
\end{align*}
By using the Gronwall's inequality, it is clear that
\begin{align}\label{8}
	\|u(t)\|_{L^2}^2+\int_0^t\|\sqrt{\xi}\nabla u\|_{L^2}^2ds\leq Cc_0^2,\quad
\mbox{for} \quad 0\leq t\leq T_2.
\end{align}

\vspace{2mm}

{\it {\bf Step 2.}} To estimate $\|\xi\nabla u\|_{L^2}$.
We multiply the equation (\ref{equ-u-linear}) by $\xi u_t$, integrate the resulting identity over $\Omega$ to demonstrate
\begin{align}\label{xiut}
	\int_{\Omega}\xi u_t^2+a\int_{\Omega}\xi^2
	\left(\alpha\mbox{curl}^2u-(2\alpha+\beta)\nabla\mbox{div}u\right)
	\cdot u_t
	=
	\int_{\Omega}\xi \left(-v\cdot\nabla u-\nabla \phi +
	 \psi \cdot Q(u)\right)\cdot u_t.
	\end{align}
Using integration by parts, the right-hand side of the above becomes
\begin{align}\label{step2-1}
	\int_{\Omega}&\xi^2
	\left(\alpha\mbox{curl}^2u-(2\alpha+\beta)\nabla\mbox{div}u\right)\cdot u_t\nonumber\\
	=&\frac{1}{2}\frac{d}{dt}\int_{\Omega}\xi^2
	\left(\alpha|\mbox{curl}u|^2+(2\alpha+\beta)|\mbox{div}u|^2\right)
	-\int_{\Omega}\xi\xi_t
	\left(\alpha|\mbox{curl}u|^2+(2\alpha+\beta)|\mbox{div}u|^2\right)\nonumber\\
	&+\int_{\Omega}\left(\alpha(\nabla\xi^2\times u_t)\cdot \mbox{curl}u
	+(2\alpha+\beta)u_t\cdot \nabla\xi^2\mbox{div}u\right)-\alpha\int_{\partial\Omega} (\mbox{curl}u\times n)\cdot (\xi^2u_t).
\end{align}
The boundary term in (\ref{step2-1}) becomes
\begin{align}
-\int_{\partial\Omega} (\mbox{curl}u\times n)\cdot (\xi^2u_t)
=\frac{1}{2}\frac{d}{dt}\int_{\partial\Omega}\xi^2 K(x)u\cdot u- \int_{\partial\Omega}\xi\xi_t K(x)u\cdot u. \nonumber
\end{align}
Then, putting these  into (\ref{xiut}), one has
\begin{align}\label{u-0}
	&\frac{a}{2}\frac{d}{dt}\int_{\Omega}\xi^2
	\left(\alpha|\mbox{curl}u|^2+(2\alpha+\beta)|\mbox{div}u|^2\right)
+\frac{\alpha}{2}\frac{d}{dt}\int_{\partial\Omega}\xi^2 K(x)u\cdot u
+\int_{\Omega}\xi u_t^2\nonumber\\
	=&\alpha\int_{\partial\Omega}\xi\xi_t K(x)u\cdot u
     +a\int_{\Omega}\xi\xi_t
	\left(\alpha|\mbox{curl}u|^2+(2\alpha+\beta)|\mbox{div}u|^2\right)\nonumber\\
	&-a\int_{\Omega}\left(\alpha(\nabla\xi^2\times u_t)\cdot \mbox{curl}u
	+(2\alpha+\beta)u_t\cdot \nabla\xi^2 \mbox{div}u\right)\nonumber\\
	&-\int_{\Omega}\xi v\cdot\nabla u\cdot u_t-\int_{\Omega}\xi \nabla \phi \cdot u_t
+\int_{\Omega}\xi \psi \cdot Q(u)\cdot u_t
	\equiv \sum_{j=1}^{6}R_j.
\end{align}

In view of (\ref{est-boundary}), the boundary integral can be estimated as follows:
\begin{align}\label{step2-R1}
	R_1&\leq C\|\xi_t\|_{L^{\infty}}\int_{\partial\Omega}\xi u^2\leq C\|\xi_t\|_{L^{\infty}} \|\nabla(\sqrt{\xi}u)\|^2_{L^{2}}\nonumber\\
&\leq C\left(\|\xi_t\|_{L^\infty}\|\xi^{-1}\|_{L^\infty}\|\xi\nabla u\|_{L^2}^2+\|\xi_t\|_{L^\infty}\|\nabla\sqrt{\xi}\|^2_{L^\infty}\|u\|_{L^2}^2\right).
\end{align}
By applying Cauchy's inequality, we obtain
\begin{align}\label{step2-R2}
	R_2&\leq C \|\xi_t\|_{L^\infty}\|\xi^{-1}\|_{L^\infty}\|\xi\nabla u\|_{L^2}^2
	,\nonumber\\
	R_3&\leq C \int_{\Omega}\xi|\nabla\xi| |u_t||\nabla u|
	\leq \frac{1}{16}\|\sqrt{\xi}u_t\|^2_{L^2}+C\|\nabla\xi\|^2_{L^\infty}\|\xi^{-1}\|_{L^\infty}
	\|\xi\nabla u\|^2_{L^2},\nonumber\\
	R_4&\leq \frac{1}{16}\|\sqrt{\xi}u_t\|^2_{L^2}+C\|v\|^2_{L^\infty}\|\xi^{-1}\|_{L^\infty}\|\xi\nabla u\|^2_{L^2}
,\nonumber\\
	R_6&\leq \frac{1}{16}\|\sqrt{\xi}u_t\|^2_{L^2}
+C\|\psi\|^2_{L^\infty}\|\xi^{-1}\|_{L^\infty}\|\xi\nabla u\|^2_{L^2}.
\end{align}
For term $R_5$,  integrating by parts  yields
\begin{align}\label{step2-R5}
	R_5=&\int_{\Omega}\mbox{div}(\xi  u_t) (\phi-\phi^{\infty})
	=\int_{\Omega}\xi\mbox{div}u_t (\phi-\phi^{\infty})+\int_{\Omega}u_t\cdot\nabla\xi (\phi-\phi^{\infty})\nonumber\\
	\leq& \frac{d}{dt}\int_{\Omega}\xi\mbox{div}u (\phi-\phi^{\infty})+  \frac{1}{16}\|\sqrt{\xi}u_t\|^2_{L^2}+C\|\xi\nabla u\|^2_{L^2}\nonumber\\
+& C\left(\|\xi^{-1}\|^2_{L^\infty}\|\xi_t\|^2_{L^\infty}\|\phi-\phi^{\infty}\|^2_{L^2}+\|\phi_t\|^2_{L^2}
	+ \|\xi^{-1}\|_{L^\infty}\|\nabla \xi\|^2_{L^\infty} \|\phi-\phi^{\infty}\|^2_{L^2}\right).
\end{align}
Putting (\ref{step2-R1})-(\ref{step2-R5}) into (\ref{u-0}), we have
\begin{align}\label{step2-2}
	 &\frac{a}{2}\frac{d}{dt}\int_{\Omega}\xi^2\left(\alpha|\mbox{curl}u|^2+(2\alpha+\beta)|\mbox{div}u|^2\right)
+\frac{\alpha}{2}\frac{d}{dt}\int_{\partial\Omega}\xi^2 K(x)u\cdot u
+\frac{1}{2}\|\sqrt{\xi} u_t\|^2_{L^2}
	\nonumber\\
&\leq \frac{d}{dt}\int_{\Omega}\xi\mbox{div}u (\phi-\phi^{\infty})+C\|\xi_t\|_{L^\infty}\|\nabla\sqrt{\xi}\|^2_{L^\infty}\|u\|_{L^2}^2
\nonumber\\
&+C\left(1+\|\xi_t\|_{L^\infty}\|\xi^{-1}\|_{L^\infty}
+\|v\|^2_{L^\infty}\|\xi^{-1}\|_{L^\infty}
+\|\psi\|^2_{L^\infty}\|\xi^{-1}\|_{L^\infty}\right)\|\xi\nabla u\|^2_{L^2}\nonumber\\
&+C\left(\|\xi^{-1}\|^2_{L^\infty}\|\xi_t\|^2_{L^\infty}\|\phi-\phi^{\infty}\|^2_{L^2}+\|\phi_t\|^2_{L^2}
	+ \|\xi^{-1}\|_{L^\infty}\|\nabla \xi\|^2_{L^\infty} \|\phi-\phi^{\infty}\|^2_{L^2}\right).
\end{align}
Integrating (\ref{step2-2}) over $s\in[0,t]$, for $0\leq t\leq T_2$, we get
\begin{align}\label{u-1}
	&\int_{\Omega}\xi^2 \left(\alpha|\mbox{curl}u|^2+(2\alpha+\beta)|\mbox{div}u|^2\right)
+\int_{\partial\Omega}\xi^2 K(x)u\cdot u
	+\int_0^t\|\sqrt{\xi} u_t\|^2_{L^2}ds\nonumber\\
	&\leq C\|\xi\nabla u\|_{L^2}^2(0)+C\int_{\partial\Omega}\xi_0^2 K(x)u_0\cdot u_0+\int_{\Omega}\xi\mbox{div}u (\phi-\phi^{\infty})-\int_{\Omega}\xi_0\mbox{div}u_0 (\phi_0-\phi^{\infty})
\nonumber\\
&+Cc_0c_4^2\int_0^t \|\xi\nabla u\|^2_{L^2}ds + Cc_0^4c_3^3c_4t.
\end{align}
By using (\ref{est-boundary}), we have
 \begin{align*}
 \int_{\partial\Omega}\xi_0^2 K(x)u_0\cdot u_0\leq C\|\nabla(\xi_0 u_0)\|^2_{L^2}
 \leq C\|\xi_0 \nabla u_0\|^2_{L^2}
 +\|\nabla\xi_0\|^2_{L^{\infty}}\|u_0\|^2_{L^2}
 \leq Cc_0^4,
 \end{align*}
where we have used  the initial compatibility condition $\|\xi_0\nabla u_0\|_{L^2}=\|g_1\|_{L^2}\leq c_0$.
By using the Cauchy's inequality, one has
 \begin{align*}
 \int_{\Omega}\xi\mbox{div}u (\phi-\phi^{\infty})\leq \frac{2\alpha+\beta}{2} \|\xi\mbox{div}u\|^2_{L^2}+C\|\phi-\phi^{\infty}\|^2_{L^2}.
 \end{align*}
Moreover, note that
\begin{align*}
	\int_{\Omega}\xi^2 |\nabla u|^2\leq C\int_{\Omega}\xi^2 \left(|\mbox{curl}u|^2+|\mbox{div}u|^2\right)+C\int_{\Omega}
	|\nabla\xi|^2|u|^2,
\end{align*}
then, it is easy to see that (\ref{u-1}) yields
 \begin{align*}
	\|\xi\nabla u\|^2_{L^2}+\int_0^t\|\sqrt{\xi} u_t\|^2_{L^2}ds
	\leq Cc_0^4+Cc_0c_4^2\int_0^t \|\xi\nabla u\|^2_{L^2}ds + Cc_0^4c_3^3c_4t,
\end{align*}
 which, along with the Gronwall's inequality, implies that for $0\leq t\leq T_2$,
\begin{align}\label{9}
		\|\xi\nabla u(t)\|_{L^2}^2+\int_0^t\|\sqrt{\xi} u_t\|_{L^2}^2ds	
\leq C(c_0^4+c_0^4c_3^3c_4t)\exp(Cc_0c_4^2t)
\leq C c_0^4,
\end{align}
and \begin{align}
\|\nabla u(t)\|_{L^2}^2\leq C c_0^6,\quad
\int_0^t\| u_t\|_{L^2}^2ds	
\leq C c_0^5.
\end{align}
Furthermore, from the equation $(\ref{linear})_2$,  we have
\begin{align}
 aL(\xi u)=-u_{t}-v\cdot\nabla u-\nabla\phi+ \psi\cdot Q(u)
 - G(\nabla\xi, u),
 \end{align}
with the boundary conditions:
\begin{align*}
\xi u\cdot n=0,\quad \mbox{curl} (\xi u)\times n=(\nabla\xi \cdot n)u-K(x)\xi u\quad \mbox{on}\quad \partial\Omega.
\end{align*}
Applying the elliptic estimates as in Lemma \ref{Appendix},  we deduce that for $0\leq t\leq T_2$
\begin{align}\label{10}
		\|\nabla^{2}(\xi u)\|^2_{L^2}&\leq C\|u_{t}\|_{L^2}^2+C\|v\cdot\nabla u+\nabla\phi- \psi\cdot Q(u)
 + G(\nabla\xi, u)\|^2_{L^2}\nonumber\\
 &+C(\|\nabla^2\xi\|^2_{H^1}+\|\nabla\xi\|^2_{L^\infty})\|u\|^2_{H^{1}}+C\| \xi\nabla u\|_{L^2}^2\nonumber\\
 &\leq C\|u_{t}\|_{L^2}^2+C c_0^6c_2c_3.
\end{align}
Then it holds that
\begin{align}
		\int_0^t\|\nabla^{2}(\xi u)\|^2_{L^2}ds&\leq Cc_0^5,  \quad \int_0^t\|\xi\nabla^{2}u\|^2_{L^2}ds\leq Cc_0^5,
\quad \int_0^t\|\nabla^{2}u\|^2_{L^2}ds\leq Cc_0^7,
\end{align}
for $0\leq t\leq T_2$.

{\bf Step 3.} To estimate $\|\xi\nabla^2 u\|_{L^2}$.  Taking the derivative of (\ref{equ-u-linear}) with respect to $t$, and multiplying the result by $u_{t}$,  integrating by parts with the boundary conditions
$u_t\cdot n|_{\partial\Omega}=0$ and $\mbox{curl}u_t\times n=-K(x)u_t$, it results in
\begin{align}\label{step3-2}
&\frac{1}{2}\frac{d}{dt}\|u_t\|^2_{L^2}+a\int_{\Omega}\xi
\left(\alpha|\mbox{curl}u_t|^2+(2\alpha+\beta)|\mbox{div}u_t|^2\right)
+a\alpha\int_{\partial\Omega}\xi K(x)u_t\cdot u_t\nonumber\\
&=-a\int_{\Omega}\alpha\mbox{curl}u_t\cdot(\nabla\xi\times u_t)+(2\alpha+\beta)\mbox{div}u_t(u_t\cdot\nabla\xi) \nonumber\\
&-a\int_{\Omega}\xi_t Lu\cdot u_t
-\int_{\Omega}(v\cdot \nabla u)_t\cdot u_t -\int_{\Omega}\nabla \phi_t\cdot u_t+\int_{\Omega}(\psi\cdot Q(u))_t\cdot u_t\equiv \sum_{j=1}^5 J_{j}.
\end{align}
Now the terms on the right-hand side of (\ref{step3-2}) can be estimated as follows:
\begin{align}
J_1&\leq \varepsilon_2 \|\sqrt{\xi}\nabla u_t\|_{L^2}^2+C(\varepsilon_2)\|\xi^{-\frac
{1}{2}}\|^2_{L^\infty}\|\nabla\xi\|^2_{L^\infty}\| u_t\|_{L^2}^2,\nonumber\\
J_2&\leq \|\xi_t\|_{L^{\infty}}\|\nabla^2u\|_{L^2}\|u_t\|_{L^2}\leq \|\xi_t\|^2_{L^{\infty}}\|u_t\|^2_{L^2}+C\|\nabla^2u\|^2_{L^2},\nonumber\\
J_3 &\leq
 \varepsilon_2\|\sqrt{\xi}\nabla u_t\|^2_{L^2}+C(\varepsilon_2)(\|\xi^{-\frac{1}{2}}\|^2_{L^{\infty}}\|v\|^2_{L^\infty}
 +\|v_t\|^2_{L^3})\|u_t\|^2_{L^2}+C\|\nabla u\|^2_{L^6},\nonumber\\
J_4&\leq \|u_t\|^2_{L^2}+\|\nabla \phi_t\|^2_{L^2},\nonumber\\
J_5&\leq \varepsilon_2\|\sqrt{\xi}\nabla u_t\|^2_{L^2}+C(\varepsilon_2)
(\|\xi^{-\frac{1}{2}}\|^2_{L^\infty}\|\psi\|^2_{L^\infty})\|u_t\|^2_{L^2}+C\|\psi_t\|^2_{L^3}\|u_t\|^2_{L^2}
+C\|\nabla u\|^2_{L^6}.
\end{align}
Putting above estimates  into (\ref{step3-2}), and  using
\begin{align*}
	\|\sqrt{\xi} \nabla u_t\|^2_{L^2}
\leq  C \left(\|\sqrt{\xi}\mbox{curl}u_t\|^2_{L^2}+\|\sqrt{\xi}\mbox{div}u_t\|^2_{L^2}
+\|\xi^{-1}\|_{L^{\infty}}\|\psi\|^2_{L^{\infty}}\|u_t\|^2_{L^2}
\right),
\end{align*}
choosing $\varepsilon_2$ small,  we have
\begin{align}\label{step3-3}
&\frac{1}{2}\frac{d}{dt}\|u_t\|^2_{L^2}+\|\sqrt{\xi}\nabla u_t\|^2_{L^2}
+a\alpha\int_{\partial\Omega}\xi Ku_t\cdot u_t
\nonumber\\
&\leq C\left(1+\|\xi^{-\frac{1}{2}}\|^2_{L^\infty}\|v\|^2_{L^\infty}+\|v_t\|^2_{L^3}+\|\xi_t\|^2_{L^{\infty}}
+\|\xi^{-\frac{1}{2}}\|^2_{L^\infty}\|\psi\|^2_{L^\infty}+\|\psi_t\|^2_{L^3}\right)
\|u_t\|^2_{L^2}\nonumber\\
&+C\left(\|\nabla \phi_t\|^2_{L^2}+\|\nabla^2u\|^2_{L^2}\right).
\end{align}
Integrating (\ref{step3-3}) over $(\tau, t)$ ($\tau\in(0, t)$) and using Young's inequality, one has for $0<t\leq T_2$
\begin{align}\label{step3-4}
\|u_t(t)\|^2_{L^2}+\int_\tau^t\|\sqrt{\xi}\nabla u_t\|^2_{L^2}ds
\leq \|u_t(\tau)\|^2_{L^2}+Cc_4^4\int_0^t\|u_t\|^2_{L^2}ds+C(c_0^2c_3^2t+c_0^7).
\end{align}
It follows from the momentum equations that
\begin{align*}
\|u_t(\tau)\|_{L^2}\leq \| \left(-v\cdot\nabla u-\nabla\phi -a\xi Lu+\psi\cdot Q(u)\right)(\tau)\|_{L^2},
\end{align*}
which, along with the assumption (\ref{c0}) implies that
\begin{align*}
 \limsup_{\tau\rightarrow0}\|u_t(\tau)\|_{L^2}
 &\leq C \left(\|u_0\|_{L^{\infty}}\|\nabla u_0\|_{L^2}+\|\nabla\phi_0\|_{L^2} +\|g_2\|_{L^2}+\|\psi_0\|_{L^{\infty}}\| \nabla u_0\|_{L^2}\right)\nonumber\\
 &\leq C c_0^{2}.
\end{align*}
Letting $\tau\rightarrow 0$ in (\ref{step3-4}), we get from the Gronwall's inequality that for $0\leq t\leq T_2$
\begin{align}\label{Step3-5}
&\|u_t(t)\|^2_{L^2}+\int_0^t\|\sqrt{\xi}\nabla u_t\|^2_{L^2}ds
\leq C\left(c_0^2c_3^2t+c_0^7\right)\exp(Cc_4^4t)
\leq Cc_0^7,\nonumber\\
&\int_0^t\|\nabla u_t\|^2_{L^2}ds\leq  Cc_0^8.
\end{align}
Recalling the elliptic estimate as in (\ref{10}), we obtain that
\begin{align*}
\|\nabla^2(\xi u)\|^2_{L^2}\leq Cc_0^6c_2c_3,
\end{align*}
and
\begin{align}\label{Step3-0}
\|\xi\nabla^2u\|^2_{L^2}&\leq C \left(\|\nabla^2(\xi u)\|^2_{L^2}+\|\nabla\xi\nabla u\|^2_{L^2}
+\|\nabla^2\xi u\|^2_{L^2}\right)\nonumber\\
&\leq Cc_0^6c_2c_3,\quad \mbox{for}\ 0\leq t\leq T_2.
\end{align}
Furthermore, employing the  elliptic estimate  as in Lemma \ref{Appendix} once more, we can find that
\begin{align}\label{ell-u3}
		\|\xi u\|^2_{D^{3}}\leq &C\|u_t\|^2_{H^{1}}+C\|-v\cdot\nabla u-\nabla\phi+ \psi\cdot Q(u)
 - G(\nabla\xi, u)\|^2_{H^{1}}\nonumber\\[2mm]
		&+C(\|\nabla^2\xi\|^2_{H^1}+\|\nabla\xi\|^2_{L^\infty})\|u\|^2_{H^{2}}+C\| \xi\nabla u\|_{L^2}^2\nonumber\\[2mm]
&\leq C\|\nabla u_t\|^2_{L^{2}}+Cc_0^8c_3^4.
		\end{align}
Therefore,  it holds for $0\leq t\leq T_2$,
\begin{align*}
\int_0^t(\|\xi \nabla^3u\|^2_{L^{2}}+\|\xi \nabla^2u\|^2_{D^{1}}+\|\xi u\|^2_{D^{3}})ds\leq Cc_0^8.
\end{align*}

{\bf {Step 4.}}  Estimate on $\|\nabla^3u\|_{L^2}$.\\
Taking $\int_{\Omega}\partial_{t}(\ref{equ-u-linear})\cdot \xi u_{tt}$, we have
\begin{align}\label{Step4-1}
&\int_{\Omega}\xi u_{tt}^2+a\int_{\Omega}\xi^2 Lu_t
\cdot u_{tt}\nonumber\\
&=-a\int_{\Omega}\xi\xi_t Lu\cdot u_{tt}
-\int_{\Omega}\xi (v\cdot\nabla u)_t\cdot u_{tt}-\int_{\Omega}\xi \nabla \phi_t \cdot u_{tt}+
\int_{\Omega}\xi (\psi\cdot Q(u))_{t}\cdot u_{tt}.
 \end{align}
Using integration by parts, the second term in left-hand side of (\ref{Step4-1}) becomes
\begin{align*}
&\int_{\Omega}\xi^2
Lu_t\cdot u_{tt}
=\int_{\Omega}\xi^2
\left(\alpha\mbox{curl}^2u_t-(2\alpha+\beta)\nabla\mbox{div}u_t\right)\cdot u_{tt}\nonumber\\
&=\frac{1}{2}\frac{d}{dt}\int_{\Omega}\xi^2
\left(\alpha|\mbox{curl}u_t|^2+(2\alpha+\beta)|\mbox{div}u_t|^2\right)
-\int_{\Omega}\xi\xi_t
\left(\alpha|\mbox{curl}u_t|^2+(2\alpha+\beta)|\mbox{div}u_t|^2\right)\nonumber\\
&+\int_{\Omega}\left(\alpha(\nabla\xi^2\times u_{tt})\cdot \mbox{curl}u_t
+(2\alpha+\beta)u_{tt}\cdot \nabla\xi^2 \mbox{div}u_t\right)
-\alpha\int_{\partial\Omega}\xi^2 (\mbox{curl}u_t\times n)\cdot u_{tt}.
\end{align*}
The boundary term on above can be handle as follows,
\begin{align*}
-\int_{\partial\Omega}\xi^2 (\mbox{curl}u_t\times n)\cdot u_{tt}
=\frac{1}{2}\frac{d}{dt}\int_{\partial\Omega}\xi^2 K(x)u_t\cdot u_t- \int_{\partial\Omega}\xi\xi_t K(x)u_t\cdot u_t.
\end{align*}
Putting all the above into (\ref{Step4-1}), we get
\begin{align}\label{Step4-2}
&\frac{a}{2}\frac{d}{dt}\int_{\Omega}\xi^2
\left(\alpha|\mbox{curl}u_t|^2+(2\alpha+\beta)|\mbox{div}u_t|^2\right)
+\frac{a\alpha}{2}\frac{d}{dt}\int_{\partial\Omega}\xi^2 Ku_t\cdot u_t
+\int_{\Omega}\xi u_{tt}^2\nonumber\\
&=a\alpha\int_{\partial\Omega}\xi\xi_t K(x)u_t\cdot u_t
+a\int_{\Omega}\xi\xi_t
\left(\alpha|\mbox{curl}u_t|^2+(2\alpha+\beta)|\mbox{div}u_t|^2\right)\nonumber\\
&-a\int_{\Omega}\left(\alpha(\nabla\xi^2\times u_{tt})\cdot \mbox{curl}u_t
+(2\alpha+\beta)u_{tt}\cdot \nabla\xi^2 \mbox{div}u_t\right)\nonumber\\
&-a\int_{\Omega}\xi\xi_t Lu\cdot u_{tt}
-\int_{\Omega}\xi (v\cdot\nabla u)_t\cdot u_{tt}-\int_{\Omega}\xi u_{tt}\cdot \nabla\phi_t\nonumber\\
&+\int_{\Omega}\xi (\psi\cdot Q(u))_{t}\cdot u_{tt}
 \equiv \sum_{j=1}^7 I_j.
\end{align}
We proceed to estimate each term individually.
For the boundary integral, in view of (\ref{est-boundary}), we have
\begin{align*}
I_1
&\leq C\|\xi_t\|_{L^{\infty}}\int_{\partial\Omega}\xi u_t\cdot u_t
\leq C\|\xi_t\|_{L^{\infty}} \|\nabla(\sqrt{\xi}u_t)\|^2_{L^{2}}\nonumber\\
&\leq C\left(\|\xi_t\|_{L^\infty}\|\xi^{-1}\|_{L^\infty}\|\xi\nabla u_t\|_{L^2}^2+\|\xi_t\|_{L^\infty}\|\nabla\sqrt{\xi}\|^2_{L^\infty}\|u_t\|_{L^2}^2\right).
\end{align*}
Using the Young's inequality and the Sobolev inequality, it holds that
\begin{align*}
I_2&\leq C \|\xi_t\|_{L^\infty}\|\xi^{-1}\|_{L^\infty}\|\xi\nabla u_t\|^2_{L^2},\nonumber\\
I_3&\leq C \int_{\Omega}\xi|\nabla\xi| |u_{tt}||\nabla u_t|
\leq \frac{1}{16}\|\sqrt{\xi}u_{tt}\|_{L^2}^2+C\|\nabla\xi\|^2_{L^\infty}
\|\xi^{-1}\|_{L^\infty}\|\xi\nabla u_t\|_{L^2}^2,\nonumber\\
I_4&\leq \frac{1}{16}\|\sqrt{\xi}u_{tt}\|_{L^2}^2+C\|\xi_t\|^2_{L^\infty}\|\sqrt{\xi}\nabla^2 u\|_{L^2}^2,\nonumber\\
I_5&\leq \frac{1}{16}\|\sqrt{\xi}u_{tt}\|_{L^2}^2+ C\left(\| v_t\|^2_{L^3}\|\sqrt{\xi} \nabla u\|_{L^6}^2+\|\xi^{-1}\|_{L^{\infty}}\|v\|^2_{L^\infty}\|\xi\nabla u_t\|_{L^2}^2\right),
\end{align*}
and
\begin{align*}
I_7
\leq\frac{1}{16}\|\sqrt{\xi}u_{tt}\|_{L^2}^2+C\left(\|\psi_t\|_{L^3}^2\|\sqrt{\xi}\nabla u\|_{L^6}^2+
\|\xi^{-\frac{1}{2}}\|^2_{L^{\infty}}\|\psi\|_{L^\infty}^2\|{\xi}\nabla u_t\|_{L^2}^2\right).
\end{align*}
For the term $I_6$, we should hand carefully due to the weight function $\xi$. We proceed with the detailed estimate as follows,
\begin{align}\label{Step4-R_6}
I_6&\equiv -\int_{\Omega}\xi u_{tt}\cdot\nabla \phi_t  =\int_{\Omega}\phi_t \mbox{div}(\xi u_{tt})=\int_{\Omega}\xi\phi_t \mbox{div}u_{tt}+\int_{\Omega}\phi_t u_{tt}\cdot\nabla\xi
\nonumber\\
&=\frac{d}{dt}\int_{\Omega}\xi\phi_t \mbox{div}u_{t}-\int_{\Omega}\xi_t\phi_t \mbox{div}u_{t}-\int_{\Omega}\xi\phi_{tt} \mbox{div}u_{t}+\int_{\Omega}\phi_t u_{tt}\cdot\nabla\xi\nonumber\\
&\leq \frac{d}{dt}\int_{\Omega}\xi\phi_t \mbox{div}u_{t}
+\frac{1}{16}\|\sqrt{\xi}u_{tt}\|_{L^2}
+C\left(\|\phi_t\|^2_{L^2}+\|\phi_{tt}\|^2_{L^2}\right)\nonumber\\
&\quad +C\left(\|\xi^{-1}\|^2_{L^{\infty}}\|\xi_t\|^2_{L^{\infty}}\|\xi\nabla u_{t}\|^2_{L^2}
 +\|\xi\nabla u_{t}\|^2_{L^2}
+\|\nabla\xi\|^2_{L^{\infty}}\|\xi^{-\frac{1}{2}}\|^2_{L^{\infty}}\|\phi_t\|^2_{L^2}
\right).
\end{align}
Putting  the terms $I_i$ ($i=1,2,\cdots 7$) into (\ref{Step4-2}), we have
\begin{align}\label{Step4-3}
&\frac{a}{2}\frac{d}{dt}\int_{\Omega}\xi^2
\left(\alpha|\mbox{curl}u_t|^2+(2\alpha+\beta)|\mbox{div}u_t|^2\right)
+\frac{a\alpha}{2}\frac{d}{dt}\int_{\partial\Omega}\xi^2 Ku_t\cdot u_t
+\frac{1}{2}\int_{\Omega}\xi u_{tt}^2\nonumber\\
&\leq \frac{d}{dt}\int_{\Omega}\xi\phi_t \mbox{div}u_{t}+ C\|\xi_t\|_{L^\infty}\|\nabla\sqrt{\xi}\|^2_{L^\infty}\|u_t\|_{L^2}^2\nonumber\\
&+C\left(1+\|\xi_t\|_{L^\infty}\|\xi^{-1}\|_{L^\infty}+\|\nabla\xi\|^2_{L^\infty}
\|\xi^{-1}\|_{L^\infty}+\|\xi^{-1}\|_{L^{\infty}}\|v\|^2_{L^\infty}\right)\|\xi\nabla u_t\|_{L^2}^2
\nonumber\\
&+C\left(\|\xi^{-1}\|^2_{L^{\infty}}\|\xi_t\|^2_{L^{\infty}}\|\xi\nabla u_t\|_{L^2}^2+\|\xi_t\|^2_{L^\infty}\|\sqrt{\xi}\nabla^2 u\|_{L^2}^2\right)\nonumber\\
&+C\left(
\| v_t\|^2_{L^3}\|\sqrt{\xi} \nabla u\|_{L^6}^2+\|\psi_t\|_{L^3}^2\|\sqrt{\xi}\nabla u\|_{L^6}^2\right)\nonumber\\
&+C\left(\|\phi_t\|^2_{L^2}
+\|\phi_{tt}\|^2_{L^2}
+\|\nabla\xi\|^2_{L^{\infty}}\|\xi^{-\frac{1}{2}}\|^2_{L^{\infty}}\|\phi_t\|^2_{L^2}\right).
\end{align}
Integrating (\ref{Step4-3}) over $(\tau, t)$ ($\tau\in (0,t)$),  it holds that for $0< t\leq T_2$,
 \begin{align}\label{Step4-4}
&\|\xi\nabla u_t(t)\|_{L^2}^2+\int_{\partial\Omega}\xi^2 Ku_t\cdot u_t+\int_{\tau}^t\|\sqrt{\xi} u_{tt}\|_{L^2}^2ds\nonumber\\
&\leq C\|\xi\nabla u_t(\tau)\|_{L^2}^2+C\int_{\partial\Omega}\xi^2 Ku_t\cdot u_t(\tau)
+Cc_0^2c_4^4\int_0^t\|\xi\nabla u_t\|_{L^2}^2ds+
Cc_0^7c_4^6t+Cc_0^7c_2^2,
\end{align}
where we have used
\begin{align*}
\int_{\Omega}\xi\phi_t \mbox{div}u_{t}
\leq \varepsilon\|\xi\nabla u_t\|_{L^2}^2+C\|\phi_t\|_{L^2}^2\leq \varepsilon\|\xi\nabla u_t\|_{L^2}^2+Cc_0^2c_2^2,
\end{align*} for any small constant $\varepsilon>0$.

 From the momentum equations, we deduce that
\begin{align*}
 \|\xi\nabla u_t(\tau)\|_{L^2}\leq \|\xi\nabla \left(-v\cdot\nabla u-\nabla\phi -a\xi Lu+\psi\cdot Q(u)\right)(\tau)\|_{L^2}.
\end{align*}
Then by the assumption (\ref{c0}), (\ref{key-phi}) and (\ref{key-2u}), one has
\begin{align*}
 &\limsup_{\tau\rightarrow0}\|\xi\nabla u_t(\tau)\|_{L^2}\nonumber\\
 &\leq C \left(\|\xi_0\nabla(u_0\cdot\nabla u_0)\|_{L^2}+\|\xi_0\nabla^2\phi_0\|_{L^2}+\|g_3\|_{L^2}
 +\|\xi_0\nabla(\psi_0\cdot Q(u_0))\|_{L^2}\right)\nonumber\\
 &\leq C c_0^{3},
\end{align*}
and
 \begin{align*}
 \limsup_{\tau\rightarrow0}\int_{\partial\Omega}\xi^2 K(x)u_t\cdot u_t(\tau)\leq C\|\nabla(\xi u_t)(0)\|^2_{L^2}
 \leq Cc_0^6.
 \end{align*}
Letting $\tau\rightarrow 0$ in (\ref{Step4-4}) and using the Gronwall's inequality, we  obtain
\begin{align*}
&\|\xi\nabla u_t(t)\|_{L^2}^2+\int_{0}^t\|\sqrt{\xi} u_{tt}\|_{L^2}^2ds
\leq C\left(c_0^6+c_0^7c_2^2+c_0^7c_4^6t\right)\exp(c_0^2c_4^4t)
\leq C c_0^{7}c_2^2,
\end{align*}
and
\begin{align*}
\|\nabla u_t(t)\|_{L^2}^2\leq Cc_0^{9}c_2^2, \quad
\int_{0}^t\| u_{tt}\|_{L^2}^2ds\leq Cc_0^{8}c_2^2,
\end{align*}
for $0\leq t\leq T_3=\min\{T^{*}, (1+c_4)^{-6}\}$.

Recalling the estimate (\ref{ell-u3}), we get
 \begin{align*}
		\|\xi u\|^2_{D^{3}}\leq C c_0^8c_3^4,\quad \|\xi \nabla^3u\|^2_{L^{2}}\leq C c_0^8c_3^4,\quad \mbox{for}\
0\leq t\leq T_2.
		\end{align*}
Moreover, $u_t$ solves the following equations
\begin{align*}
 aL(\xi u_t)
 =-u_{tt}-a\xi_tLu-(v\cdot\nabla u)_{t}-\nabla\phi_t+ (\psi\cdot Q(u))_{t}
 - G(\nabla\xi, u_t)
 \end{align*}
 with the boundary conditions
 \begin{align*}
\xi u_t\cdot n=0,\quad \mbox{curl}(\xi u_t)\times n=-\xi K(u)u_t+(\nabla\xi\cdot n)u_t\quad \mbox{on}\  \partial\Omega.
 \end{align*}
 Applying elliptic estimates as in Lemma \ref{Appendix} to $\xi u_t$, we have
\begin{align*}
		\|\nabla^2(\xi u_t)\|^2_{L^2}&\leq C\|u_{tt}\|^2_{L^2}+C\|\xi_tLu+(v\cdot\nabla u)_{t}+\nabla\phi_t-(\psi\cdot Q(u))_{t}
 + G(\nabla\xi, u_t)\|^2_{L^2}
\nonumber\\
		&+C(\|\nabla^2\xi\|^2_{H^1}+\|\nabla\xi\|^2_{L^\infty})\|u_t\|^2_{1}+C\| \xi\nabla u_t\|_{L^2}^2
\nonumber\\
		&\leq C\|u_{tt}\|^2_{L^2}+Cc_0^8c_3^2c_4^4,
\end{align*}
which yields
\begin{align*}
		&\int_0^t\|\nabla^2(\xi u_t)\|^2_{L^2}ds\leq C\int_0^t\|u_{tt}\|^2_{L^2}ds+Cc_0^8c_3^2c_4^4t
		\leq Cc_0^{8}c_2^2,\\
		&\int_0^t\|(\xi\nabla^2 u)_t\|^2_{L^2}ds\leq Cc_0^{8}c_2^2,
\end{align*}
for $\ 0\leq t\leq T_3.$
This completes the proof of Lemma \ref{lem-u}.
\end{proof}

The following lemma provides a high-order derivative estimate $\int_0^t\|\xi \nabla^2 u\|^2_{D^2}ds$. For initial-boundary value problems, obtaining such weighted  estimates presents significantly greater challenges compared to Cauchy problems. To overcome the difficulty arising from the lack of boundary conditions for $D^2u$,
we adopt an approach based on conormal spaces. Rather than estimating quantity $\|\xi \nabla^2 u\|^2_{D^2}$ directly, we instead develop estimates for  $\|\xi \nabla u\|^2_{D^3}$.
\begin{lem}\label{lem-u-2}
	Let $(\phi,\xi,u)$ be the unique strong solution to (\ref{linear}) on $[0,T]\times \Omega$. Then for $0\leq t\leq T_3$, it holds
\begin{align}\label{step5}
\int_0^t\|\xi \nabla^2 u\|^2_{D^2}ds\leq  Cc_0^{12}c_2^2.
\end{align}
\end{lem}
\begin{proof}
In order to avoid complications due to the geometry of the domain, we shall first give the proof of Lemma \ref{lem-u-2} in the case where $\Omega$ is the half space.

(i) The case of a half-space: $\Omega=\mathbb{R}^3_+$.
Set
 \begin{align*}
 w=\mbox{curl}u, \quad  w_{h}=(w_1,w_2)=\left(\partial_{2}u_3-\partial_3u_2, \partial_{3}u_1-\partial_1u_3\right).
 \end{align*}
 Denoting $\partial$ as the tangential derivative, and
noting that
\begin{align*}
\partial_{3}u_1=w_2+\partial_1u_3,\quad   \partial_{3}u_2=-w_1+\partial_2u_3,
\end{align*} we have
 \begin{align}\label{Du4}
 \|\xi \nabla u\|_{D^3}
 \leq C\left(\|\xi \partial u\|_{D^3} +\|\xi\mbox{div}u\|_{D^3}+\|\xi w_{h}\|_{D^3}\right).
 \end{align}
 Next, we will estimate $\int_0^t\|\xi \partial u\|^2_{D^3}ds$, $\int_0^t\|\xi\mbox{div}u\|^2_{D^3}ds$  and $\int_0^t\|\xi w_{h}\|^2_{D^3}ds$, respectively.
 \vspace{2mm}

{\bf{Step 1.}} \ The estimate of $\int_0^t\|\xi \partial u\|^2_{D^3}ds$.  Taking the tangent derivative $\partial$ to (\ref{equ-u-linear}), we have
 \begin{align*}
 	aL(\xi \partial u)
 =&-\partial u_{t}-a\partial\xi Lu- \partial(v\cdot\nabla u)-\nabla \partial\phi+  \partial(\psi\cdot Q(u))
 	- aG(\nabla\xi,  \partial u)\nonumber\\
 	\equiv& - \partial u_{t}+\widetilde{N}_1,
 \end{align*}
 with the boundary conditions:
 \begin{align*}
 	\xi  \partial u \cdot n=0,\quad
 \mbox{curl}(\xi  \partial u)\times n=(\nabla\xi\cdot n) \partial u-\xi K(x)\partial u
 \quad \mbox{on}\  \partial\Omega.
 \end{align*}
Applying the elliptic estimate as in Lemma \ref{Appendix}  and using (\ref{nabla^2u_t}), we get
\begin{align}\label{half-1}
\int_0^t\|\nabla^2(\xi\partial u)\|^2_{H^{1}}ds&\leq C\int_0^t\big(\|\partial u_{t}\|^2_{H^{1}}+\|\widetilde{N}_1\|^2_{H^{1}}+
\|\nabla(\xi\partial  u)\|_{L^2}^2\big)ds\nonumber\\
&+C\int_0^t(\|\nabla^2\xi\|^2_{H^1}+\|\nabla\xi\|^2_{L^\infty})\|\partial u\|^2_{H^{2}}ds\nonumber\\
&\leq C c_0^{10}c_2^2, \quad \mbox{for} \quad 0\leq t\leq T_3.
	\end{align}

{\bf{Step 2.}}\ To estimate $\int_0^t\|\xi\mbox{div}u\|^2_{D^3}ds$.

Taking the operator $\mbox{curl}$ to the equation (\ref{equ-u-linear}), then  $w$ satisfies the following system
 \begin{align}\label{equ-w}
 w_t-a\alpha\xi\Delta w=-a\nabla \xi \times Lu-\mbox{curl}(v\cdot\nabla u)
 +\mbox{curl}(\psi\cdot Q (u)).
 \end{align}
Applying Lemma \ref{lem-div-curl}  with $v = \xi\mbox{curl}^2 u$, we have
\begin{align}\label{curl^4u}
\|\xi\mbox{curl}^2 u\|^2_{D^2}
&\leq C\|\xi\mbox{curl}^3u\|^2_{H^1}
+C\left(\|\nabla\xi\|^2_{L^\infty}\|\nabla^2u\|^2_{1}
+\|\nabla^2\xi\|^2_{L^3}\|\nabla^2u\|^2_{L^6}
\right)\nonumber\\
&+C\|\xi\mbox{curl}^2u\|^2_{L^2}+C|\xi\mbox{curl}^2u\cdot n|_{H^{\frac{3}{2}}(\partial\Omega)}.
\end{align}
Let $v^{\perp}=-v\times n$, then $v=v^{\perp}\times n$ on $\partial\Omega$  provided  $v\cdot n|_{\partial\Omega}=0$.
Using the boundary condition $\xi\mbox{curl}u\times n=-\xi K u_{\tau}$ on $\partial\Omega$, it holds
\begin{align*}
\mbox{curl}\left(\xi\mbox{curl}u+\xi(K u)^{\perp}\right)\cdot n|_{\partial\Omega}=0.
\end{align*}
Therefore,
\begin{align*}
\xi \mbox{curl}^2u\cdot n=-\nabla\xi\times \mbox{curl}u\cdot n-\mbox{curl}\left(\xi(Ku)^{\perp}\right)\cdot n \quad \mbox{on}\quad  \partial\Omega.
\end{align*}
So, the boundary integral term on the right-hand side of (\ref{curl^4u}) can be estimated as follows,
\begin{align}\label{step5-1}
|\xi\mbox{curl}^2u\cdot n|^2_{H^{\frac{3}{2}}(\partial\Omega)}
&\leq C\left(|\nabla\xi\times \mbox{curl}u\cdot n|^2_{H^{\frac{3}{2}}(\partial\Omega)}
+|\mbox{curl}\left(\xi(Ku)^{\perp}\right)\cdot n|^2_{H^{\frac{3}{2}}(\partial\Omega)}\right)\nonumber\\
&\leq C \left(\|\nabla\xi\cdot\nabla u\|^2_{{2}}+\|\nabla(\xi u)\|^2_{{2}}\right)\nonumber\\
&\leq Cc_0^{12}c_3^4, \quad \mbox{for}\ 0\leq t\leq T_3.
\end{align}
To control the first term on the right-hand side in (\ref{curl^4u}),  we  use the equation of $w$ (\ref{equ-w}) directly to deduce that
\begin{align}\label{step5-2}
\|\xi\mbox{curl}^3u\|^2_{H^1}&=\|\xi\Delta w\|^2_{H^1}\nonumber\\
&\leq \|w_t\|^2_{H^1}+\|a\nabla \xi \times Lu-\mbox{curl}(v\cdot\nabla u)
 +\mbox{curl}(\psi\cdot Q (u))\|^2_{H^1}\nonumber\\
  &\leq C\|\nabla^2 u_t\|^2_{L^2}+Cc_0^{10}c_3^6,\quad \mbox{for}\ 0\leq t\leq T_3.
\end{align}
Putting (\ref{step5-1}), (\ref{step5-2}) into (\ref{curl^4u}), and using (\ref{nabla^2u_t}), we obtain
\begin{align}\label{step5-3}
\int_0^t \|\xi\mbox{curl}^2 u\|_{D^2}^2ds&\leq C\int_{0}^t\|\nabla^2 u_t\|^2_{L^2}ds+Cc_0^{10}c_3^6t
 \leq Cc_0^{10}c_2^2,\nonumber\\
\int_0^t\|\xi\nabla^2\mbox{curl}^2 u\|_{L^2}^2ds
&\leq C c_0^{10}c_2^2, \quad
\int_0^t\|\nabla^2\mbox{curl}^2 u\|_{L^2}^2ds
\leq C c_0^{12}c_2^2,
\end{align}
for $0\leq t\leq T_3$.

Furthermore, taking $\partial_{x_i}\partial_{x_j}$  to (\ref{equ-u-linear}), multiplying the resulting identities by $\xi\partial_{x_i}\partial_{x_j}\nabla \mbox{div}u$, $i =
1, 2, 3, j = 1, 2, 3$, summing them up, and integrating over $\Omega$, we obtain
\begin{align*}
&a(2\alpha+\beta)\|\xi\nabla^3\mbox{div}u\|^2_{L^2}\\[2mm]
=&\int_{\Omega}\left(\nabla^2u_t
+a\alpha\nabla ^2\mbox{curl}^2 u
+a\nabla^2\xi\cdot Lu
+2a\nabla\xi\cdot \nabla Lu
+\nabla^2(v\cdot \nabla u+\nabla\phi-\psi\cdot Q(u))\right)\cdot\xi\nabla^3\mbox{div}u.
\end{align*}
Hence, by using the Young's inequality, we get that
\begin{align*}
\|\xi\nabla^3\mbox{div}u\|^2_{L^2}
\leq
C(\|\nabla^2u_t\|_2^2+\|\nabla^2\mbox{curl}^2u\|^2_{L^2}+
\|\nabla^2 \xi\|_1^2\|\nabla^2 u\|^2_1+\|v\|^2_{2}\|\nabla u\|^2_{2}+\|\nabla^3\phi\|_{L^2}^2).
\end{align*}
Integrating $t$ and using (\ref{nabla^2u_t}) and (\ref{step5-3}), we finally get
\begin{align*}
\int_0^t\|\xi \nabla^3\mbox{div}u\|^2_{L^2} ds
\leq Cc_0^{12}c_2^2, \quad \mbox{for}\ 0\leq t\leq T_3.
\end{align*}
Therefore, we obtain the estimate
\begin{align}\label{half-2}
\int_0^t\|\xi \mbox{div} u\|^2_{D^3}ds
\leq Cc_0^{12}c_2^2, \quad \mbox{for}\ 0\leq t\leq T_3.
\end{align}

{\bf{Step 3.}}\ At last, we estimate $\int_0^t\|\xi w_{h}\|^2_{D^3}ds$.
When $\Omega=\mathbb{R}^3_{+}$, $S(n) =0$, then the Navier-slip boundary condition (\ref{boundary2-2}) reduces to
 \begin{align}
 u_3=0,  \quad  w_1=-2\vartheta u_2,\quad  w_2=2\vartheta u_1,\quad   \mbox{on}\ \partial\Omega.
 \end{align}
This leads us to introduce the unknown
\begin{align*}
\eta=w_{h}-2\vartheta u_{h}^{\perp}.
\end{align*}
Then
 \begin{align*}
\eta=0\quad \mbox{on}\ \partial\Omega,
\end{align*}
and
\begin{align}\label{step5-4}
\int_0^t\|\xi w_{h}\|^2_{D^3}ds\leq C\int_0^t\|\xi \eta\|^2_{D^3}ds+C\int_0^t\|\xi u_{h}\|^2_{D^3}ds.
\end{align}
Consequently, we shall only  estimate $\int_0^t\|\xi \eta\|^2_{D^3}ds$  in the following.
Rewriting the equations of $u$ and $w$ as
 \begin{align*}
 &(u_h)_t-a \alpha\xi \Delta u_h=a(\alpha+\beta)\xi \nabla_{h} \mbox{div}u-v\cdot \nabla u_{h}-\nabla_{h}\phi+(\psi\cdot Q(u))_{h}
 \equiv(\Upsilon_{u})_{h},\\
&  (w_h)_t-a \alpha\xi \Delta w_h=\left(-a\nabla\xi\times Lu-\mbox{curl}(v\cdot u)+\mbox{curl}(\psi\cdot Q(u))\right)_{h}
 \equiv(\Upsilon_{w})_{h}.
 \end{align*}
From the definition of $\eta$, we find that it solves the equation
 \begin{align*}
 \eta_t-a \alpha\xi \Delta\eta=(\Upsilon_{w})_{h}-2\vartheta(\Upsilon_{u})_h^{\perp},
 \end{align*}
that is,
\begin{align*}
 a \alpha\Delta(\xi \eta)=\eta_t-(\Upsilon_{w})_{h}+2\vartheta(\Upsilon_{u})_h^{\perp}+a\alpha \Delta\xi \eta+2a\alpha \nabla\xi\cdot\nabla\eta\equiv \eta_t+\Upsilon,
 \end{align*}
  with boundary condition $\xi \eta|_{\partial\Omega}=0$. Applying the elliptic estimates as in Lemma \ref{lem6-0-1}, it holds
  \begin{align*}
 \int_0^t\|\xi \eta\|^2_{D^3}ds\leq& C\int_0^t(\|\eta_t\|^2_{H^1}+\|\Upsilon\|^2_{H^1}+\|\nabla(\xi \eta)\|^2_{L^2})ds\nonumber\\
 &\leq Cc_0^{10}c_2^2,\quad \mbox{for}\quad   0\leq t\leq T_2.
 \end{align*}
 Then, recalling (\ref{step5-4}), we get
\begin{align}\label{half-3}
\int_0^t\|\xi w_{h}\|^2_{D^3}ds\leq C c_0^{10}c_2^2, \quad \mbox{for}\  0\leq t\leq T_3.
\end{align}
Finally,  combining the estimates from (\ref{half-1}), (\ref{half-2}),(\ref{half-3}), and utilizing (\ref{Du4}), we obtain
\begin{align*}
\int_0^t\|\xi \nabla u\|^2_{D^3}ds\leq  Cc_0^{12}c_2^2,
\quad \mbox{for}\  0\leq t\leq T_3,
\end{align*}
which immediately yields the desired estimate (\ref{step5}) in the case of half space.
\vspace{2mm}

(ii)\ For general exterior domain $\Omega$, one assumes that the domain $\Omega$ has
a covering such that
\begin{align*}
\Omega\subset \Omega_0\cup_{k=1}^{m}\Omega_{k},
\end{align*}
where $\bar{\Omega}_0\subset\Omega$, and in each $\Omega_{k}$
there exists a function $\varrho_k$ such that
\begin{equation*}
\Omega\cap \Omega_k=\{(x_1,x_2,x_3)|x_3<\varrho_{k}(x_1,x_2)\}\cap \Omega_{k},\quad
\partial\Omega\cap \Omega_k=\{x_3=\varrho_{k}(x_1,x_2)\}\cap \Omega_{k}.
\end{equation*}
By using that $\partial\Omega$ is given locally by $x_3 = \varrho(x_1, x_2)$ (we omit the subscript $k$
for notational convenience), it is convenient to use the coordinates:
\begin{align*}
\Psi:(y,z)\mapsto (y,\varrho(y)+z)=x.
\end{align*}
Denote by $n$ the unit outward normal in the physical space which is
given locally by
\begin{align*}
n(x)\equiv n(\Psi(y,z))=\frac{1}{\sqrt{1+|\nabla \varrho(y)|^2}}
\left(
\partial_1\varrho(y),
\partial_2\varrho(y),
-1\right)^{T}
\end{align*}
and by $\Pi$, the orthogonal projection
\begin{align}\label{orth}
 \Pi u\equiv\Pi (\Psi(y,z))u=u-[u\cdot n(\Psi(y,z))]n(\Psi(y,z)),
 \end{align}
 which gives the orthogonal projection onto the tangent space of the boundary.

 Note that
 \begin{align}\label{gen-1}
 \mbox{div}u=\partial_{n}u\cdot n+(\Pi \partial_{y^1}u)_1+(\Pi \partial_{y^2}u)_2,
 \end{align}
 \begin{align}\label{gen-2}
 \partial_{n}u=[\partial_n u\cdot n]n+\Pi (\partial_{n}u).
 \end{align}
Since
 \begin{align}\label{com}
 \xi\chi \nabla u= \xi\chi\Pi(\nabla u)+ \xi\chi [\partial_{n}u\cdot n]n+ \xi\chi\Pi(\partial_nu),
 \end{align}
 in view of (\ref{gen-1}), we find that
  \begin{align}\label{tan}
  \int_0^t\|\xi\chi\Pi(\nabla u)+ \xi\chi [\partial_{n}u\cdot n]n\|^2_{D^3}ds&\leq
  \int_0^t\|\xi\chi\Pi(\nabla u)\|^2_{D^3}ds+\int_0^t\| \xi\chi \mbox{div}u\|^2_{D^3}ds\nonumber\\
  &\leq C  c_0^{12}c_2^2, \quad \mbox{for}\  0\leq t\leq T_3,
  \end{align}
  which can be estimated in the similar way as the Step 1  and Step 2 in the case of half space.
Therefore, to obtain a complete estimate of $\int_0^t\|\xi\chi\nabla u\|^2_{D^3}ds$, it remains to estimate $\int_0^t\|\xi\chi\Pi(\partial_{n}u)\|^2_{D^3}ds$. Referring the idea of the conormal space (one can refer to \cite{M2012}, \cite{WXY2015}),   we define
\begin{align*}
\eta= \chi(w\times n+\Pi(Ku))=\chi(\Pi(w\times n)+\Pi(K u)).
\end{align*}
Then, in view of the Navier-slip condition (\ref{boundary2-2}),
$\eta$ satisfies
\begin{align*}
\eta=0,\quad \mbox{on}\ \partial\Omega.
\end{align*}
In view of  $w\times n=(\nabla u-(\nabla u)^{T})\cdot n$, then $\eta$ can be rewritten as
\begin{align*}
\eta=\chi\left(\Pi(\partial_{n}u)-\Pi(\nabla(u\cdot n))+\Pi((\nabla n)^{T}\cdot u)+\Pi(Ku)\right),
\end{align*}
which together with (\ref{tan}) and the estimations of lower-order derivatives   in Lemma \ref{lem-u} yields that
\begin{align*}
\int_0^t\|\xi\chi\Pi(\partial_{n}u)\|^2_{D^3}ds\leq \int_0^t\|\xi\eta\|^2_{D^3}ds+
C  c_0^{12}c_2^2, \quad \mbox{for}\  0\leq t\leq T_3.
\end{align*}
Hence, it remains to estimate $\int_0^t\|\xi\eta\|^2_{D^3}ds$.  Recalling the equations  of $u$ and $w$ as follows,
\begin{align*}
 &u_t-a\alpha\xi\Delta u=a(\alpha+\beta)\xi \nabla\mbox{div}u-v\cdot \nabla u-\nabla\phi+\psi\cdot Q(u)
 \equiv \mathcal{F}_1,\\
 &w_t-a\alpha\xi\Delta w=-a\nabla \xi \times Lu-\mbox{curl}(v\cdot\nabla u)
 +\mbox{curl}(\psi\cdot Q (u))
 \equiv \mathcal{F}_2,
 \end{align*}
then, it is easy to see that
$\eta$ solves the equations
\begin{align*}
\eta_t-a\alpha \xi\Delta\eta=\chi (\mathcal{F}_2\times n+\Pi(K\mathcal{F}_1))+a\alpha \xi \mathcal{F}_3-a\alpha \xi\chi \Delta(\Pi K) u,
\end{align*}
where
\begin{align*}
\mathcal{F}_3&=-\chi w\times \Delta n-2\chi \sum_{j=1}^2\partial_jw\times \partial_jn
-2\chi \sum_{j=1}^2\partial_j(\Pi K)\partial_ju\\
&-\Delta\chi (w\times n+\Pi(Ku))-2\sum_{j=1}^2\partial_j\chi\partial_j(w\times n+\Pi(Ku)).
\end{align*}
Furthermore, we have
\begin{align*}
a\alpha \Delta( \xi\eta)
=\eta_t+\mathcal{F},
\end{align*}
with  the boundary condition $\xi\eta|_{\partial\Omega}=0$,
 and
 \begin{align*}
\mathcal{F}=
-\chi (\mathcal{F}_2\times n+\Pi(K\mathcal{F}_1))-a\alpha \xi \mathcal{F}_3+a\alpha \xi\chi \Delta(\Pi K) u
+a\alpha\Delta\xi \eta
+2a\alpha \nabla\xi\cdot\nabla\eta,
\end{align*}
 Thus,
  applying the elliptic estimates as in Lemma \ref{lem6-0-1}, it holds
\begin{align*}
\int_0^t\|\xi\eta\|^2_{D^3}ds&\leq C\int_0^t \left(\|\eta_t\|^2_{H^1}+\|\mathcal{F}\|^2_{H^1}+\|\nabla(\xi\eta)\|^2_{L^2}\right)ds\nonumber\\
&\leq Cc_0^{10}c_2^2, \quad \mbox{for}\  0\leq t\leq T_3.
\end{align*}
Hence, the proof of Lemma \ref{lem-u-2} is complete.
\end{proof}

Then combining the estimates obtained in Lemmas \ref{lem-phi}-\ref{lem-u-2},
defining the time
\begin{equation*}
T^{*}=\min\{T, (1+C^7c_0^{57})^{-6}\}
\end{equation*}
and constants
\begin{equation}\label{c_i}
c_1=C^{\frac{1}{2}}c_0,\quad c_2=C^{\frac{1}{2}}c_0^{\frac{7}{2}},\quad c_3=C^{\frac{3}{2}}c_0^{\frac{23}{2}},
\quad c_4=C^7c_0^{57},
\end{equation}
we can obtain
\begin{align}\label{local-linear}
&\|\phi(t)-\phi^\infty\|_3^2+\|\phi_t(t)\|_2^2+\|\phi_{tt}(t)\|_{L^2}^2\leq c_4^7,\nonumber\\
&\|\psi(t)\|^2_{D^1\cap D^2}\leq c_1^2, \quad
\|\psi_t(t)\|^2_{1}+\int_0^t\|\psi_{tt}(s)\|^2_{L^2}ds\leq c_4^5,\nonumber\\
& \|\nabla \xi_t(t)\|^2_{L^2}+\|\xi_t(t)\|^2_{L^\infty}\leq c_4^2,
\nonumber\\
&\|\varphi(t)\|^2_{L^\infty\cap D^{1,6}\cap D^{2,3}\cap D^3}+\|f(t)\|^2_{L^\infty\cap L^6\cap D^{1,3}\cap D^2}\leq c^2_2,\nonumber\\
&\xi>\frac{1}{2c_0},\quad \frac{2}{3}\sigma^{-2\kappa}<\varphi,\quad
\|\varphi_t(t)\|^2_{ L^6\cap D^{1,3}\cap D^2}+\|f_t(t)\|^2_{L^3\cap D^1}\leq c_4^{11},\nonumber\\
&\|\xi\nabla u(t)\|_{L^2}^2+\|u(t)\|^2_1+\int_0^t\left(\|\nabla u\|^2_1+\|u_t\|^2_{L^2}\right)ds\leq c_2^2,\nonumber\\
&\|u_t(t)\|_{L^2}^2+\|\xi\nabla^2u\|_{L^2}^2
+\|\nabla^2u(t)\|_{L^2}^2
+\int_0^t\left(\|\nabla^3u\|_{L^2}^2+\|\xi\nabla^2 u\|_{D^1}^2
+\|\nabla u_t\|_{L^2}^2\right)ds\leq c_3^2,\nonumber\\
&\|\xi\nabla u_t(t)\|_{L^2}^2+\|\nabla u_t(t)\|_{L^2}^2+\|\nabla^3u(t)\|_{L^2}^2
+\|\xi\nabla^2 u(t)\|_{D^1}^2+\int_0^t\|u_{tt}\|^2_{L^2}ds\leq c_4^2,\nonumber\\
&\int_0^t\left(\|\nabla^2 u_t\|_{L^2}^2+\|\nabla^4u\|_{L^2}^2+\|\xi\nabla^2 u\|_{D^2}^2+\|(\xi\nabla^2 u)_t\|_{L^2}^2\right)ds\leq c_4^2.
\end{align}

In conclusion, given fixed constants $c_0$ and $T$, there really exists constants $T^{*}$ and $c_i$, $(i=1,2,3,4)$, which depend only on $c_0$ and $T$, such that if $v$ and $g$ satisfy the conditions (\ref{known}),
then (\ref{local-linear}) holds for the strong solution to (\ref{linear}) on $[0,T^*]\times\Omega$.
In the proof  of Theorem \ref{thm2}, these estimates play an important role in the iteration process.

\subsection{Construction of the nonlinear approximation solutions away from vacuum}\label{sub3.4}

In this subsection, based on the assumption that $\phi_0>\sigma>0$, we will prove the local-in-time well-posedness of the strong solution to the following initial-boundary value problem:
\begin{equation}\label{non-1}
\left\{
\begin{array}{llll}
\phi_t+u\cdot\nabla\phi+(\gamma-1)\phi\mbox{div}u=0,\\
u_{t}+u\cdot\nabla u+\nabla\phi +a\phi^{2\kappa}Lu=\psi\cdot Q(u),\\
\psi_t+\nabla(u\cdot\psi)+(\delta-1)\psi\mbox{div}u+\delta a \phi^{2\kappa}\nabla \mbox{div}u=0,\\
 (\phi,\psi, u)|_{t=0}=\left(\phi_0,\frac{a\delta}{\delta-1}\nabla\phi_0^{2\kappa},u_0\right),\\
  (\phi,\psi, u)\rightarrow(\phi^{\infty},0,0),\quad \mbox{as} \quad |x|\rightarrow \infty,\quad t\geq0,\\
 u\cdot n=0,\quad \mbox{curl}u\times n= -K(x)u,\quad \mbox{on}\quad \partial\Omega,
 \end{array}
\right.
\end{equation}
whose life span is independent of $\sigma$.

\begin{thm}\label{thm-nonlinear}
Let (\ref{constant}) hold and $\phi^{\infty}$ be a positive constant. Assume that the initial data $(\phi_0,\xi_0=(\phi_0)^{2\kappa},\psi_0=\frac{a\delta}{\delta-1}\nabla(\phi_0)^{2\kappa}, u_0)$ satisfies the hypothesis of Lemma \ref{global}, and there exists a positive constant $c_0$ independent of $\sigma$ such that (\ref{c0}) holds. Then there exist a time $T_{*}>0$ and a unique strong solution
$(\phi,\xi=\phi^{2\kappa},\psi=\frac{a\delta}{\delta-1}\nabla\xi, u)$
in $[0, T_{*}]\times\Omega$ to the initial-boundary value problem  (\ref{non-1}),
where $T_{*}$ is independent of $\sigma$. Moreover,
the estimates (\ref{local-linear}) hold for $(\phi,\xi,\psi, u)$ with $T^{*}$ replaced by $T_{*}$, and are independent of $\sigma$.
\end{thm}

The proof is given by an iteration scheme based on the estimates for the linearized problem obtained in Sections \ref{sub3.2}-\ref{sub3.3}. As in Section \ref{sub3.3}, we define constants $c_i$ ($i=1,2,3,4$).

\begin{proof}{\it\textbf{Step 1:} Existence.}
Let $(u^0,\xi^0)$ be the solution to the following  problem:
\begin{equation*}
\left\{
\begin{array}{lll}
X_t-Y\Delta X =0\quad \mbox{in}\quad \Omega\times(0,\infty),\\
Y_{t}+u_0\cdot \nabla Y=0\quad \mbox{in}\quad \Omega\times(0,\infty),\\
(X,Y)|_{t=0}=(u_0,\phi_0^{2\kappa})\quad \mbox{in}\quad \Omega,\\
X\cdot n=0,\quad \mbox{curl}X\times n=-K(x)X\quad \mbox{on}\quad \partial\Omega,\\
(X,Y)\rightarrow (0,(\phi^{\infty})^{2\kappa}) \quad\mbox{as}\quad |x|\rightarrow \infty, \quad t\geq 0.
\end{array}
\right.
\end{equation*}
Choose a time $T^{**}\in (0,T^{*}]$ small enough such that
\begin{align*}
&\|\nabla \xi^{0}(t)\|^2_{D^1\cap D^2}\leq c_1^2,\nonumber\\
&\|u^{0}(t)\|^2_1+\int_0^{T^{**}}\left(\|\nabla^2 u^{0}\|_{L^2}^2+\|u^{0}_t\|_{L^2}^2\right)ds\leq c_2^2,\nonumber\\
&\|\nabla^2u^{0}(t)\|_{L^2}^2+\|u^{0}_t(t)\|_{L^2}^2+\|\xi^{0}\nabla^2u^{0}(t)\|_{L^2}^2
+\int_0^{T^{**}}\left(\|\nabla^3u^{0}\|_{L^2}^2+\|\nabla u^{0}_t\|_{L^2}^2\right)ds\leq c_3^2,\nonumber\\
&\|\nabla^3u^{0}(t)\|_{L^2}^2+\|\nabla u^{0}_t(t)\|_{L^2}^2
+\int_0^{T^{**}}\left(\|\nabla^4u^{0}\|_{L^2}^2+\|\nabla^2u^{0}_t\|_{L^2}^2
+\|u^{0}_{tt}\|_{L^2}^2\right)ds
\leq c_4^2,\nonumber\\
&\|\xi^{0}\nabla^2u^{0}(t)\|_{D^1}^2
+\int_0^{T^{**}}\left(\|(\xi^{0}\nabla^2u^{0})_t\|_{L^2}^2
+\|\xi^{0}\nabla^2u^{0}\|_{D^2}^2\right)ds
\leq c_4^2,\nonumber\\
&\|\nabla \xi^{0}_t\|_{L^2}^2+\|\xi^{0}_t(t)\|_{L^\infty}^2\leq c_4^2.
\end{align*}
At the beginning step of iteration, pick  $(v,g)=(u^0,\xi^0)$, then
 $(\phi^1, \xi^1, u^1)$ could be obtained as a strong solution to the linearized problem \eqref{linear}. Then we construct approximate solutions $(\phi^{k+1},\xi^{k+1}, u^{k+1})$ inductively. In other words, given $(u^k,\xi^k)$ for $k\geq 1$,
then $(\phi^{k+1},\xi^{k+1},  u^{k+1})$ should be defined by solving the following problem:
\begin{equation}\label{inter-0}
\begin{split}
\left\{
\begin{array}{lll}
\phi^{k+1}_t+u^{k}\cdot\nabla\phi^{k+1}+(\gamma-1)\phi^{k+1}\mbox{div}u^{k}=0,\\
\xi_t^{k+1}+u^k\cdot \nabla\xi^{k+1}+(\delta-1)\xi^k\mbox{div}u^k=0,\\
u_t^{k+1}+u^k\cdot\nabla u^{k+1}+\nabla\phi^{k+1}+a\xi^{k+1}Lu^{k+1}=\psi^{k+1}\cdot Q(u^{k+1}),\\
(\phi^{k+1},\xi^{k+1},u^{k+1})|_{t=0}=(\phi_0,\phi_0^{2\kappa},u_0),\\
u^{k+1}\cdot n|_{\partial\Omega}=0,\quad \mbox{curl}u^{k+1}\times n|_{\partial\Omega}=-K(x)u^{k+1},\\
(\phi^{k+1},\xi^{k+1},u^{k+1})\rightarrow (\phi^{\infty},(\phi^{\infty})^{2\kappa},0)\quad \mbox{as}\quad |x|\rightarrow \infty, \  t\geq 0,
\end{array}
\right.
\end{split}
\end{equation}
where
$\psi^{k+1}=\frac{a\delta}{\delta-1}\nabla\xi^{k+1}.$
The problem (\ref{inter-0}) can be solved from (\ref{linear}) by replacing $(v, g)$ with $(u^k, \xi^k)$, and
$(\phi^{k},\xi^{k},u^{k})$ meet the uniform a priori estimates \eqref{local-linear}.

	Let
\begin{align*}
\varphi^{k+1}=\frac{1}{\xi^{k+1}},\quad f^{k+1}=\frac{\psi^{k+1}}{\xi^{k+1}}=\frac{a\delta}{\delta-1}\frac{\nabla\xi^{k+1}}{\xi^{k+1}}.
	\end{align*}
Then  problem (\ref{inter-0}) can be rewritten as
\begin{equation}\label{inter-1}
\begin{split}
\left\{
\begin{array}{lll}
\phi^{k+1}_t+u^{k}\cdot\nabla\phi^{k+1}+(\gamma-1)\phi^{k+1}\mbox{div}u^{k}=0,\\
\varphi^{k+1}_t+u^{k}\cdot\nabla\varphi^{k+1}-(\delta-1)(\varphi^k)^{-1}(\varphi^{k+1})^2\mbox{div}u^{k}=0,\\
f^{k+1}_t+\sum_{l=1}^3A_l(u^k)\partial_lf^{k+1}+B(u^k)f^{k+1}+a\delta(\varphi^k)^{-1}\varphi^{k+1}\nabla\mbox{div}u^k
\\
	\quad =-a\delta\varphi^{k+1}\nabla \xi^k\mbox{div}u^k+(\delta-1)(\varphi^k)^{-1}\varphi^{k+1}f^{k+1}\mbox{div}u^{k},\\
\varphi^{k+1}(u_t^{k+1}+u^k\cdot\nabla u^{k+1}+\nabla\phi^{k+1})+aLu^{k+1}=f^{k+1}\cdot Q(u^{k+1}),\\
(\phi^{k+1},\varphi^{k+1},f^{k+1},u^{k+1})|_{t=0}=(\phi_0,\phi_0^{-2\kappa},\phi_0^{-2\kappa}\psi_0,u_0),\\
u^{k+1}\cdot n|_{\partial\Omega}=0,\quad \mbox{curl}u^{k+1}\times n|_{\partial\Omega}=-K(x)u^{k+1},\\
(\phi^{k+1},\varphi^{k+1},f^{k+1},u^{k+1})\rightarrow (\phi^{\infty},(\phi^{\infty})^{-2\kappa},0,0),\quad \mbox{as}\quad |x|\rightarrow \infty, \  t\geq 0,
\end{array}
\right.
\end{split}
\end{equation}

\vspace{4mm}
{\it{Step 1.1: Strong convergence of $(\phi^{k},\varphi^{k},f^{k},u^{k})$.}}

Now we will prove the full sequence $(\phi^{k},\varphi^{k},f^{k},u^{k})$ converges  to a limit $(\phi,\varphi,f,u)$
in some strong sense.
Set
\begin{align*}
\overline{\phi}^{k+1}=\phi^{k+1}-\phi^{k},\quad \overline{f}^{k+1} =f^{k+1}-f^{k},\quad\overline{\varphi}^{k+1}=\varphi^{k+1}-\varphi^{k}, \quad
  \overline{u}^{k+1}=u^{k+1}-u^{k}.
\end{align*}
Then it follows from \eqref{inter-1} that
\begin{equation}\label{inter1}
	\begin{split}
		\left\{
		\begin{array}{lll}
			 \overline{\phi}^{k+1}_t+u^{k}\cdot\nabla\overline{\phi}^{k+1}+\overline{u}^{k}\cdot\nabla \phi^{k}
			 +(\gamma-1)\left(\overline{\phi}^{k+1}\mbox{div}u^{k}+\phi^{k}\mbox{div}\overline{u}^{k}\right)=0,\\
\overline{\varphi}_t^{k+1}+u^k\cdot\nabla\overline{\varphi}^{k+1}
+\overline{u}^k\cdot\nabla \varphi^{k}+(1-\delta)(\overline{\varphi}^{k}\mbox{div}u^k
			+\varphi^{k-1}\mbox{div}\overline{u}^k+\Upsilon_1^k)=0,\\
\overline{f}^{k+1}_t+\sum_{l=1}^3A_l(u^k)\partial_l	\overline{f}^{k+1}+B(u^k)	 \overline{f}^{k+1}+a\delta\nabla\mbox{div}\overline{u}^k
		=\Upsilon_2^k+\Upsilon_3^k+\Upsilon_4^k,\\
				\varphi^{k+1}(\overline{u}_t^{k+1}+u^k\cdot\nabla \overline{u}^{k+1}+\overline{u}^k\cdot\nabla u^k)+\varphi^{k+1}\nabla\overline{\phi}^{k+1}+\overline{\varphi}^{k+1}\nabla\phi^k+aL\overline{u}^{k+1}\\
				\quad =-\overline{\varphi}^{k+1}(u_t^k+u^{k-1}\cdot \nabla u^k)+f^{k+1}\cdot Q(\overline{u}^{k+1})+ \overline{f}^{k+1}\cdot Q(u^k),
			\end{array}
		\right.
	\end{split}
\end{equation}
 with the following initial and boundary conditions:
\begin{align*}	 &(\overline{\phi}^{k+1},\overline{f}^{k+1},\overline{\varphi}^{k+1},\overline{u}^{k+1})|_{t=0}=(0,0,0,0),\\
	&\overline{u}^{k+1}\cdot n|_{\partial\Omega}=0,\quad \mbox{curl}\overline{u}^{k+1}\times n|_{\partial\Omega}=-K(x)\overline{u}^{k+1},
\end{align*}
where $\Upsilon_i^k(i=1,\cdots,4)$ are introduced as
\begin{align*}
\Upsilon_1^k &= \overline{\varphi}^{k+1} \varphi^{k+1} \xi^k \text{div} u^k - \overline{\varphi}^k \varphi^k \xi^{k-1} \text{div} u^{k-1} + \overline{\varphi}^{k+1} \text{div} u^k - \overline{\varphi}^k \text{div} u^{k-1},\\
\Upsilon_2^k &= (1 - \delta) \left( \overline{f}^k \text{div} u^k + f^{k-1} \text{div} \overline{u}^k + \overline{\varphi}^{k+1} f^k \xi^k \text{div} u^k - \overline{\varphi}^k f^{k-1} \xi^{k-1} \text{div} u^{k-1} \right), \\
\Upsilon_3^k &= -\sum_{l=1}^3 (A_l (u^k) - A_l (u^{k-1})) \partial_l f^k - (B (u^k) - B (u^{k-1})) f^k \\
&\quad - a \delta (\overline{\varphi}^{k+1} \xi^k \nabla \text{div} u^k - \overline{\varphi}^k \xi^{k-1} \nabla \text{div} u^{k-1}), \\
\Upsilon_4^k &= (\delta - 1) \left( f^{k+1} \text{div} u^k \overline{\varphi}^{k+1} \xi^k - f^k \text{div} u^{k-1} \overline{\varphi}^k \xi^{k-1} + \overline{f}^{k+1} \text{div} u^k + f^k \text{div} \overline{u}^k \right).
\end{align*}
For convenience, we introduce some notations in the rest of the proof of this subsection,
\begin{align*}
	R^k(t)&=\|\nabla u^k\|_{L^\infty}+\|\varphi^{k+1}\|_{L^\infty}\|\xi^{k}\nabla u^k\|_{L^\infty},
\quad S^k(t)=\| u_t^k\|_{L^3}+\|u^{k-1}\|_{L^\infty}\|\nabla u^k\|_{L^3},\\
	M^k(t)&=\left(\|f^k\|_{L^\infty}+\|f^{k-1}\|_{L^\infty}\right)\|\xi^{k-1}\nabla u^{k-1}\|_{L^\infty}.
	\end{align*}

 {\it \underline{{Estimate on $\|\overline{\phi}^{k+1}\|_{1}$.}}}\\

Multiplying $(\ref{inter1})_1$ by $\overline{\phi}^{k+1}$, integrating the result over $\Omega$, and applying boundary condition $u^k\cdot n|_{\partial\Omega}=0$ to have the following identity
\begin{equation*}
	-\int_{\Omega} u^k\cdot\nabla\overline{\phi}^{k+1}\overline{\phi}^{k+1}=\frac{1}{2}\int_{\Omega}\mbox{div}u^k|\overline{\phi}^{k+1}|^2,
\end{equation*}
then it implies that
\begin{align}\label{phi-n-1}
	\frac{1}{2}\frac{d}{dt}\|\overline{\phi}^{k+1}\|_{L^2}^2
\leq  C\left(\|\nabla u^{k}\|_{L^\infty}\|\overline{\phi}^{k+1}\|_{L^2}
+\|\overline{u}^{k}\|_{L^6}\|\nabla\phi^{k}\|_{L^3} +\|\nabla\overline{u}^{k}\|_{L^2}\|\phi^{k}\|_{L^\infty}\right)\|\overline{\phi}^{k+1}\|_{L^2}.
	\end{align}
Similarly, computing $\int_{\Omega}\partial_{x_i}(\ref{inter1})_1 \cdot\partial_{x_i}\overline{\phi}^{k+1}$, $i=1,2,3$, it gives
\begin{align}\label{phi-n-2}
	\frac{1}{2}\frac{d}{dt}\|\nabla\overline{\phi}^{k+1}\|_{L^2}^2
	&\leq \left(\|\nabla u^{k}\|_{L^\infty}\|\nabla\overline{\phi}^{k+1}\|_{L^2}
+\|\nabla\overline{u}^{k}\|_{L^2}\|\nabla\phi^{k}\|_{L^{\infty}}
	 +\|\overline{u}^{k}\|_{L^6}\|\nabla^2\phi^k\|_{L^3}\right)\|\nabla\overline{\phi}^{k+1}\|_{L^2}\nonumber\\
	&+\left(\|\overline{\phi}^{k+1}\|_{L^6}\|\nabla^2u^{k}\|_{L^3}
	 +\|\nabla^2\overline{u}^{k}\|_{L^2}\|\phi^{k}\|_{L^\infty}\right)
\|\nabla\overline{\phi}^{k+1}\|_{L^2}.
	\end{align}
Thus, the combination of \eqref{phi-n-1} and \eqref{phi-n-2} results in
\begin{align}\label{step1}
	\frac{d}{dt}\|\overline{\phi}^{k+1}\|_1^2
	\leq  E^k_1(t)\|\overline{\phi}^{k+1}\|_1^2+\epsilon\|\nabla\overline{u}^{k}\|_1^2,
\end{align}
with
\begin{align*}
	\int_0^tE_1^{k}(s)ds\leq C+C\epsilon^{-1}t,
\end{align*}
thanks to the a priori estimates \eqref{local-linear}. Here $\epsilon>0$ is a small constant to be determined.
\vspace{1mm}

 {\underline{\it{Estimate on $\|\overline{\varphi}^{k+1}\|_1$.}}} \\

    Multiplying $(\ref{inter1})_3$ by $\overline{\varphi}^{k+1}$, and integrating over $\Omega$ result in
\begin{align}\label{step3-0}
\frac{d}{dt}\|\overline{\varphi}^{k+1}\|_{L^2}^2
	\leq& CR^k(t)\|\overline{\varphi}^{k+1}\|_{L^2}^2+\|\overline{u}^k\|_{L^6}\|\nabla\varphi^k\|_{L^3}
\|\overline{\varphi}^{k+1}\|_{L^2}\nonumber\\[2mm]
	&+C(\left\|\nabla \overline{u}^k\|_{L^2}\|\varphi^{k-1}\|_{L^\infty}
+\|\overline{\varphi}^k\|_{L^2}\|\nabla u^{k}\|_{L^\infty} \right)\|\overline{\varphi}^{k+1}\|_{L^2}\nonumber\\[2mm]
&+CR^{k-1}(t)\|\overline{\varphi}^{k}\|_{L^2}\|\overline{\varphi}^{k+1}\|_{L^2},
\end{align}
where we have used $u^k\cdot n|_{\partial\Omega}=0$ to have
\begin{equation*}
	-\int_{\Omega} u_j^k\partial_{x_j}\overline{\varphi}^{k+1}\overline{\varphi}^{k+1}=\frac{1}{2}\int_{\Omega}\mbox{div} u^{k}|\overline{\varphi}^{k+1}|^2 \leq \|\nabla u^{k}\|_{L^{\infty}}\|\overline{\varphi}^{k+1}\|_{L^2}^2.
\end{equation*}
Next applying $\partial_x^l(|l|=1)$ to $(\ref{inter1})_3$, producting with $\partial_x^l\overline{\varphi}^{k+1}$,  integrating over $\Omega$, we have
\begin{align}
	&\frac{d}{dt}\|\partial_x^l\overline{\varphi}^{k+1}\|_{L^2}^2\nonumber\\
	\leq& C\left(R^k(t)+\|\xi^k\nabla u^k\|_{L^6}\|\nabla \varphi^{k+1}\|_{L^6}
+\|\varphi^{k+1}\|_{L^\infty}\|\xi^k\nabla^2 u^k\|_{L^3}
+\|\nabla^2u^k\|_{L^3}\right)\|\nabla\overline{\varphi}^{k+1}\|_{L^2}^2\nonumber\\
	 &+C\left(\|\nabla\varphi^k\|_{L^6}\|\nabla\overline{u}^k\|_{L^3}+\|\overline{\varphi}^{k+1}\|_{L^6}
\|\varphi^{k+1}\|_{L^\infty}\|\psi^{k}\|_{L^\infty}
\|\nabla u^{k}\|_{L^3}\right)\|\nabla\overline{\varphi}^{k+1}\|_{L^2}\nonumber\\
	&+C\left(\|\overline{u}^k\|_{L^6}\|\nabla^2\varphi^k\|_{L^3}+\|\nabla^2u^k\|_{L^3}
\|\overline\varphi^k\|_{L^6}+\|\nabla u^k\|_{L^\infty}
\|\nabla\overline\varphi^k\|_{L^2}\right)\|\nabla\overline{\varphi}^{k+1}\|_{L^2}\nonumber\\
	&+\left(\|\nabla\overline u^k\|_{L^3}\|\nabla \varphi^{k-1}\|_{L^6}
+\|{\varphi}^{k-1}\|_{L^\infty}\|\nabla^2\overline {u}^{k}\|_{L^2} \right)\|\nabla\overline{\varphi}^{k+1}\|_{L^2}\nonumber\\
	&+C\left(R^{k-1}(t)+\|\varphi^k\|_{L^\infty}\|\psi^{k-1}\|_{L^\infty}\|\nabla u^{k-1}\|_{L^3}
+\|\nabla^2  u^{k-1}\|_{L^3}\right)\|\nabla \overline{\varphi}^{k}\|_{L^2}
\|\nabla\overline{\varphi}^{k+1}\|_{L^2}\nonumber\\
	&+C\left(\|\varphi^k\|_{L^\infty}\|\xi^{k-1}\nabla^2u^{k-1}\|_{L^3}
+\|\nabla \varphi^{k}\|_{L^6}\|\xi^{k-1}\nabla  u^{k-1}\|_{L^6}\right)
\|\nabla \overline{\varphi}^{k}\|_{L^2}\|\nabla\overline{\varphi}^{k+1}\|_{L^2},\nonumber
	\end{align}
which together with (\ref{step3-0}) implies that
\begin{align}\label{step3}
	&\frac{d}{dt}\|\overline{\varphi}^{k+1}\|_1^2
	\leq
	 E^{k}_{2}(t)\|\overline{\varphi}^{k+1}\|_1^2+\epsilon (\|\nabla \overline{u}^{k}\|_1^2+\|\overline{\varphi}^k\|_1^2),
\end{align} with
\begin{align*}
	\int_0^tE_2^{k}(s)ds\leq C+C\epsilon^{-1}t.
\end{align*}

{\underline{\it{ Estimate on $\|\overline{f}^{k+1}\|_{L^2}$.}}}\\

Multiplying $(\ref{inter1})_2$ by $\overline{f}^{k+1}$, integrating the result over $\Omega$, and applying boundary condition $u^k\cdot n|_{\partial\Omega}=0$, we get
\begin{align}\label{step2}
	\frac{d}{dt}\|\overline{f}^{k+1}\|_{L^2}^2
	\leq &\|\overline{\varphi}^{k+1}\|_1^2+E_3^k(t)\|\overline{f}^{k+1}\|_{L^2}^2
+\epsilon(\|\overline{f}^k\|^2_{L^2}+\|\overline{\varphi}^k\|^2_1+\|\nabla\overline{u}^k\|^2_1),
\end{align}
with
\begin{align*}
	\int_0^tE_3^{k}(s)ds\leq C+C\epsilon^{-1}t.
\end{align*}
Here we have used the following facts
\begin{align}\label{r1r2r3}
	\|\Upsilon_2\|_{L^2}\leq&C\left(\|\overline{f}^{k}\|_{L^2}\|\nabla u^{k}\|_{L^\infty}
+\|\nabla\overline{ u}^{k}\|_{L^2}\|f^{k-1}\|_{L^\infty}+\|f^k\|_{L^\infty}\|\overline{\varphi}^{k+1}\|_{L^2}
\|\xi^k\nabla u^{k}\|_{L^\infty}\right)\nonumber\\
	&+C\|f^{k-1}\|_{L^\infty}\|\xi^{k-1}\nabla u^{k-1}\|_{L^\infty}\|\overline{\varphi}^{k}\|_{L^2},\nonumber\\
	\|\Upsilon_3\|_{L^2}\leq&C\left(\|\nabla f^{k}\|_{L^3}\|\overline{ u}^{k}\|_{L^6}+\|f^k\|_{L^\infty}
\|\nabla\overline{u}^{k}\|_{L^2}\right)\nonumber\\
&+C\left(\|\overline{\varphi}^{k+1}\|_{L^6}
\|\xi^k\nabla\mbox{div} u^{k}\|_{L^3}+\|\overline{\varphi}^{k}\|_{L^6}
\|\xi^{k-1}\nabla\mbox{div} u^{k-1}\|_{L^3}\right),\nonumber\\
		 \|\Upsilon_4\|_{L^2}\leq&C\left(\|f^{k+1}\|_{L^\infty}\|\overline{\varphi}^{k+1}\|_{L^2}\|\xi^{k}\mbox{div} u^{k}\|_{L^\infty}+\|f^{k}\|_{L^\infty}\|\overline{\varphi}^{k}\|_{L^2}\|\xi^{k-1}\mbox{div} u^{k-1}\|_{L^\infty}\right)\nonumber\\
&+C\left(\|\nabla u^k\|_{L^\infty}\|\overline{f}^{k+1}\|_{L^2}+\|f^k\|_{L^\infty}\|\nabla\overline{u}^{k}\|_{L^2}\right).
\end{align}

{\underline{\it  Estimate on $\overline{u}^{k+1}$.}}\\

 Multiplying $\eqref{inter1}_4$ by $2\overline{u}^{k+1}$, integrating over $\Omega$ and integrating by parts yield
\begin{align}\label{section5-1}
	\frac{d}{dt}&\|\sqrt{\varphi^{k+1}}\overline{u}^{k+1}\|_{L^2}^2+a\alpha \|\mbox{curl}\overline{u}^{k+1}\|_{L^2}^2+a(2\alpha+\beta)\|\mbox{div}\overline{u}^{k+1}\|_{L^2}^2
+a\alpha\int_{\partial\Omega}K(x)\overline{u}^{k+1}\cdot \overline{u}^{k+1}
\nonumber\\
 =&\int_\Omega\varphi^{k+1}_t|\overline{u}^{k+1}|^2
 -2\int_\Omega\varphi^{k+1}(u^k\cdot\nabla\overline{u}^{k+1}+\overline{u}^{k}\cdot\nabla u^k)\cdot\overline{u}^{k+1}\nonumber\\
	&-2\int_{\Omega}\overline{\varphi}^{k+1}(u_t^k+u^{k-1}\cdot \nabla u^k)\cdot\overline{u}^{k+1}
-2\int_{\Omega}
(\varphi^{k+1}\nabla\overline{\phi}^{k+1}+\overline{\varphi}^{k+1}\nabla\phi^k)\cdot\overline{u}^{k+1}\nonumber\\
	&+2\int_\Omega\left(f^{k+1}\cdot Q(\overline{u}^{k+1})+\overline{f}^{k+1}\cdot Q(u^k)\right)\cdot\overline{u}^{k+1}
	=\sum_{i=1}^{5}J_i.
	\end{align}
The terms on the right-hand side of (\ref{section5-1}) can be estimated as follows:
\begin{align*}
	J_1
=&\int_\Omega\mbox{div}\left(u^{k}|\overline{u}^{k+1}|^2\right)\varphi^{k+1}
-(\delta-1)\int_\Omega\left((\varphi^k)^{-1}(\varphi^{k+1})^2\mbox{div}u^{k}\right)|\overline{u}^{k+1}|^2\\
	\leq &\varepsilon\|\nabla\overline{u}^{k+1}\|^2_{L^2}
+C\left(\varepsilon^{-1}\|u^k\|^2_{L^\infty}\|\varphi^{k+1}\|_{L^\infty}+(R^{k}(t))^2\right)
\|\sqrt{\varphi^{k+1}}\overline{u}^{k+1}\|^2_{L^2},\\
	J_2\leq& \varepsilon\|\nabla\overline{u}^{k+1}\|^2_{L^2}
+\epsilon\|\nabla\overline{u}^{k}\|_{L^2}^2
+C\left(\varepsilon^{-1}\|\varphi^{k+1}\|_{L^\infty}\|u^k\|^2_{L^\infty}
+\epsilon^{-1}\|\varphi^{k+1}\|_{L^\infty}\|\nabla u^k\|_{L^3}^2
\right)\|\sqrt{\varphi^{k+1}}\overline{u}^{k+1}\|_{L^2}^2,\\
	J_3\leq& \varepsilon\|\nabla\overline{u}^{k+1}\|^2_{L^2}+ C\varepsilon^{-1}(S^{k}(t))^2\|\overline{\varphi}^{k+1}\|^2_{L^2},\\
	J_4
	\leq &\varepsilon\|\nabla\overline{u}^{k+1}\|^2_{L^2}
+C\varepsilon^{-1}\|\nabla\phi^k\|^2_{L^3}\|\overline{\varphi}^{k+1}\|^2_{L^2}
+C\left(\|\varphi^{k+1}\|_{L^\infty}
\|\sqrt{\varphi^{k+1}}\overline{u}^{k+1}\|^2_{L^2}+\|\nabla\overline{\phi}^{k+1}\|^2_{L^2}\right),
	\end{align*}
and
\begin{align*}
	J_5
	=&2\int_\Omega\left(\varphi^{k+1}\psi^{k+1}\cdot Q(\overline{u}^{k+1})+\overline{f}^{k+1}\cdot Q(u^k)\right)\cdot\overline{u}^{k+1}\\
		\leq& \varepsilon\|\nabla\overline{u}^{k+1}\|^2_{L^2}+C\varepsilon^{-1}\left(\|\varphi^{k+1}\|_{L^\infty}\|\psi^{k+1}\|^2_{L^\infty}
\|\sqrt{\varphi^{k+1}}\overline{u}^{k+1}\|^2_{L^2}
+\|\nabla u^k\|^2_{L^3}
\|\overline{f}^{k+1}\|^2_{L^2}\right).
	\end{align*}
Combining the above estimates for $J_i$ ($i=1,2,\cdots, 5$) into (\ref{section5-1}),
and  choosing $\varepsilon$ small, then there exists $c_0>0$ such that
\begin{align}\label{step7-1}
		\frac{d}{dt}&\|\sqrt{\varphi^{k+1}}\overline{u}^{k+1}\|_{L^2}^2+ c_0\|\nabla\overline{u}^{k+1}\|_{L^2}^2\nonumber\\[2mm]
				\leq &E^k_4(t)\|\sqrt{\varphi^{k+1}}\overline{u}^{k+1}\|_{L^2}^2+\epsilon\|\nabla\overline{u}^{k}\|_{L^2}^2
+C\|\nabla \overline{\phi}^{k+1}\|_{L^2}^2\nonumber\\[2mm]
&+C((S^{k}(t))^2+\|\nabla\phi^k\|^2_{L^3})
\|\overline{\varphi}^{k+1}\|^2_{L^2}
+C\|\nabla u^k\|^2_{L^3}\| \overline{f}^{k+1}\|_{L^2}^2,			
		\end{align}
with
\begin{align*}
	\int_0^tE_4^{k}(s)ds\leq C+C\epsilon^{-1}t.
\end{align*}
Furthermore, multiplying $\eqref{inter1}_4$ by $2\overline{u}_t^{k+1}$,  and integrating by parts  imply that
 \begin{align}\label{section5-2}
 	&	a\frac{d}{dt}\left(\alpha \|\mbox{curl}\overline{u}^{k+1}\|_{L^2}^2+(2\alpha+\beta)\|\mbox{div}\overline{u}^{k+1}\|_{L^2}^2\right)
 +2\|\sqrt{\varphi^{k+1}}	 \overline{u}_t^{k+1}\|_{L^2}^2
 +2a\alpha\int_{\partial\Omega} K\overline{u}^{k+1}\cdot \overline{u}_t^{k+1}\nonumber\\
 		 =&-2\int_\Omega\varphi^{k+1}(u^k\cdot\nabla\overline{u}^{k+1}+\overline{u}^{k}\cdot\nabla u^k)\cdot\overline{u}_t^{k+1}
 		-2\int_{\Omega}\overline{\varphi}^{k+1}(u_t^k+u^{k-1}\cdot \nabla u^k)\cdot\overline{u}_t^{k+1}\nonumber\\
 &-2\int_{\Omega}\varphi^{k+1}\nabla\overline{\phi}^{k+1}\cdot\overline{u}_t^{k+1}
 -2\int_{\Omega}\overline{\varphi}^{k+1}\nabla\phi^k\cdot\overline{u}_t^{k+1}\nonumber\\
 		&+2\int_\Omega f^{k+1}\cdot Q(\overline{u}^{k+1})\cdot\overline{u}_t^{k+1}
 +2\int_\Omega\overline{f}^{k+1}\cdot Q(u^k) \cdot\overline{u}_t^{k+1}
 		=\sum_{i=1}^{6}K_i.
  \end{align}
 Next we estimate the terms on the right-hand side of (\ref{section5-2}).
\begin{align*}
		 K_{1}
		\leq &\frac{1}{16}\|\sqrt{\varphi^{k+1}}\overline{u}_t^{k+1}\|^2_{L^2}
+C\|\varphi^{k+1}\|_{L^\infty}\|u^k\|^2_{L^\infty}
\|\nabla\overline{u}^{k+1}\|_{L^2}^2
+C\|\varphi^{k+1}\|_{L^\infty}
\|\nabla u^k\|^2_{L^3}\|\nabla\overline{u}^{k}\|^2_{L^2},\\
	K_{2}
			=&-2\frac{d}{dt}\int_{\Omega}\overline{\varphi}^{k+1}(u_t^k+u^{k-1}\cdot \nabla u^k)\cdot \overline{u}^{k+1}
			+2\int_{\Omega}\left(\overline{\varphi}^{k+1}(u_t^k+u^{k-1}\cdot \nabla u^k)\right)_t\cdot \overline{u}^{k+1}\\
		\leq &-2\frac{d}{dt}\int_{\Omega}\overline{\varphi}^{k+1}(u_t^k+u^{k-1}\cdot \nabla u^k)\cdot \overline{u}^{k+1}
+E^k_{5}(t)\|\nabla\overline{u}^{k+1}\|^2_{L^2}+ \|\overline\varphi^{k+1}\|_1^2+\|\nabla\overline{u}^{k}\|^2_{L^2}+\epsilon\|\overline\varphi^{k}\|_{L^2}^2,
	\end{align*}
with
\begin{align*}
	\int_0^tE_5^{k}(s)ds\leq C+C\epsilon^{-1}t.
\end{align*}
For terms $K_3$ and $K_4$,
\begin{align*}
		K_3\leq& \frac{1}{16}\|\sqrt{\varphi^{k+1}}\overline{u}_t^{k+1}\|^2_{L^2}
+C\|{\varphi}^{k+1}\|_{L^\infty}\|\nabla\overline\phi^{k+1}\|_{L^2}^2,\\
	K_4=&-2\frac{d}{dt}\int_{\Omega}\overline{\varphi}^{k+1}\nabla\phi^k\cdot\overline{u}^{k+1}
+2\int_{\Omega}\overline{\varphi}_t^{k+1}\nabla\phi^k\cdot\overline{u}^{k+1}
		+\int_{\Omega}\overline{\varphi}^{k+1}\nabla\phi_t^k\cdot\overline{u}^{k+1}\\
			 \leq&-2\frac{d}{dt}\int_{\Omega}\overline{\varphi}^{k+1}\nabla\phi^k\cdot\overline{u}^{k+1}
+E^k_6(t)\|\nabla\overline{u}^{k+1}\|^2_{L^2}+\|\overline\varphi^{k+1}\|_1^2
			+\|\nabla\overline{u}^{k}\|_{L^2}^2+\epsilon\|\overline\varphi^k\|_{L^2}^2,
	\end{align*}
with
\begin{align*}
	\int_0^tE_6^{k}(s)ds\leq C+C\epsilon^{-1}t.
\end{align*}
For terms $K_5$ and $K_6$,
\begin{align*}
		K_5\leq& \frac{1}{16}\|\sqrt{\varphi^{k+1}}\overline{u}_t^{k+1}\|^2_{L^2}
+C\|{\varphi}^{k+1}\|_{L^\infty}\|\psi^{k+1}\|_{L^\infty}^2\|\nabla\overline u^{k+1}\|_{L^2}^2,\\
		K_6
		 =&2\frac{d}{dt}\int_{\Omega}\overline{f}^{k+1}\cdot Q(u^k)\cdot\overline{u}^{k+1}
-2\int_{\Omega}\overline{f}^{k+1}\cdot Q(u^k_t)\cdot\overline{u}^{k+1}
-2\int_{\Omega}\overline{f}_t^{k+1}\cdot Q(u^k)\cdot\overline{u}^{k+1}
\\
\leq&2\frac{d}{dt}\int_{\Omega}\overline{f}^{k+1}\cdot Q(u^k)\cdot\overline{u}^{k+1}+E_7^k(t)\|\nabla \overline{u}^{k+1}\|^2_{L^2}+C\| \overline{f}^{k+1}\|^2_{L^2}
+C\| \overline{\varphi}^{k+1}\|^2_{1}+C\|\nabla \overline{u}^{k}\|^2_{L^2}\\
&+\epsilon\left(\|\nabla^2 \overline{u}^{k}\|^2_{L^2}+\| \overline{f}^{k}\|^2_{L^2}
+\|\overline{\varphi}^{k}\|^2_{1}\right),
	\end{align*}
with
\begin{align*}
	\int_0^tE_7^{k}(s)ds\leq C+C\epsilon^{-1}t,
\end{align*}
where  we have used the expression of $\overline{f}^{k+1}_t$ as in  $(\ref{inter1})_2$ and the estimates
(\ref{r1r2r3}).
Putting the above estimates into  (\ref{section5-2}),
we obtain that
\begin{align}\label{step7}
		&	a\frac{d}{dt}\left(\alpha \|\mbox{curl}\overline{u}^{k+1}\|_{L^2}^2+(2\alpha+\beta)\|\mbox{div}\overline{u}^{k+1}\|_{L^2}^2\right)
+\|\sqrt{\varphi^{k+1}}	 \overline{u}_t^{k+1}\|_{L^2}^2
 +\frac{d}{dt}\int_{\partial\Omega} K\overline{u}^{k+1}\cdot \overline{u}^{k+1}\nonumber\\
		\leq &\frac{d}{dt}\mathcal{D}(t)
+ (E_{6}^{k}(t)+E_{7}^{k}(t))\|\nabla \overline{u}^{k+1}\|_{L^2}^2+
C\| \overline{\varphi}^{k+1}\|_{1}^2+C\|\nabla\overline\phi^{k+1}\|_{L^2}^2
\nonumber\\
		&+C\| \overline{f}^{k+1}\|_{L^2}^2
		 +C\|\nabla\overline{u}^{k}\|^2_{L^2}
+\epsilon\left(\|\nabla^2 \overline{u}^{k}\|^2_{L^2}+\|\overline{\varphi}^k\|_{1}^2+\|\overline{f}^k\|_{L^2}^2\right),		 \end{align}
where
\begin{align*}
	\mathcal{D}(t)=2\int_{\Omega}\left(-\overline{\varphi}^{k+1}(u^k_t+u^{k-1}\cdot\nabla u^k+\nabla\phi^k)
+\overline{f}^{k+1}\cdot Q(u^k)\right)\cdot \overline{u}^{k+1}.
\end{align*}
Then it follows from inequalities \eqref{step1}, \eqref{step2}, \eqref{step3}, \eqref{step7-1} and \eqref{step7} that
\begin{align}\label{fin}
	&\frac{d}{dt}\left(
\|\overline{\phi}^{k+1}\|^2_1+\|\overline{\varphi}^{k+1}\|_1^2
	 +\|\overline{f}^{k+1}\|_{L^2}^2+\|\sqrt{\varphi^{k+1}}\overline{u}^{k+1}\|_{L^2}^2\right)\nonumber\\
	&+ a\epsilon_1\frac{d}{dt}\left(\alpha \|\mbox{curl}\overline{u}^{k+1}\|_{L^2}^2
+(2\alpha+\beta)\|\mbox{div}\overline{u}^{k+1}\|_{L^2}^2\right)
+\epsilon_1\|\sqrt{\varphi^{k+1}}\overline u^{k+1}_t\|_{L^2}^2+
c_0\|\nabla \overline{u}^{k+1}\|_{L^2}^2\nonumber\\
	\leq& \epsilon_1\frac{d}{dt}\mathcal{D}(t)+
C\epsilon(1+\epsilon_1)(\|\nabla^2\overline{u}^{k}\|_{L^2}^2+\|\overline{\varphi}^k\|_1^2
+\|\overline{f}^k\|^2_{L^2})+C(\epsilon+\epsilon_1)\|\nabla\overline{u}^{k}\|_{L^2}^2\nonumber\\
&+(E_{\epsilon}^k(t)+\epsilon_1)\left(\|\nabla\overline{u}^{k+1}\|^2_{L^2}
+\|\overline{\phi}^{k+1}\|^2_1\right)\nonumber\\
&+(E_{\epsilon}^k(t)+\epsilon_1)\left(\|\overline{\varphi}^{k+1}\|_1^2
	+\|\overline{f}^{k+1}\|_{L^2}^2+\|\sqrt{\varphi^{k+1}}\overline{u}^{k+1}\|_{L^2}^2\right),
\end{align}
where $\epsilon_1>0$ is a sufficiently small constant, and  $E_{\epsilon}^k(t)$ satisfies
\begin{align}\label{Ek}
	\int_0^tE_{\epsilon}^k(s)ds\leq C+C\epsilon^{-1}t,\quad \mbox{for}\quad 0\leq t\leq T^{**}.
\end{align}
According to equations $(\ref{inter1})_4$ and the elliptic estimate as in Lemma \ref{lem-elliptic-1}, we have
\begin{align}\label{D^2u}
\|\nabla^2\overline{u}^{k}\|_{L^2}^2	\leq&C\left( \|\sqrt{\varphi^{k}}\overline u^{k}_t\|_{L^2}^2 + \|\nabla \overline{u}^{k-1}\|_{L^2}^2+\|\nabla\overline\phi^{k}\|_{L^2}^2+\|\overline \varphi^{k}\|_1^2\right)\nonumber\\
&+
C\left(\|\overline f^{k}\|_{L^2}^2+\|\nabla \overline{u}^{k}\|_{L^2}^2\right).
\end{align}
Using the Cauchy's inequality, it holds
\begin{align}\label{Dt}
	\int_0^t\frac{d}{dt}\mathcal{D}(s)ds	\leq&C( \|\overline u^{k+1}\|_{L^2}^2 +
\|\overline{\varphi}^{k+1}\|_{L^2}^2+ \| \overline{f}^{k+1}\|_{L^2}^2).
\end{align}
Finally, we denote
\begin{align*}
	\Gamma ^{k+1}(t,\epsilon_1)=& \sup_{0\leq s\leq t}\left(\|\overline{\phi}^{k+1}\|^2_1+\|\overline{\varphi}^{k+1}\|_1^2
	+\|\overline{f}^{k+1}\|_{L^2}^2+\|\sqrt{\varphi^{k+1}}\overline{u}^{k+1}\|_{L^2}^2\right)\\
	&+ a\epsilon_1\sup_{0\leq s\leq t}\left(\alpha \|\mbox{curl}\overline{u}^{k+1}\|_{L^2}^2
+(2\alpha+\beta)\|\mbox{div}\overline{u}^{k+1}\|_{L^2}^2\right).
	\end{align*}
Then it follows from (\ref{fin})-(\ref{Dt})  that
\begin{align*}
	&\Gamma^{k+1}(t,\epsilon_1)+\int_0^t \left(c_0\|\nabla \overline{u}^{k+1}\|_{L^2}^2
	+\epsilon_1\|\sqrt{\varphi^{k+1}}\overline u^{k+1}_t\|_{L^2}^2\right) ds\nonumber\\
	&\leq C\Big\{\int_0^t
	\epsilon(1+\epsilon_1)\left(\|\nabla \overline{u}^{k-1}\|_{L^2}^2+	 \|\nabla\overline{u}^{k} \|_{L^2}^2+\|\sqrt{\varphi^{k}}\overline u^{k}_t\|_{L^2}^2\right)ds+ (\epsilon+\epsilon_1) t\Gamma^{k}(t)
\Big\}
\cdot\exp{(C+C\epsilon^{-1}t)}.
\end{align*} 	
Now we choose $\epsilon\in(0,1)$ and $\epsilon_1\in(0,1)$ small enough, and  pick up  $T_{*}\in(0,\min(1,T^{**}))$ small enough such that
\begin{align*}
	&C(\epsilon+\epsilon_1)\exp C\leq \frac{1}{32},\quad
C\epsilon(1+\epsilon_1)\exp C\leq \frac{\epsilon_1}{32},\\
&C\epsilon(1+\epsilon_1)\exp C\leq \frac{c_0}{32},\quad
(1+T_{*})\exp(C\epsilon T_{*})\leq 4.
\end{align*}
Then it is easy to see that
\begin{align}\label{full}
	&\sum _{k=1}^{\infty}\sup_{0\leq t\leq T_*}\left(\|\overline{\phi}^{k+1}\|^2_1+\|\overline{\varphi}^{k+1}\|_1^2
	+\|\sqrt{\varphi^{k+1}}\overline{u}^{k+1}\|_{L^2}^2
	+  \|\nabla \overline{u}^{k+1}\|_{L^2}^2+
	\|\overline{f}^{k+1}\|_{L^2}^2\right)\nonumber\\
	&+\sum _{k=1}^{\infty}\int_0^{T_*} \left(\|\nabla \overline{u}^{k+1}\|_{L^2}^2
	+\|\sqrt{\varphi^{k+1}}\overline u^{k+1}_t\|_{L^2}^2\right) ds\leq C.
\end{align}
Thanks to (\ref{full}) and the local estimates (\ref{local-linear}) independent of $k$, we have
\begin{align}
	\lim_{k\rightarrow \infty}\| \overline{f}^{k+1}\|_{L^6}=0,\quad
	\lim_{k\rightarrow \infty}\| \overline{\varphi}^{k+1}\|_{L^\infty}=0.\nonumber
\end{align}
Thus, the whole sequence $(\phi^{k},\varphi^k,f^k,u^{k})$ converges
to a limit $(\phi,\varphi,  f,u)$ in the following strong sense: for any $s'\in[1,3)$,
\begin{align}\label{solu}
	&\phi^{k}-\phi^\infty\rightarrow \phi -\phi^\infty\quad \mbox{in}\quad L^{\infty}([0,T_*];H^{s'}(\Omega)),\nonumber\\
	&f^{k}\rightarrow f \quad \mbox{in}\quad L^{\infty}([0,T_*]; L^6(\Omega)),\nonumber\\
	&\varphi^k\rightarrow \varphi \quad \mbox{in}\quad L^{\infty}([0,T_*];L^\infty(\Omega)),\nonumber\\
	&u^{k}\rightarrow u \quad \mbox{in}\quad  L^{\infty}([0,T_*]; D^1\cap D^{s'}(\Omega)).
\end{align}
Furthermore, due to the a priori estimate \eqref{local-linear} independent of $k$,  there exists a subsequence (still denoted by $(\phi^k, \varphi^k, f^k,  u^k )$) converging to the limit $(\phi, \varphi, f,  u )$ in the weak or weak$^{*}$ sense. According to the lower semi-continuity of norms, the corresponding estimates in \eqref{local-linear} for $(\phi, \varphi, f,  u )$ still hold except those weighted estimates on $u$.
Thus, $(\phi, \varphi, f, u )$ is a weak solution in the sense of distributions to the  following initial-boundary value problem:
\begin{equation}\label{non-limit1}
	\left\{
	\begin{array}{llll}
		\phi_t+u\cdot\nabla\phi+(\gamma-1)\phi\mbox{div}u=0,\\
		\varphi(u_{t}+u\cdot\nabla u+\nabla\phi) +aLu=f\cdot Q(u),\\
		f_t+\sum_{l=1}^3A_l(u)\partial_lf+B(u)f
+a\delta\nabla\mbox{div}u=0,\\
		\varphi_t+u\cdot\nabla\varphi-(\delta-1)\varphi \mbox{div}u=0,\\
		(\phi,\varphi,f,u)|_{t=0}
=\left(\phi_0,\phi_0^{-2\kappa}, \phi_0^{-2\kappa}\psi_0, u_0\right),\\
		(\phi,\varphi,f,u)\rightarrow(\phi^{\infty},(\phi^{\infty})^{-2\kappa},0,0),\quad \mbox{as} \quad |x|\rightarrow \infty,\quad t\geq0,\\
		u\cdot n=0,\quad \mbox{curl}u\times n= -K(x)u,\quad \mbox{on}\quad \partial\Omega.
		\end{array}
	\right.
\end{equation}
\vspace{1mm}

{\it{Step 1.2:} Strong convergence of $\psi^{k}$ and the existence to the problem \eqref{non-1}.}

Note that the conclusions
obtained in Step 1.1 above  means that the  existence of the strong solution $(\phi, \varphi, f, u )$ to initial-boundary value problem \eqref{non-limit1}. Next, we may show the  existence of the strong solution $(\phi,\psi, u)$ to initial-boundary value problem \eqref{non-1}.

We first need to check the strong convergence of $\psi^k$:
\begin{align*}
 \|\psi^{k+1}-\psi^{k}\|_{L^6}
 \leq C_{\sigma}\left(\|\varphi^k\|_{L^\infty}\|\overline{f}^{k+1}\|_{L^6}
+\|f^k\|_{L^6}\|\overline{\varphi}^{k+1}\|_{L^\infty}\right)
	\end{align*}
which along with  (\ref{solu}) yields that
\begin{align}\label{psi-k}
	\psi^{k}\rightarrow \psi \quad \mbox{in}\quad L^{\infty}([0,T_*]; L^6(\Omega)).
	\end{align}
Next, we need to show the relation $f=\varphi\psi$ still holds for the limit functions. Note that
\begin{align*}
\|f^{k}-\psi\varphi\|_{L^6}\leq C\left(\|\varphi^k-\varphi\|_{L^\infty}\|\psi^k\|_{L^6}+\|\psi^k-\psi\|_{L^6}\|\varphi\|_{L^\infty}\right),
\end{align*}
which implies that
\begin{align}\label{f-k}
	f(t,x)= \psi\varphi(t,x) \quad  \mbox{a.e. on }\ [0,T_{*}]\times\Omega.
\end{align}
In order to check the following relations hold out
\begin{equation}\label{relation-lim}
		f=\frac{2a\delta\kappa}{\delta-1}\frac{\nabla\phi}{\phi},\quad \varphi=\frac{1}{\phi^{2\kappa}},\quad\psi=\frac{a\delta}{\delta-1}\nabla\phi^{2\kappa}.
\end{equation}
Denote
\begin{equation*}
	f^*=f-\frac{2a\delta\kappa}{\delta-1}\frac{\nabla\phi}{\phi},\quad \varphi^*=\varphi-\phi^{-2\kappa}.
\end{equation*}
Then it follows from the equations $(\ref{non-limit1})_1$ and $(\ref{non-limit1})_3-(\ref{non-limit1})_4$ that
\begin{equation*}
	\left\{
	\begin{array}{lll}
		f^*_t+\sum_{l=1}^3A_l(u)\partial_l f^{*}+B(u)f^{*}=0,\\
		\varphi^*_t+u\cdot\nabla\varphi^*-(\delta-1)\varphi^* \mbox{div}u=0,\\
		(f^*, \varphi^*)|_{t=0}=\left(0,0\right),\\[2mm]
		(f^*,\varphi^*)\rightarrow(0,0),\quad \mbox{as} \quad |x|\rightarrow \infty,\quad t\geq0.
	\end{array}
	\right.
\end{equation*}
By a standard energy method, we show that
\begin{equation*}
f^*=0,\quad \varphi^*=0 \quad\mbox{for}\quad  (t,x)\in[0,T_*]\times\Omega.
\end{equation*}
Therefore the first two relations of \eqref{relation-lim} have been confirmed. At the same time, the relation $\psi=\frac{a\delta}{\delta-1}\nabla\phi^{2\kappa}$ can be verified by $f=\psi\varphi$.

Next, we should show the weak convergence of the $\xi^{k}$-weighted sequences.  Let $\xi=\varphi^{-1}$.
In view of the uniform positivity for $\varphi^{k}\geq C(\sigma)>0$,  $\varphi\geq C(\sigma)>0$ and (\ref{solu}), we have
 \begin{align*}
 	&\int^{T_*}_0\int_\Omega\left(\xi^k\nabla u^{k}-\varphi^{-1}\nabla u\right)w
 	=\int^{T_*}_0\int_\Omega\left(\left(\frac{\varphi-\varphi^k}{\varphi^k\varphi}\right)\nabla u^{k}
 +\varphi^{-1}(\nabla u^k-\nabla u)\right)w\nonumber\\
 	&\leq C(\sigma)\left(\|\varphi^k-\varphi\|_{L^\infty([0,T_*];L^{\infty}(\Omega))}
 +\|\nabla u^k-\nabla u\|_{L^\infty([0,T_*];L^2(\Omega))}\right)T_*\rightarrow0, \quad\text{as}\ k\rightarrow+\infty,
 	\end{align*}
 for any test functions $w(x,t)\in C_c^\infty(\Omega\times [0,T_*])$, which implies that
 \begin{align}\label{weightedu1}
 	\xi^k\nabla u^k\rightharpoonup \varphi^{-1}\nabla u \quad \mbox{weakly$^*$ in}\quad L^{\infty}([0,T_*];  H^2(\Omega)).
 \end{align}
Similarly, we also have
\begin{align}\label{weightedu2}
	&\xi^k\nabla u_t^k\rightharpoonup{\varphi^{-1}}\nabla u_t\quad \mbox{weakly$^*$ in}\quad L^{\infty}([0,T_*]; L^2(\Omega)),\nonumber\\
	&\xi^k\nabla^2u^k\rightharpoonup \varphi^{-1}\nabla^2u \quad \mbox{weakly$^*$ in}\quad L^{\infty}([0,T_*]; H^1(\Omega)),\nonumber\\
	&	\xi^k\nabla^2u^k\rightharpoonup \varphi^{-1}\nabla^2u \quad \mbox{weakly in}\quad L^{2}([0,T_*]; D^1\cap D^2(\Omega)),\nonumber\\
		&(\xi^k\nabla^2u^k)_t\rightharpoonup( \varphi^{-1}\nabla^2u)_t \quad \mbox{weakly in}\quad L^{2}([0,T_*]; L^2(\Omega)).
	\end{align}
Therefore the weighted estimates for $u$ shown in the a priori estimates \eqref{local-linear} also hold for the limit functions. Under the help of  a priori estimates \eqref{local-linear}, the above  convergences  \eqref{solu}, \eqref{psi-k}, \eqref{weightedu1}-\eqref{weightedu2}  and relations (\ref{f-k}),(\ref{relation-lim}), it is clear that the functions
\begin{align*}
(\phi,u,\xi=\phi^{2\kappa},\psi=\frac{a\delta}{\delta-1}\nabla\phi^{2\kappa})
\end{align*}
meet the initial-boundary value problem \eqref{non-1} in the sense of distributions. Moreover $(\phi, \varphi, f, u )$ satisfy the a priori estimates \eqref{local-linear}.

 {\it\textbf{Step 2:} Uniqueness.} The uniqueness
can be obtained by standard procedure, we omit it.
\end{proof}

\subsection{Taking limit from the non-vacuum flows to the flow with far field vacuum}\label{sub3.5}
Based on the  estimates in (\ref{local-linear}), now we are ready to prove Theorem \ref{thm2}.

{\bf{Proof of Theorem \ref{thm2}.}}
	{\it{Step 1: The locally uniform positivity of $\phi$.} }

For any $\sigma\in(0,1)$, define
	\begin{align*}
		\phi_0^\sigma=\phi_0+\sigma,\quad \psi_0^\sigma=\frac{a\delta}{\delta-1}\nabla(\phi_0+\sigma)^{2\kappa},\quad \xi_0^\sigma=(\phi_0+\sigma)^{2\kappa}.
		\end{align*}
		Then the initial compatibility conditions can be given as
	\begin{equation*}
		\left\{
		\begin{array}{ll}
		\nabla u_0=(\phi_0+\sigma)^{-2\kappa}g_1^\sigma,\quad Lu_0=(\phi_0+\sigma)^{-2\kappa}g_2^\sigma,\\
			 \nabla\left(a(\phi_0+\sigma)^{2\kappa}Lu_0\right)=(\phi_0+\sigma)^{-2\kappa}g_3^\sigma,
		\end{array}
		\right.
	\end{equation*}	
which implies that
		\begin{align*}
			&g_1^\sigma=\frac{\phi_0^{-2\kappa}}{(\phi_0+\sigma)^{-2\kappa}}g_1,\quad g_2^\sigma=\frac{\phi_0^{-2\kappa}}{(\phi_0+\sigma)^{-2\kappa}}g_2,\\
			&g_3^\sigma=\frac{\phi_0^{-4\kappa}}{(\phi_0+\sigma)^{-4\kappa}}
\left(g_3-\frac{a\sigma\nabla\phi_0^{2\kappa}}{\phi_0+\sigma}\phi_0^{2\kappa}Lu_0\right)
			=\frac{\phi_0^{-4\kappa}}{(\phi_0+\sigma)^{-4\kappa}}
\left(g_3-\frac{a\sigma\nabla\xi_0}{\phi_0+\sigma}g_2\right).
	\end{align*}
Then according to the initial assumption (\ref{intial2}) and (\ref{intial3}),  there exists a $\sigma_1$ such that if $0<\sigma<\sigma_1$, then
\begin{align}\label{know-2}
		2&+\sigma+\|\phi^\sigma_0-\sigma\|_3+\|\psi_0^\sigma\|_{D^1\cap D^2}+\|u_0\|_3+\|g^\sigma_1\|_{L^2}+\|g^\sigma_2\|_{L^2}+\|g^\sigma_3\|_{L^2}\nonumber\\
		&+\|(\xi^\sigma_0)^{-1} \|_{L^\infty\cap D^{1,6}\cap D^{2,3}\cap D^3}+\|\nabla\xi^\sigma_0/\xi^\sigma_0\|_{L^\infty\cap D^{1,3}\cap D^2}\leq \bar{c}_0,
	\end{align}
where $ \bar{c}_0$ is a positive constant independent  of $\sigma$. Therefore, if we take $(\phi^{\sigma}_0,\psi_0^\sigma, u_0)$ as the initial data, then the initial boundary problem (\ref{non-1}) exists a unique strong solution $(\phi^{\sigma}, \psi^\sigma, u^{\sigma})$, which satisfies the local estimates in (\ref{c_i})-(\ref{local-linear}) with $c_0$ replaced by $\bar{c}_0$, and the life span $T_{*}$  is independent of $\sigma$.

Moreover, we have the following result, which is analogous to Lemma 3.9 in \cite{XZ2021jmpa}.
\begin{lem}\label{lem-3.11}
For any $R_0 > 0$ and $\sigma \in (0, 1]$, there exists a constant $a_{R_0}$ such that
\begin{equation} \label{lower-bound}
    \phi^\sigma(t, x) \geq a_{R_0} > 0, \quad \forall (t, x) \in [0, T^*] \times B_{R_0},
\end{equation}
where $a_{R}$ is independent of $\sigma$.
\end{lem}

{\it Step 2: Existence.}

Because the estimates (\ref{local-linear}) are independent of $\sigma$, one can find a subsequence of solutions (still denoted by) $(\phi^{\sigma},\psi^{\sigma},u^{\sigma})$ that converges to $(\phi,\psi,u)$ in the weak$^*$ sense as $\sigma\rightarrow0$, that is:
\begin{equation}\label{weak-1}
	\begin{split}
		&(\phi^\sigma-\sigma, u^{\sigma})\rightharpoonup(\phi,u)\quad \mbox{weakly}^{*}\quad \mbox{in}\ L^\infty([0,T^*];H^3);\\
		&\psi^\sigma\rightharpoonup\psi\quad \mbox{weakly}^{*}\quad \mbox{in}\ L^\infty([0,T^*];D^1\cap D^2);\\
		&\phi_t^\sigma\rightharpoonup\phi_t\quad \mbox{weakly}^{*}\quad \mbox{in}\ L^\infty([0,T^*];H^2);\\
		&(\psi_t^{\sigma},u_t^{\sigma})\rightharpoonup(\psi_t,u_t)\quad \mbox{weakly}^{*}\quad \mbox{in}\ L^\infty([0,T^*];H^1).
	\end{split}
\end{equation}
In addition, for any constant $R>0$, employing the uniform estimates (\ref{local-linear}) and Aubin-Lions Lemma (refer to Lemma \ref{lem-AL}), there  exists  a subsequence of solutions (still denoted by) $(\phi^{\sigma},\psi^{\sigma}, u^{\sigma})$ satisfying
\begin{equation}\label{strong}
	\lim_{\sigma\rightarrow0}(\phi^\sigma,\psi^{\sigma},u^{\sigma})=(\phi,\psi, u)\in C\left([0,T^*];H^1(B_R\cap\Omega)\right)
\end{equation}
where $B_R=\{x\in\mathbb{R}^3|\,|x|\leq R\}$.
Then through the uniform estimates (\ref{local-linear}), the weak or weak$^*$ convergence in (\ref{weak-1}) along with the strong convergence in (\ref{strong})  expected those weighted estimates on $u$.
At the same time, the following relations satisfy the limit functions:
\begin{align*}
		f=\frac{2a\delta\kappa}{\delta-1}\frac{\nabla\phi}{\phi},\quad \varphi=\frac{1}{\phi^{2\kappa}},\quad\psi=\frac{a\delta}{\delta-1}\nabla\phi^{2\kappa},
	\end{align*}
which could be proved by the similar argument as the proof of  \eqref{relation-lim}.

Set $\xi=\varphi^{-1}$, employing  Lemma \ref{lem-3.11}  and the weak$^*$ convergences  given in \eqref{weak-1} and strong convergences given in \eqref{strong}, we can obtain
\begin{align*}
	&\sqrt{\xi^\sigma}u^\sigma_t\rightharpoonup\varphi^{-\frac{1}{2}} u_t\quad \mbox{weakly}^{*}\quad \mbox{in}\ L^{\infty}([0,T_*]; L^2(\Omega)),\nonumber\\
	&\xi^\sigma\nabla u^\sigma\rightharpoonup \varphi^{-1}\nabla u \quad \mbox{weakly}^{*}\quad \mbox{in}\
 L^{\infty}([0,T_*]; H^2(\Omega)),\nonumber\\
	&\xi^\sigma\nabla^2u^\sigma\rightharpoonup \varphi^{-1}\nabla^2u \quad \mbox{weakly}\quad \mbox{in}\  L^{2}([0,T_*]; D^1\cap D^2(\Omega)),\nonumber\\
	&(\xi^\sigma\nabla^2u^\sigma)_t\rightharpoonup( \varphi^{-1}\nabla^2u)_t \quad \mbox{weakly}\quad \mbox{in} \ L^{2}([0,T_*]; L^2(\Omega)),\nonumber
\end{align*}
by the similar approach for proving \eqref{weightedu1}. Therefore, the uniform weighted estimates for $u$ established in (\ref{local-linear}) remain valid for the limiting functions.

	{\it{Step 3: The uniqueness and time continuity.}}  The uniqueness and time continuity for $(\phi,\psi,u)$ can be obtained by standard procedure,  we omit the details here.
\qed

\subsection{The proof for Theorem \ref{thm1}}\label{sub3.6}

With Theorem \ref{thm2} at hand, we are ready to establish the local-in-time well-posedness of the regular solution to the original the initial-boundary value problem (\ref{prob1})-(\ref{1.8})  shown in Theorem \ref{thm1}.

{\bf{Proof of Theorem \ref{thm1}.}}
It follows from the initial assumptions (\ref{initial data}), (\ref{comp}) and Theorem \ref{thm2} that there exists a time $T_{*}>0$ such that  the  enlarged  initial-boundary value problem (\ref{ibvp-sect3})  has a unique regular solution  $(\phi ,\psi, u)$ satisfying the regularity (\ref{thm2-reg}), which implies that
$\phi\in C^{1}([0,T_{*}]\times \Omega)$.
Set $\rho=\left(\frac{\gamma-1}{A\gamma}\phi\right)^{\frac{1}{\gamma-1}}$ with $\rho(0,x)=\rho_0$. It follows from the proof of Theorem \ref{thm2} that the following relations hold:
\begin{align*}
f=\frac{2a\delta\kappa}{\delta-1}\frac{\nabla\phi}{\phi},\quad \varphi=\frac{1}{\phi^{2\kappa}},\quad\psi=\frac{a\delta}{\delta-1}\nabla\phi^{2\kappa},
\end{align*}
which implies that
\begin{align*}
f=a\delta\nabla \log \rho,\quad \varphi=a\rho^{1-\delta}\quad \mbox{and}\quad \psi=\frac{\delta}{\delta-1}\nabla \rho^{\delta-1}.
\end{align*}
Due to the above regularities and relations of  $(\phi, \psi, u)$, then multiplying $(\ref{ibvp-sect3})_1$ by
\begin{align*}
\frac{\partial\rho}{\partial\phi}(t,x)=\frac{1}{\gamma-1}\left(\frac{\gamma-1}{A\gamma}\right)^{\frac{1}{\gamma-1}}
\phi^{\frac{2-\gamma}{\gamma-1}}(t,x)
\end{align*}
yields the continuity equation in $(\ref{prob1})_1$, while multiplying $(\ref{ibvp-sect3})_2$  by $\rho(t,x)$ gives the momentum equations in $(\ref{prob1})_2$.
Thus we have shown that $(\rho,u )$ satisfies problem (\ref{prob1})-(\ref{1.8})
 in the sense of distributions and has the regularities shown in Definition \ref{def1} and (\ref{thm1-reg}).
In summary, the  problem (\ref{prob1})-(\ref{1.8}) has a unique regular solution.
Moreover, if $1<\gamma\leq2$, one has $\rho\in C^{1}([0,T_{*}]\times \Omega)$ thanks to $\phi\in C^{1}([0,T_{*}]\times \Omega)$.  So, $(\rho,u)$ is a strong solution to (\ref{prob1})-(\ref{1.8}).
\qed
\\

\centerline{\bf Acknowledgements}
Liu is supported by National Natural Science Foundation of China (No.12471198, No.
12431018).  Zhong is supported by National Natural Science Foundation of China (No. 12201520).


\begin{thebibliography}{99}





\bibitem{BD2003}
	D. Bresch, B. Desjardins, Existence of global weak solutions for a 2D viscous shallow water equations and convergence to the quasi-geostrophic model, Comm. Math. Phys., 238 (2003) 211-223.


	\bibitem{BVY2022}
	D. Bresch, A. Vasseur, C. Yu, Global existence of entropy-weak solutions to the compressible Navier-Stokes equations with non-linear density-dependent viscosities, J. Eur. Math. Soc., 24 (2022) 1791-1837.

\bibitem{CLL2021}
G.Cai, J.Li, B.Lv, Global classical solutions to the compressible
Navier-Stokes equations with slip boundary onditions in
3D exterior domains, arXiv:2112.05586.


\bibitem{CL2021}
G.Cai, J. Li, Existence and exponential growth of global classical solutions to the compressible Navier-Stokes
equations with slip boundary conditions in 3D bounded domains, Indiana Univ. Math. J. (in press)

\bibitem{CLZ2022}
Y. Cao, H. Li, S. Zhu, Global regular solutions for one-dimensional degenerate compressible Navier-Stokes equations with large data and far field vacuum, SIAM J. Math. Anal., 54 (2022), 4658-4694.

	\bibitem{CLZ2024}Y. Cao,  H. Li, S. Zhu, Global spherically symmetric solutions to degenerate compressible Navier-Stokes equations with large data and far field vacuum, Calc. Var., 63 (2024) 230.

\bibitem{ChapmanCowling}
S. Chapman, T. Cowling, The Mathematical Theory of Non-Uniform Gases: An Account of the Kinetic Theory of Viscosity, Thermal Conduction and Diffusion in Gases, Cambridge University Press, 1990.

\bibitem{Cho2004}
Y. Cho, H. Choe, H.  Kim,  Unique solvability of the initial boundary value problems
for compressible viscous fluids, J. Math. Pures Appl., 83  (2004) 243-275.

\bibitem{Cho2006}
Y. Cho, H. Kim, On classical solutions of the compressible Navier-Stokes equations
with nonnegative initial densities, Manuscr. Math., 120 (2006) 91-129.

\bibitem{DL2024}
Y. Du, H. Liu, Existence and nonlinear stability of steady state solutions of the 3-D Navier-Stokes-Poisson equations with non-flat doping profile of large variations in exterior domains,
J. Math. Phys., 66 (2025) 081511.

	\bibitem{DXZ2023}
	Q. Duan, Z. Xin, S. Zhu, On regular solutions for three-dimensional full compressible Navier-Stokes
	equations with degenerate viscosities and far field vacuum, Arch. Ration. Mech. Anal., 247 (2023) 3.




\bibitem{Galdi1994book}
G. Galdi, An Introduction to the Mathematical Theory of the Navier-Stokes Equations, Springer, New York, 1994.

\bibitem{GLZ2019}
Y. Geng, Y. Li, S. Zhu, Vanishing viscosity limit of the Navier-Stokes equations to the Euler equa-tions for compressible fluid flow with vacuum, Arch. Ration. Mech. Anal., 234 (2019) 727-775.

\bibitem{GL2016}
	P. Germain, P. LeFloch, Finite energy method for compressible fluids: the Navier-Stokes-Korteweg model, Comm.
	Pure Appl. Math., 69 (2016) 3-61.



\bibitem{HLX2012}
X. Huang,   J. Li, Z.  Xin, Global well-posedness of classical solutions with large oscillations
and vacuum to the three-dimensional isentropic compressible Navier-Stokes
equations, Commun. Pure Appl. Math., 65 (2012)  549-585.


\bibitem{JXZ2005}
S. Jiang, Z. Xin,  P. Zhang, Global weak solutions to 1D compressible isentropic Navier-Stokes equations with
density-dependent viscosity, Methods Appl. Anal. 12(3) (2005) 239.

	\bibitem{JLY2014}
	Q. Jiu, M. Li,  Y. Ye, Global classical solution of the Cauchy problem to 1D compressible Navier-Stokes equations
	with large initial data, J. Differential Equations, 257 (2014)  311-350.


	\bibitem{JX2008}
	Q. Jiu, Z. Xin, The cauchy problem for 1D compressible flows with density-dependent viscosity coefficients, Kinet.
	Relat. Models, 1 (2008) 313-330.

\bibitem{LX2015}
J. Li, Z. Xin, Global existence of weak solutions to the barotropic compressible Navier-Stokes flows
with degenerate viscosities, preprint, (2015), arXiv:1504.06826.

\bibitem{LiQin} Tatsien Li, T. Qin, Physics and Partial Differential Equations, Siam: Philadelphia, Higher Education Press, Beijing, 2014.

\bibitem{LPZ2019}
Y. Li, R. Pan, S. Zhu, On classical solutions for viscous polytropic fluids with degenerate viscosities and vacuum, Arch. Ration. Mech. Anal., 234 (2019) 1281-1334.

\bibitem{LPZ20172D}
Y. Li, R. Pan, S. Zhu, On classical solutions to 2D shallow water equations with degenerate viscosities, J. Math. Fluid Mech., 19 (2017) 151-190.

\bibitem{LLZ2025}
H. Liu, T. Luo, H. Zhong, Strong solutions to the 3-D compressible MHD equations with density-dependent viscosities in exterior domains with far-field vacuum,  arxiv:2504.04376.


\bibitem{LiuXinYang}T. Liu, Z. Xin, T. Yang, Vacuum states for compressible flow, Discrete Contin. Dyn. Syst. 4 (1998) 1-32.


\bibitem{LXZ2016}
	T. Luo, Z. Xin, H. Zeng,  Nonlinear asymptotic stability of the Lane-Emden solutions for the viscous
	gaseous star problem with degenerate density dependent viscosities, Comm. Math. Phys.,  347 (2016) 657-702.


\bibitem{Majda1986book}
A. Majda, Compressible fluid flow and systems of conservation laws in several space variables,
Applied Mathematical Science 53 Spinger-Verlag: New York, Berlin Heidelberg, 1986.

\bibitem{M2012}N. Masmoudi, F. Rousset, Uniform regularity for the Navier-Stokes equation with Navier boundary condition, Arch. Ration. Mech. Anal. 203 (2012) 529-575.

\bibitem{Matsumura1980}A. Matsumura, T. Nishida,
 The initial value problem for the equations of motion of viscous and heat-conductive gases, J. Math. Kyoto Univ., 20 (1980) 67-104.

\bibitem{M}
A. Matsumura, T. Nishida, Initial boundary value problems for the equations of motion of compressible
viscous and heat-conductive fluids, Commun. Math. Phys., 89 (1983) 445-464.



\bibitem{nash}
J. Nash, Le probleme de Cauchy pour les  \'{e}quations diff\'{e}rentielles d\'{u}n fluide g\'{e}n\'{e}ral. Bull. Soc. Math. Fr., 90 (1962) 487-491.

\bibitem{Navierslip}
C.L.M.H. Navier, Sur les lois de l'\'{e}quilibre et du mouvement des corps \'{e}lastiques, Mem. Acad. R. Sci. Inst. France, 6 (1827) 369.

\bibitem{A-L}J. Simon, Compact sets in $L^{p}(0,T; B)$, Ann. Mat. Pura Appl., 146 (1987) 65-96.



	\bibitem{VY2016}A. Vasseur, C. Yu, Existence of global weak solutions for 3D degenerate compressible Navier-Stokes equations, Invent. Math., 206 (2016) 935-974.

\bibitem{wahl}
W. Von Wahl, Estimating $\nabla u$ by $\mbox{div}u$   and $\mbox{curl}u$,  Mathematical Methods in the Applied Sciences, 15 (1992) 123-143.


\bibitem{WXY2015}Y. Wang, Z. Xin, Y. Yong, Uniform regularity and vanishing viscosity limit for the compressible Navier-Stokes with general Navier-slip boundary conditions in 3-dimensional domains, SIAM J. Math. Anal. 47 (2015).
\bibitem{XX2007}
Y. Xiao, Z. Xin, On the vanishing viscosity limit for the 3D Navier-Stokes equations with a slip boundary
condition, Commun. Pure Appl. Math., 60 (2007) 1027-1055.

\bibitem{XX2013}
Y. Xiao, Z. Xin,
On 3D Lagrangian Navier-Stokes $\alpha$ model with a class of vorticity-slip
boundary conditions,J. Math. Fluid Mech., 15 (2013) 215-247.

\bibitem{XZ2021jmpa}
Z. Xin, S. Zhu, Well-posedness of the three-dimensional isentropic compressible Navier-Stokes equations with degenerate viscosities and far field vacuum,
J. Math. Pures Appl.,  152 (2021) 94-144.

\bibitem{XZ2021advance}
Z. Xin, S. Zhu, Global well-posedness of regular solutions
to the three-dimensional isentropic compressible Navier-Stokes equations with degenerate viscosities and vacuum, Advances in Mathematics, 393 (2021) 108072.

\bibitem{YYZ2001}
T. Yang, Z. Yao,  C. Zhu, Compressible Navier-Stokes equations with density-dependent viscosity and vacuum, Communications in Partial Differential Equations, 26 (2001) 965-981.

	\bibitem{YZ2002-1}
	T. Yang, C. Zhu, Compressible Navier-Stokes equations with degenerate viscosity coefficient and vacuum, Comm. Math. Phys. 230 (2002) 329-363.
	\bibitem{YZ2002}
	T. Yang, H. Zhao, A vacuum problem for the one-dimensional compressible Navier-Stokes equations with density-dependent viscosity, J. Differential Equations, 184 (2002) 163-184.

\bibitem{ZF2006}
T. Zhang, D. Fang, Global behavior of compressible Navier-Stokes equations with a degenerate viscosity
coefficient, Arch. Ration. Mech. Anal. 182(2) (2006) 223.
\end{thebibliography}
\end{document}